\documentclass{amsart}
\usepackage{amsfonts,amsthm,amsmath}
\usepackage{amssymb}
\usepackage[usenames,dvipsnames]{color}
\usepackage{xypic}
\usepackage[a4paper]{geometry}

\renewcommand{\t}{\mathfrak{t}}

\newcommand{\N}{\mathcal{N}}

\DeclareMathOperator*{\End}{End}
\DeclareMathOperator*{\Hom}{Hom}
\theoremstyle{plain}
	\newtheorem{thm}{Theorem}
	\newtheorem*{thm*}{Theorem}
	
	\newtheorem{lemma}{Lemma}
	\newtheorem*{lemma*}{Lemma}
	\newtheorem{prop}[thm]{Proposition}
	\newtheorem*{prop*}{Proposition}
	\newtheorem{corollary}{Corollary}
	\newtheorem*{corollary*}{Corollary}
\theoremstyle{remark}
	\newtheorem{remark}{Remark}
	\newtheorem*{remark*}{Remark}

\theoremstyle{definition}
	\newtheorem*{definition}{Definition}	
\title{Equivariant Jeffrey-Kirwan localization theorem in non-compact setting}
\author{Zsolt Szil\'agyi}
\address{Section de Math\'ematiques, Universit\'e de Gen\`eve, Gen\`eve, Suisse.}
\email{zsolt.szilagyi@unige.ch}
\begin{document}
\maketitle
\begin{abstract}
We generalize the Jeffrey-Kirwan localization theorem (\cite{JK1}, \cite{JK2}) for  non-compact symplectic and hyperK\"ahler quotients. Similarly to the circle compact integration of \cite{HP2} we define equivariant integrals on non-compact manifolds using the Atiyah-Bott-Berline-Vergne localization formula as formal definition. We introduce a so called equivariant Jeffrey-Kirwan residue and we show that it shares similar properties as the usual one. 
Our localization formula has the same structure as the usual Jeffrey-Kirwan formula, but it uses formal integration and equivariant residue. We also give a version for hyperK\"ahler quotients. Finally, we apply our formula to compute the equivariant cohomology ring of Hilbert scheme of points on the plane constructed as a hyperK\"ahler quotient \cite{Na}. 
\end{abstract}

\section{Introduction}

Let $(M,\omega)$ be a non-compact symplectic manifold with Hamiltonian $G\times S$-action and let $\mu_{G\times S}:M\to \mathfrak{g}^{*}\times \mathfrak{s}^{*}$ be its moment map. We assume that there is an one-dimensional torus $K$ in the center of $G\times S$ such that it has proper and bounded below moment map. We define integrals of equivariant cohomology classes of $M$ by the Atiyah-Bott-Berline-Vergne formula (\cite{AB}, \cite{BeVe}), assuming that $M^{T\times S}$ is compact
$$
\oint_{M} \alpha := \sum_{F\times M^{T\times S}} \int_{F} \frac{i^{*}_{F}\alpha}{ e_{T\times S}\mathcal{N}(F \,|\, M) } 
$$
where $\alpha \in H_{G\times S}(M)$, $T\subset G$ maximal torus and $\mathcal{N}(F \,|\, M)$ normal bundle of $F$ in $M$. Let $0\in (\mathfrak{g}^{*})^{G}$ be a regular value of $\mu_{G}$ and consider the symplectic quotient $M/\!\!/G = \mu^{-1}(0)/G$  (we suppose that it is non-compact). In Theorem \ref{Thm-JK-3} we give the following formula for integrals on the quotient
$$
\oint_{M/\!\!/G} \kappa_{S}(\alpha e^{\omega - \mu_{G\times S}})
=
\lim_{s\to 0} \textnormal{EqRes}_{x} \left(
\frac{\varpi}{|W| vol(T)} \oint_{M} \alpha e^{\omega - \mu_{T\times S} +s\rho}
\right)(x), 
$$
where $\kappa_{S}:H_{G\times S}(M) \to H_{S}(M/\!\!/G)$ is the equivariant Kirwan map, $\varpi$ is the product of roots of $G$, $|W|$ is the order of its Weyl group, $\rho$ is a small regular value of the torus moment map $\mu_{T}$ (we can forget about it if $0$ is a regular value of $\mu_{T}$), and the equivariant Jeffrey-Kirwan residue $\textnormal{EqRes}_{x}$ is defined in  section \ref{secEqRes}. For hyperK\"ahler quotient we have a similar formula (Theorem \ref{Thm-JK-4})
$$
\oint_{M/\!\!/\!\!/\!\!/_{(\xi,0)}G} \kappa_{S}(\alpha e^{\omega_{\mathbb{R}} - \mu_{\mathbb{R}} - \mu_{S} +\xi})
=
\lim_{s\to 0} \textnormal{EqRes}_{x} \left(
\frac{ \varpi_{\mathbb{R}} \vartheta\varpi_{\mathbb{C}}}{|W| vol(T)} \oint_{M} \alpha e^{\omega_{\mathbb{R}} - \mu_{\mathbb{R}}^{T} - \mu_{S} +\xi+ s\rho}
\right)(x), 
$$
where $\kappa_{S}:H_{G\times S}(M) \to H_{S}(M/\!\!/\!\!/\!\!/_{(\xi,0)}G)$ is the Kirwan map, $\omega_{\mathbb{R}}$ and $\mu_{\mathbb{R}}$ are the real symplectic form and the moment map on $M$, $\vartheta\varpi_{\mathbb{C}}$ is the product of $T\times S$-weights on $\mathfrak{g}^{*}_{\mathbb{C}}$ (it is assumed that $(\mathfrak{t}^{*}_{\mathbb{C}})^{S} = \{0\}$), and $\varpi_{\mathbb{R}}$ is the product of roots.

The content of the paper is as follows. In the first part of section 2 we define and review properties of the Jeffrey-Kirwan residue. In the second part of the section we introduce the equivariant Jeffrey-Kirwan residue and we show that it has the same properties as the usual one. In section 3 we review in detail the symplectic cut technique which will be used in the subsequent section. In section 4 we first prove a generalization of the Jeffrey-Kirwan localization formula (Theorem \ref{Thm-JK-1}) using similar techniques as in \cite{JKo}. We use it to prove the non-compact version (Theorem \ref{Thm-JK-2} and \ref{Thm-JK-3}). Moreover, we prove first the abelian version of the theorem and then by the same method of \cite{Martin} we deduce the  non-abelian version. We conclude the section with the discussion of the hyperK\"ahler case (Theorem \ref{Thm-JK-4}). In the last section we apply our formulas for computation of the equivariant cohomology ring of Hilbert scheme of points on the plane (Theorem \ref{Thm-H-2}), constructed as hyerpK\"ahler quotient \cite{Na}. By equivariant formality the ordinary cohomology may be derived (cf. \cite{LS}, \cite{Va}).


\vskip1ex
\noindent {\bf Acknowledgement.}
The author is grateful to Andr\'as Szenes and Mich\`ele Vergne for discussions.

\section{Iterated residues}

In this section we recall the definition of the Jeffrey-Kirwan residue (\cite{JK1}, \cite{JK2}) and some of its properties (cf. \cite{BV}, \cite{JKo}). We prove that under some analyticity condition it is independent on choice of bases and polarization.
 In the second part of the section we introduce an equivariant version of the Jeffrey-Kirwan residue. It is defined in terms of ordinary Jeffrey-Kirwan residue and we prove that it admits similar properties.

\subsection{Jeffrey-Kirwan residue}

Let $\mathfrak{t}$ be an $r$-dimensional real vector space.

	\begin{definition}
Let $x=\{ x_{1},\ldots,x_{r} \}$ be an ordered bases of $\mathfrak{t}^{*}$. For any non-zero $\alpha =\sum_{i=1}^{r} a_{i}x_{i} \in\mathfrak{t}^{*}$ we define its polarization as
$$
\widetilde{\alpha}:=
\begin{cases}
\alpha & \textnormal{if }a_{1}=\ldots=a_{k-1}=0,\,a_{k}>0,
\\
-\alpha & \textnormal{if }a_{1}=\ldots=a_{k-1}=0,\,a_{k}<0. 
\end{cases}
$$
 We say that $\alpha$ is polarized with respect to $x$ if $\alpha = \widetilde{\alpha}$. We define $\varepsilon(\alpha)\in\{\pm 1\}$ by $\alpha = \varepsilon(\alpha) \cdot \widetilde{\alpha}$. The set of polarized vectors in $\mathfrak{t}^{*}$ form a cone, which we call the polarized cone of $\mathfrak{t}^{*}$.
	\end{definition}

The relationship between our notion of polarization and the one used by Jeffrey and Kirwan is as follows.

	\begin{remark}\label{Rk-I-4}
Let $\mathcal{A} = [\alpha_{i}\,|\,i\in I]$ be a finite collection of non-zero vectors in $\mathfrak{t}^{*}$. All possible simultaneous polarizations of elements of $\mathcal{A}$ are parametrized by connected components $\Lambda$ of $\{ t\in \mathfrak{t}\,|\,\alpha_{i}(t)\neq0,\,\forall i\in I \}$ as follows. For any $\xi \in \Lambda$ define
$$
\overline{\alpha}_{i}:=
\begin{cases}
\alpha_{i} & \textnormal{if }\alpha_{i}(\xi)>0,
\\
-\alpha_{i} & \textnormal{if }\alpha_{i}(\xi)<0,
\end{cases}
$$
which does not depend on the choice of $\xi$. Consider an ordered bases $x=\{x_{1},\ldots,x_{r}\}$ of $\mathfrak{t}^{*}$ such that $x_{1}(\xi)=1$, $x_{2}(\xi) = \ldots = x_{r}(\xi) =0$. The polarization $\widetilde{\alpha}_{i}$ of $\alpha_{i}$ with respect to $x$ agrees with $\overline{\alpha}_{i}$. Conversely, given $x$ ordered bases consider $\Lambda = \{t \in \mathfrak{t}\,|\,\widetilde{\alpha}_{i}(t)>0,\,\forall i \in I \}$, which is non-empty and $\overline{\alpha}_{i}$ agrees with $\widetilde{\alpha}_{i}$ for all $i\in I$.
	\end{remark}

Let $\lambda_{I},\alpha_{i}\in \mathfrak{t}^{*}$, $(i\in I)$, $P_{I}\in S\mathfrak{t}^{*}=\mathbb{R}[\mathfrak{t}]$ and consider the function $F = \frac{P_{I} e^{\lambda_{I}}}{ \prod_{i\in I} \alpha_{i}}$. Write $F$ in bases $x$ and denote it by $F(x)$.

	\begin{definition}
For $\beta_{1},\ldots,\beta_{k} \in \mathfrak{t}^{*}$ non-zero vectors we define
	\begin{equation}\label{Eq-I-1}
\textnormal{Res}_{x_{k}|\beta_{k}} \ldots \textnormal{Res}_{x_{1}|\beta_{1}} F(x)dx_{1}\ldots dx_{k}
	\end{equation} 
inductively as follows. If $\beta_{1} \notin \langle x_{2},\ldots,x_{r}\rangle$ then from $u_{1} = \widetilde{\beta}_{1}(x)$ we express $x_{1}$ in terms of $u_{1},x_{2},\ldots,x_{r}$, i.e.
$x_{1} = \beta_{1}'(u_{1},x_{2},\ldots,x_{r}).$
 We substitute $x_{1} = \beta'_{1}$ in $F(x)$ and we expand it as $u_{1}\ll x_{2},\ldots,x_{r}$.  Denote the result  by $\mathcal{F}(u_{1},x_{2},\ldots,x_{r})$.
For $i>1$ let $\beta'_{i} = \pi_{\langle \beta_{1} \rangle}\beta_{i}$ be the projection of $\beta_{i}$ to $\langle x_{2},\ldots,x_{r} \rangle$ along $\beta_{1}$, i.e. $\beta'_{i}(x_{2},\ldots,x_{r}) = \beta_{i}(\beta'_{1}(0,x_{2},\ldots,x_{r}),x_{2},\ldots,x_{r})$. Then (\ref{Eq-I-1}) is defined as
$$
\bigg( \frac{\partial\widetilde{\beta}_{1}(x)}{\partial x_{1}} \bigg)^{-1} 
\textnormal{Res}_{x_{k}|\beta'_{k}} \ldots \textnormal{Res}_{x_{2}|\beta'_{2}} \textnormal{Res}_{u_{1}=0} \mathcal{F}(u_{1},x_{2},\ldots, x_{r}) du_{1} dx_{2} \ldots dx_{r}.
$$
If $\beta_{1}\in \langle x_{2},\ldots, x_{r}\rangle$ then we set (\ref{Eq-I-1}) to zero.
	\end{definition}

If
	\begin{equation}\label{Eq-I-4}
\pi_{\langle \beta_{1},\ldots,\beta_{i-1} \rangle}\beta_{i}\in \langle x_{i},\ldots,x_{r} \rangle \setminus \langle x_{i+1},\ldots,x_{r}\rangle\quad \forall i=1,\ldots, k 
	\end{equation}
then the system of equations
	\begin{equation}\label{Eq-I-3a}
u_{1} = \widetilde{\beta_{1}}(x),\ 
u_{2} = \widetilde{\pi_{\langle \beta_{1}\rangle}\beta_{2}}(x),\ 
\ldots,\ 
u_{k} = \widetilde{\pi_{\langle \beta_{1},\ldots,\beta_{k-1}\rangle}\beta_{k}}(x)
	\end{equation}
can be expressed in matrix form
	\begin{equation}\label{Eq-I-13}
(u_{1},\ldots,u_{k})^{t} =  B \cdot (x_{1},\ldots,x_{r})^{t}
	\end{equation}
where $B\in M_{k,r}(\mathbb{R})$ is    upper triangular matrix with positive diagonal entries and we have
	\begin{multline}\label{Eq-I-3}
\textnormal{Res}_{x_{k}|\beta_{k}} \ldots \textnormal{Res}_{x_{1}|\beta_{1}} \frac{P_{I}(x)e^{\lambda_{I}(x)}}{\prod_{i\in I}\alpha_{i}(x)}dx_{1}\ldots dx_{k}
=
\\
\frac{1}{\delta} \cdot \textnormal{Res}_{u_{k}=0} \ldots \textnormal{Res}_{u_{1}=0} 
\frac{ \mathcal{F}_{I} (u,x_{k+1},\ldots,x_{r}) e^{(\lambda'_{I})_{1}u_{1}+ \ldots +(\lambda'_{I})_{k}u_{k} +\pi_{\langle \beta_{1},\ldots,\beta_{k} \rangle}\lambda(x) }}{\prod_{\alpha' \in\mathcal{A}'}\alpha'(x)}du_{1}\ldots du_{k}, 
	\end{multline}
where $\delta$ is the product of diagonal entries in $B$, 
$\mathcal{A}'$ is the collection of all non-zero $\pi_{\langle \beta_{1},\ldots, \beta_{k} \rangle} \alpha_{i} \in \langle x_{k+1},\ldots,x_{r} \rangle$, ($i\in I$), 
$\frac{ \mathcal{F}_{I} (u,x_{k+1},\ldots,x_{k})}{\prod_{\alpha'\in \mathcal{A}'}\alpha'(x)}$ is the Laurent series in $u_{1},\ldots,u_{k}$ with coefficients in $\mathbb{R}[[x_{k+1},\ldots,x_{r}]]$ got from $\frac{P_{I}(x)}{\prod_{i\in I}\alpha_{i}(x)}$ after base change $(\ref{Eq-I-13})$ and expansion as $u_{1} \ll \ldots \ll u_{k} \ll x_{k+1},\ldots,x_{r}$.
If $\beta_{1},\ldots,\beta_{k}$ does not satisfy (\ref{Eq-I-4}) then (\ref{Eq-I-1}) is zero.

	\begin{remark*}
For $\beta_{1},\ldots,\beta_{k}$ and $\gamma_{1},\ldots,\gamma_{k}$ satisfying (\ref{Eq-I-4}) we have
$$
\textnormal{Res}_{x_{k}|\beta_{k}} \ldots \textnormal{Res}_{x_{1}|\beta_{1}} = \textnormal{Res}_{x_{k}|\gamma_{k}} \ldots \textnormal{Res}_{x_{1}|\gamma_{1}}
$$ 
if and only if $\pi_{\langle \beta_{1},\ldots,\beta_{i-1} \rangle} \beta_{i} = c_{i}\cdot \pi_{\langle \gamma_{1},\ldots,\gamma_{i-1} \rangle} \gamma_{i}$ for some constants $c_{i}\in \mathbb{R}^{*}$, $i=1,\ldots,k$. In this case we say that tuples $\beta_{1},\ldots,\beta_{k}$ and $\gamma_{1},\ldots,\gamma_{k}$ are equivalent.
	\end{remark*}

	\begin{remark}\label{Rk-I-6}
$\textnormal{Res}_{x_{k}|\beta_{k}} \ldots \textnormal{Res}_{x_{1}|\beta_{1}} \frac{P_{I}(x) e^{\lambda_{I}(x)}}{\prod_{i\in I} \alpha_{i}(x)}dx_{1}\ldots dx_{k}$ may not be zero only if there are $i_{1},\ldots,i_{k}\in I$ such that $\alpha_{i_{1}},\ldots,\alpha_{i_{k}}$ and $\beta_{1},\ldots,\beta_{k}$ are equivalent. If they are equivalent we may suppose that $\beta_{l} = \alpha_{i_{l}}$ for all $l=1,\ldots,k$.
	\end{remark}

	\begin{definition}
We define $\textnormal{Res}^{+}_{x_{k}|\beta_{k}} \ldots \textnormal{Res}^{+}_{x_{1}|\beta_{1}} \frac{P_{I}(x) e^{\lambda_{I}(x)}}{\prod_{i\in I} \alpha_{i}(x)}dx_{1}\ldots dx_{k}$ to be equal to (\ref{Eq-I-3}) if (\ref{Eq-I-4}) holds and $(\lambda'_{I})_{1},\ldots,(\lambda'_{I})_{k} \geq 0$, otherwise it is defined to be zero. Moreover, define
$$
\textnormal{Res}^{+}_{x_{k}} \ldots \textnormal{Res}^{+}_{x_{1}} \frac{P_{I}(x) e^{\lambda_{I}(x)}}{\prod_{i\in I} \alpha_{i}(x)}dx_{1}\ldots dx_{k}
=
\sum
\textnormal{Res}^{+}_{x_{k}|\alpha_{i_{k}}} \ldots \textnormal{Res}^{+}_{x_{1}|\alpha_{i_{1}}} \frac{P_{I}(x) e^{\lambda_{I}(x)}}{\prod_{i\in I} \alpha_{i}(x)}dx_{1}\ldots dx_{k}, 
$$
where the sum is all over non-equivalent tuples $\alpha_{i_{1}},\ldots,\alpha_{i_{k}}$, $i_{1},\ldots,i_{k} \in I$. We use short notation $\textnormal{Res}^{+}_{x}$ for $\textnormal{Res}^{+}_{x_{r}} \ldots \textnormal{Res}^{+}_{x_{1}}$.
	\end{definition}

	\begin{definition}
Fix a scalar product and an orientation of $\mathfrak{t}^{*}$. Let $x=\{x_{1},\ldots,x_{r}\}$ be an ordered bases of $\mathfrak{t}^{*}$. Define
$$
\textnormal{JKRes}_{x} F(x) dx
=
\frac{1}{\sqrt{det(x_{i},x_{j})_{i,j}}}\textnormal{Res}^{+}_{x} F(x) dx,
$$
where $\det(x_{i},x_{j})_{i,j}$ is the determinant of the Gramm matrix $[(x_{i},x_{j})]_{i,j=1}^{r}$. 
	\end{definition}

	\begin{remark}\label{Rk-I-5}
If $\tau=\{ \tau_{1},\ldots,\tau_{r} \}$ is an orthonormal bases with the same  orientation as $x$, then
$$
\textnormal{JKRes}_{x} F(x) dx
=
\det\left(\frac{\partial x_{i}(\tau)}{\partial \tau_{j}}\right)^{-1}  \textnormal{Res}^{+}_{x} F(x) dx.
$$
	\end{remark}

	\begin{definition}
$\rho \in \mathfrak{t}^{*}$ is generic with respect to $F = \sum_{I} \frac{P_{I} e^{\lambda_{I}}}{\prod_{i\in I} \alpha_{i}}$ if it is not on any $(r-1)$- or less dimensional affine subspace $\lambda_{I} + \langle \alpha_{j}\in J\subset I \rangle$. It is equivalent to $0$ is generic with respect to $Fe^{-\rho}$.
	\end{definition}

	\begin{definition}
A bases $x = \{x_{1},\ldots,x_{r} \}$ of $\mathfrak{t}^{*}$ is generic with respect to $F = \sum_{I}\frac{P_{I} e^{\lambda_{I}}}{\prod_{i\in I} \alpha_{i}}$ if for any $\lambda_{I} \notin \langle \alpha_{i_{1}},\ldots,\alpha_{i_{k}} \rangle$, ($i_{1},\ldots,i_{k} \in I$) we have $\lambda_{I} \notin \langle \alpha_{i_{1}},\ldots,\alpha_{i_{k}}, x_{j_{k+1}},\ldots,x_{j_{r-1}} \rangle $ for any $j_{k+1},\ldots,j_{r-1}$.
	\end{definition}

	\begin{remark*}
Let $x$ is generic with respect to $F$. If $\lambda_{I} + \langle \alpha_{i_{1}},\ldots,\alpha_{i_{k}} \rangle \cap \langle x_{j_{1}},\ldots x_{j_{l}} \rangle\neq \emptyset$ then they intersect transversally.
	\end{remark*}

	\begin{remark*}
If $\alpha_{i_{1}},\ldots, \alpha_{i_{k}}$ satisfy (\ref{Eq-I-4}) and $x$ is generic then
	\begin{equation}
\lambda_{I} + \langle \alpha_{i_{1}},\ldots,\alpha_{i_{k}} \rangle \cap \langle x_{k+2},\ldots,x_{r} \rangle \neq \emptyset
\ \Leftrightarrow\ 
0 \in \lambda_{I} + \langle \alpha_{i_{1}},\ldots,\alpha_{i_{k}} \rangle
\ \Leftrightarrow\ 
\pi_{\langle \alpha_{i_{1}},\ldots,\alpha_{i_{k}} \rangle}\lambda_{I} = 0. 
	\end{equation}
In addition, if $0\in \mathfrak{t}^{*}$ is generic, then
$$
\lambda_{I} = (\lambda'_{I})_{1}\widetilde{\alpha}_{i_{1}} + \ldots + (\lambda'_{I})_{k}\, \widetilde{\pi_{\langle \alpha_{i_{1}},\ldots,\alpha_{i_{k-1}} \rangle}\alpha_{i_{k}}} + \pi_{\langle \alpha_{i_{1}},\ldots,\alpha_{i_{k}} \rangle}\lambda_{I}
$$
with $(\lambda'_{I})_{1},\ldots,(\lambda'_{I})_{k} \neq 0$.
	\end{remark*}

	\begin{definition}
A fraction $\frac{P_{I} e^{\lambda_{I}}}{\prod_{i\in I} \alpha_{i}}$ is called generating if $\alpha_{i}$, $i\in I$ span $\mathfrak{t}^{*}$.
	\end{definition}

Any fraction $\frac{P_{I} e^{\lambda_{I}}}{\prod_{i\in I} \alpha_{i}}$ can be decomposed to sum of generating fractions of form $\frac{e^{\lambda_{I}}}{\alpha_{i_{1}}^{n_{1}+1}\cdots \alpha_{i_{r}}^{n_{r}+1}}$, $i_{1},\ldots,i_{r} \in I$, $n_{1},\ldots,n_{r}\geq0$ and non-generating fractions. The decomposition is not unique. Moreover, $\textnormal{Res}_{x}$ and $\textnormal{Res}^{+}_{x}$ vanish on non-generating fractions.

	\begin{prop}[\cite{JK2} Proposition 3.2]\label{Prop-I-1} 
	Let $x = \{x_{1},\ldots,x_{r}\}$ be an ordered bases. Let $\alpha_{1},\ldots,\alpha_{r}$ be linearly independent. Suppose that $\lambda$ not in any $(r-1)$- or less dimensional subspace spanned by some of $\alpha_{i}$'s. Write $\lambda = \lambda_{1}\widetilde{\alpha}_{1}+\ldots + \lambda_{r}\widetilde{\alpha}_{r}$ and $\alpha_{i} = a_{i1}x_{1}+\ldots a_{ir}x_{r}$, $i=1,\ldots,r$. Then
	\begin{equation}\label{Eq-I-7}
\textnormal{Res}^{+}_{x_{r}} \ldots \textnormal{Res}^{+}_{x_{1}} \frac{e^{\lambda(x)}}{ \prod_{i=1}^{r}\alpha_{i}(x)^{n_{i}+1}}dx 
= 
	\begin{cases}
\displaystyle\frac{1}{|\det([a_{ij}])|} \prod_{i=1}^{r} \dfrac{\varepsilon(\alpha_{i})^{n_{i}+1} (\lambda_{i})^{n_{i}}}{n_{i}!}
 & \lambda_{1},\ldots,\lambda_{r}>0\\
0 & \textnormal{otherwise}
\end{cases}
\end{equation}
	\end{prop}

	\begin{proof}
It is enough to prove the proposition for polarized $\alpha_{i}$'s, i.e. $\varepsilon(\alpha_{i})=1$.
First we prove it for $n_{1}=\ldots=n_{r}=0$ from which we will deduce the general case. We proceed by induction on $r$. For $r=1$ the statement is obvious. 
For $\sigma\in S_{n}$, $\sigma(1)=j$ if $\alpha_{\sigma(1)},\ldots,\alpha_{\sigma(r)}$ satisfy (\ref{Eq-I-4}) then we compute 
	\begin{multline}\label{Eq-I-8}
\textnormal{Res}^{+}_{x_{r}|\alpha_{\sigma(r)}}\ldots \textnormal{Res}^{+}_{x_{2}|\alpha_{\sigma(2)}}\textnormal{Res}^{+}_{x_{1}|\alpha_{j}}\frac{e^{\lambda(x)}}{\prod_{i=1}^{r}\alpha_{i}(x)}dx =
\\
\left( \frac{\partial\alpha_{j}(x)}{\partial x_{1}}\right)^{-1}
\textnormal{Res}^{+}_{x_{r}|\pi_{\langle \alpha_{j}\rangle}\alpha_{\sigma(r)}}\ldots 
\textnormal{Res}^{+}_{x_{2}|\pi_{\langle \alpha_{j}\rangle}\alpha_{\sigma(2)}}
\textnormal{Res}^{+}_{u_{1}=0} \frac{e^{\lambda'_{j}u_{1} + \pi_{\langle \alpha_{j}\rangle}\lambda(x)}}{u_{1}\prod_{i\neq j}^{r}(c_{ji}u_{1}+\pi_{\langle \alpha_{j}\rangle}\alpha_{i}(x))}du_{1}dx_{2}\ldots dx_{r}  
\\
=
\left(\frac{\partial\alpha_{j}(x)}{\partial x_{1}} \right)^{-1}
\chi_{[0,+\infty)}(\lambda'_{j})\,
\textnormal{Res}^{+}_{x_{r}|\pi_{\langle \alpha_{j}\rangle}\alpha_{\sigma(r)}}\ldots 
\textnormal{Res}^{+}_{x_{2}|\pi_{\langle \alpha_{j}\rangle}\alpha_{\sigma(2)}}
\frac{e^{\pi_{\langle \alpha_{j}\rangle}\lambda(x)}}{\prod_{i\neq j}^{r}\pi_{\langle \alpha_{j}\rangle}\alpha_{i}(x)} dx_{2}\ldots dx_{r},  
	\end{multline}
where $u_{1} = \alpha_{j}(x)$, $\chi_{[0,+\infty)}$ is the characteristic function of $[0,+\infty)$ and $c_{ji}\in \mathbb{R}$ such that $\alpha_{i} = c_{ji}\alpha_{j} + \pi_{\langle \alpha_{j} \rangle}\alpha_{i}$.
By induction,
	\begin{multline}
\textnormal{Res}^{+}_{x_{r}}\ldots \textnormal{Res}^{+}_{x_{1}}\frac{e^{\lambda(x)}}{\prod_{i=1}^{r}\alpha_{i}(x)}dx 
= 
\sum_{j=1}^{r}\sum_{\substack{\sigma\in S_{n}\\\sigma(1)=j}}
\textnormal{Res}^{+}_{x_{r}|\alpha_{\sigma(r)}}\ldots \textnormal{Res}^{+}_{x_{2}|\alpha_{\sigma(2)}}\textnormal{Res}^{+}_{x_{1}|\alpha_{j}}
\frac{e^{\lambda(x)}}{\prod_{i=1}^{r}\alpha_{i}(x)}dx
\\
=
\sum_{j=1}^{r} 
\frac{1}{ \left[ \frac{\partial\alpha_{j}(x)}{\partial x_{1}}\right] } 
\chi_{[0,+\infty)}(\lambda'_{j}) 
\sum_{\substack{\sigma\in S_{n}\\\sigma(1)=j}}
\textnormal{Res}^{+}_{x_{r}|\alpha_{\sigma(r)}}\ldots \textnormal{Res}^{+}_{x_{2}|\alpha_{\sigma(2)}}
\frac{e^{\pi_{ \langle \alpha_{j} \rangle } \lambda(x) } }{ \prod_{i\neq j} \pi_{ \langle \alpha_{j} \rangle } \alpha_{i}(x) } dx_{2} \ldots dx_{r}
\\
=
\sum_{j=1}^{r} \frac{1}{ \left[ \frac{\partial\alpha_{j}(x)}{\partial x_{1}} \right] }
\chi_{[0,+\infty)}(\lambda'_{j}) 
\prod_{i\neq j} \varepsilon(\pi_{\langle \alpha_{j}\rangle}\alpha_{i}) 
\frac{1}{ \left| \det \left[ \frac{\partial \pi_{\langle \alpha_{j}\rangle}\alpha_{k}(x)}{\partial x_{l}} \right]_{ \substack{ k \neq j \\ l \neq 1 } } \right| }
\chi_{ Cone( \widetilde{\pi_{\langle \alpha_{j}\rangle} \alpha_{i} }\, |\, 1\leq i \leq r  ) } ( \pi_{ \langle \alpha_{j} \rangle } \lambda )
\\
= 
\frac{1}{ \left| \det \left[ \frac{\partial\alpha_{k}(x)}{\partial x_{l}} \right] \right| }
\sum_{j=1}^{r}
\prod_{i\neq j} \varepsilon(\pi_{\langle \alpha_{j}\rangle}\alpha_{i})
\,\chi_{Cone( \alpha_{j},\, \widetilde{\pi_{\langle \alpha_{j}\rangle}\alpha_{i}}\, |\, 1\leq i \leq r )}(\lambda)
	\end{multline}
We will show that
	\begin{equation}\label{Eq-I-9}
\chi_{Cone(\alpha_{1},\ldots\alpha_{r})}(\lambda) 
= 
\sum_{j=1}^{r}
\prod_{i\neq j}\varepsilon(\pi_{\langle \alpha_{j}\rangle}\alpha_{i})
\,\chi_{Cone( \alpha_{j},\, \widetilde{\pi_{\langle \alpha_{j}\rangle} \alpha_{i}}\, |\, 1\leq i \leq r )}(\lambda).
	\end{equation}
It is enough to show (\ref{Eq-I-9}) when $\alpha_{i}(x) = x_{1} + \beta_{i}(x)$, $1\leq i\leq q$, $\alpha_{i}(x) = \beta_{i}(x)$, $q< i \leq r$ with $\beta_{i}\in \langle x_{2},\ldots, x_{r} \rangle$ and $\beta_{i} - \beta_{j} = \widetilde{\beta_{i}-\beta_{j}}$ if $j<i\leq q$.
Then (\ref{Eq-I-9}) can be reformulated as
	\begin{equation}\label{Eq-I-10}
\chi_{Cone(\alpha_{1},\ldots,\alpha_{r})}(\lambda) = \sum_{j=1}^{q}(-1)^{i-1}\chi_{Cone( \alpha_{j},\, \widetilde{\beta_{j}-\beta_{i}},\,\beta_{l}\ |\ i,j\leq q,\, q<l)}(\lambda).
	\end{equation}
For any $j \leq q$ we have
	\begin{align*}
\lambda &= \lambda_{1}\alpha_{1} + \ldots +\lambda_{r}\alpha_{r}
 = (\lambda_{1}+\ldots+\lambda_{r})\alpha_{j} + \sum_{i=1}^{q}\lambda_{i}(\beta_{i}-\beta_{j}) + \sum_{l>q} \lambda_{l}\beta_{l}
\\
&=
(\lambda_{1}+\ldots+\lambda_{r})\alpha_{j} + \sum_{i<j}(-\lambda_{i})(\widetilde{\beta_{i}-\beta_{j}}) + \sum_{j<i\leq q}\lambda_{i}(\widetilde{\beta_{i}-\beta_{j}}) + \sum_{l>q} \lambda_{l}\beta_{l}.
	\end{align*}
By hypothesis $0$ is generic for $\frac{e^{\lambda}}{\prod_{i=1}^{r}\alpha_{i}}$, hence $\lambda_{1},\ldots,\lambda_{r}\neq 0$. We have
$$
\chi_{Cone(\alpha_{1},\ldots,\alpha_{r})}(\lambda) = \chi_{(\mathbb{R}_{>0})^{r}}(\lambda_{1},\ldots,\lambda_{r}),
$$
$$
\chi_{Cone( \alpha_{j},\, \widetilde{\beta_{j}-\beta_{i}},\,\beta_{l}\, |\, i \leq q,\, q<l)}(\lambda) 
= 
\chi_{\{ s\in (\mathbb{R}^{*})^{r}\, |\, \sum_{i=1}^{q}s_{i}\geq0,\, s_{i}<0,\, s_{l}>0,\, i<j<l\}}(\lambda_{1},\ldots,\lambda_{r}).
$$
Then (\ref{Eq-I-10}) is equivalent to the following easy relation on $(\mathbb{R}^{*})^{r}$
	\begin{align*}
\chi_{\{s\in (\mathbb{R}^{*})^{r}\, |\, s_{1},\ldots,s_{r}>0\}} 
={}&
\chi_{\{s\in (\mathbb{R}^{*})^{r}\, |\, \sum_{i=1}^{q}s_{i}\geq0,\ s_{1},\ldots,s_{r}>0\}}
\\
={}&
\chi_{\{s\in (\mathbb{R}^{*})^{r}\, |\,\sum_{i=1}^{q}s_{i}\geq0,\ s_{2},\ldots,s_{r}>0\}}
-
\chi_{\{s\in (\mathbb{R}^{*})^{r}\, |\, \sum_{i=1}^{q}s_{i}\geq0,\ s_{1}<0,\ s_{3},\ldots,s_{r}>0\}}
\\&
+
\chi_{\{s\in (\mathbb{R}^{*})^{r}\, |\, \sum_{i=1}^{q}s_{i}\geq0,\ s_{1},s_{2}<0,\ s_{4},\ldots,s_{r}>0\}}
\\&
\ldots + (-1)^{q-1}
\chi_{\{s\in (\mathbb{R}^{*})^{r}\, |\,\sum_{i=1}^{q}s_{i}\geq0,\ s_{1},\ldots,s_{q-1}<0,\ s_{q+1},\ldots,s_{r}>0\}}. 
	\end{align*}
Thus we proved (\ref{Eq-I-7}) for $n_{1}=\ldots=n_{r}=0$. To deduce the general case we set $y_{i} = tx_{i}$, $(t>0)$ and set $\mathcal{R}_{n_{1},\ldots,n_{r}}$ to the left hand side of (\ref{Eq-I-7}). Then for $\lambda = \sum_{j=1}^{r} \lambda^{j} x_{j}$ we have
	\begin{align*}\label{Eq-I-11}
\textnormal{Res}^{+}_{x_{r}}\ldots \textnormal{Res}^{+}_{x_{1}} \frac{e^{t\sum_{i=1}^{r}\lambda_{i}\alpha_{i}(x)}}{\prod_{i=1}^{r}\alpha_{i}(x)^{n_{i}+1}}dx 
&=
t^{n_{1}+\ldots+n_{r}}
\textnormal{Res}^{+}_{y_{r}}\ldots \textnormal{Res}^{+}_{y_{1}} \frac{ e^{ \lambda^{1}y_{1} + \ldots +\lambda^{r} y_{r} } }{ \prod_{i=1}^{r} (a_{i1}y_{1} + \ldots + a_{ir}y_{r})^{ n_{i}+1} }dy 
\\
&
=t^{n_{1}+\ldots +n_{r}}\mathcal{R}_{n_{1},\ldots,n_{r}}. 
	\end{align*}
Take the derivative of both sides with respect to $t$ at $t=1$ to get
	\begin{equation}\label{Eq-I-12}
\sum_{i=1}^{r}\lambda_{i}\mathcal{R}_{n_{1},\ldots,n_{i-1}, \, n_{i}-1 , \, n_{i+1 } , \ldots , n_{r}} = ( n_{1} + \ldots + n_{r} ) \mathcal{R}_{n_{1} , \ldots , n_{r} }.
	\end{equation}
From (\ref{Eq-I-12}) follows (\ref{Eq-I-7}) by induction on $n_{1}+\ldots+n_{r}$.
	\end{proof}

The following cone property will be used extensively through the paper.

	\begin{corollary}[\cite{JK2} Proposition 3.2]\label{Cor-I-1}
Let $x$ be an ordered bases. If $\lambda_{I} \notin Cone( \widetilde{\alpha_{i}}\ |\ i\in I)$ then 
\begin{equation}\label{Eq-I-15}
\textnormal{Res}^{+}_{x}\frac{ P_{I}(x)e^{ \lambda_{I}(x) } }{ \prod_{i\in I} \alpha_{i}(x) } dx =0
\quad and \quad
\textnormal{JKRes}_{x}\frac{ P_{I}(x)e^{ \lambda_{I}(x) } }{ \prod_{i\in I} \alpha_{i}(x) } dx =0.
\end{equation}
	\end{corollary}

	\begin{proof}
We can decompose $\frac{P_{I}(x) e^{\lambda_{I}(x)}}{\prod_{i\in I}\alpha_{i}(x)}$ to sum of generating and non-generating fractions. On non-generating fractions $\textnormal{Res}^{+}_{x}$ vanishes. Moreover, the generating fractions are of form $ \frac{e^{\lambda_{I}(x)}}{ \prod_{k=1}^{r}\alpha_{i_{k}}(x)^{n_{k}} }$. Since $Cone(\widetilde{\alpha}_{i_{1}},\ldots, \widetilde{\alpha}_{i_{r}})\subset Cone(\widetilde{\alpha}_{i}\ |\ i\in I)$ for all $i_{1},\ldots,i_{r}\in I$, hence from Proposition \ref{Prop-I-1} follows (\ref{Eq-I-15}).
	\end{proof}

	\begin{prop}[cf. \cite{JKo} Lemma 3.3]\label{Prop-I-2}
Let $F = \sum_{I} \frac{P_{I}e^{\lambda_{I}}}{\prod_{i\in I}\alpha_{i}}$ and suppose that $0\in \mathfrak{t}^{*}$ is generic with respect to $F$. If $F(x)$ is analytic then $\textnormal{JKRes}_{x}F(x)dx$ is independent of $x$.
	\end{prop}

	\begin{proof}
By the previous proposition the $\textnormal{JKRes}_{x}F(x)dx$ may only depend on polarizations. Let $\mathcal{A} = \cup_{I}\{ \alpha_{i}\ |\ i\in I\}$. Polarizations on $\mathcal{A}$ correspond to connected components of $\{ t\in \mathfrak{t}\ |\ \alpha(t)\neq 0,\, \forall \alpha \in \mathcal{A} \}$. These components are open polyhedral cones. Let $\Lambda$ and $\Lambda'$ be two neighboring cones, separated by a hyperplane $\{t\in\mathfrak{t}\ |\ \alpha(t) = 0 \}$ for some $\alpha\in\mathcal{A}$. Let $\xi\neq 0$ be in the relative interior of intersection of closures $cl(\Lambda) \cap cl(\Lambda')$.

Choose an ordered bases $\{ x_{1},\ldots,x_{r} \}$ of $\mathfrak{t}^{*}$ such that $x_{1}(\xi)=1$, $x_{2}(\xi)=\ldots =x_{r-1}(\xi)=0$ and $x_{r}=\alpha$. Since $0$ is generic for $F$, we can choose $\xi$ and $x$ such that $x$ to be generic with respect to $F$. Suppose that $\alpha$ is polarized with respect to $\Lambda$, hence the ordered bases $\{x_{1},\ldots, x_{r}\}$ induces on $\mathcal{A}$ the same polarization as $\Lambda$. Also consider the ordered bases $\{x'_{1},\ldots,x'_{r}\}$ such that $x'_{i}=x_{i}$ for $i< r$ and $x'_{r} = -x_{r}$. It induces on $\mathcal{A}$ the same polarization as $\Lambda'$.

Let $\alpha_{i_{1}},\ldots,\alpha_{i_{r}}$ satisfying (\ref{Eq-I-4}). From (\ref{Eq-I-3}) we have
\begin{equation}\label{Eq-I-14}
\textnormal{Res}^{+}_{x_{r}|\alpha_{i_{r}}} \ldots \textnormal{Res}^{+}_{x_{1}|\alpha_{i_{1}}}
\sum_{I}\frac{P_{I}(x) e^{\lambda_{I}(x)}}{\prod_{i\in I} \alpha_{i}(x) } dx
=
\delta^{-1}\cdot
\textnormal{Res}^{+}_{u_{r}=0} \ldots \textnormal{Res}^{+}_{u_{1}=0} 
\sum_{I} \mathcal{F}_{I}(u) e^{\sum_{k=1}^{r}(\lambda'_{I})_{k}u_{k}}du,
\end{equation}
where $\delta = \left| \det\left( \frac{\partial\alpha_{i_{k}}(x)}{\partial x_{l}} \right) \right| = \left| \det\left( \frac{\partial\alpha_{i_{k}}(x')}{\partial x'_{l}} \right) \right| $, and  $\mathcal{F}_{I}(u)$ is the Laurent series got from $\frac{P_{I}(x)}{\prod_{i\in I}\alpha_{i}(x)}$ by change of variable (\ref{Eq-I-13}) and expansion with respect to $u_{1}\ll u_{2} \ll \ldots \ll u_{r}$. 
Since $0$  and bases $x$ are generic with respect to $F$, we have $(\lambda'_{I})_{k}\neq0$ for all $k=1,\ldots,r$ and $I$. Moreover, $\mathcal{F}(u):=\sum_{I}\mathcal{F}_{I}(u) e^{\sum_{k=1}^{r}(\lambda'_{I})_{k}u_{k}}$ is analytic, hence $\textnormal{Res}^{+}_{u_{r}=0} \mathcal{F}(u)du_{r} = -\textnormal{Res}^{-}_{u_{r}=0} \mathcal{F}(u)du_{r}$. 
With notation $u'_{i} = u_{i}$ for $i<r$ and $u'_{r} = - u_{r}$, (\ref{Eq-I-14}) is equal to
\begin{multline*}
\textnormal{Res}^{+}_{x_{r}|\alpha_{i_{r}}} \ldots \textnormal{Res}^{+}_{x_{1}|\alpha_{i_{1}}}
F(x) dx
\\
=
\delta^{-1}\cdot
\textnormal{Res}^{+}_{u_{r-1}=0} \ldots \textnormal{Res}^{+}_{u_{1}=0}\textnormal{Res}^{+}_{u_{r}=0} 
 \mathcal{F}(u)du_{r}du_{1}\ldots du_{r-1}
\\
=
\delta^{-1}\cdot
\textnormal{Res}^{+}_{u_{r-1}=0} \ldots \textnormal{Res}^{+}_{u_{1}=0}\textnormal{Res}^{-}_{u_{r}=0} 
- \mathcal{F}(u)du_{r}du_{1}\ldots du_{r-1}
\\
=
\delta^{-1}\cdot
\textnormal{Res}^{-}_{u_{r}=0} \textnormal{Res}^{+}_{u_{r-1}=0} \ldots \textnormal{Res}^{+}_{u_{1}=0} 
\mathcal{F}(u) du_{1} \ldots du_{r-1} d(-u_{r})
\\
=
\delta^{-1}\cdot
\textnormal{Res}^{+}_{u'_{r}=0} \textnormal{Res}^{+}_{u'_{r-1}=0} \ldots \textnormal{Res}^{+}_{u'_{1}=0} 
\mathcal{F}(u') du'_{1} \ldots du'_{r}
\\
=
\textnormal{Res}^{+}_{x'_{r}|\alpha_{i_{r}}} \textnormal{Res}^{+}_{x'_{r-1}|\alpha_{i_{r-1}}} \ldots \textnormal{Res}^{+}_{x'_{1}|\alpha_{i_{1}}} 
F(x')dx'
\end{multline*}
and summing up we get
\begin{multline*}
\textnormal{Res}^{+}_{x_{r}} \ldots \textnormal{Res}^{+}_{x_{1}}
F(x) dx
=
\sum \textnormal{Res}^{+}_{x_{r}|\alpha_{i_{r}}} \ldots \textnormal{Res}^{+}_{x_{1}|\alpha_{i_{1}}}
F(x) dx
\\
=
\sum \textnormal{Res}^{+}_{x'_{r}|\alpha_{i_{r}}} \textnormal{Res}^{+}_{x'_{r-1}|\alpha_{i_{r-1}}} \ldots \textnormal{Res}^{+}_{x'_{1}|\alpha_{i_{1}}} 
F(x')dx'
=
\textnormal{Res}^{+}_{x'_{r}} \ldots \textnormal{Res}^{+}_{x'_{1}}
F(x') dx'.
\end{multline*}
The bases $x$ and $x'$ have the same Gramm matrix $\det[(x_{i},x_{j})]_{i,j=1}^{r} = \det[(x'_{i},x'_{j})]_{i,j=1}^{r}$, thus
$$
\textnormal{JKRes}_{x} F(x) dx = \textnormal{JKRes}_{x'} F(x') dx'.
$$
	\end{proof}

Look back what we have showed so far. If $0$ is a generic for $F= \sum_{I} \frac{P_{I} e^{\lambda_{I}}}{\prod_{i\in I}\alpha_{i}}$ then by Proposition \ref{Prop-I-1} the $\textnormal{JKRes}_{x} F(x)dx$ depends only on the polarization induced by $x$ on $\cup_{I} \{ \alpha_{i} \,|\, i\in I \}$, not on the particular bases. We emphasis that $x$ may be non-generic. From the same proposition it also follows that $\textnormal{JKRes}_{x}F(x)e^{\rho(x)}dx$ depends continuously on small $\rho$. If in addition $F$ is analytic then $\textnormal{JKRes}_{x}F(x)dx$ does not depend at all on the bases $x$.

In examples we are interested may arise functions $F$ for which $0$ is not generic. The same result will hold if $x$ is a generic bases and $F$ has some additional properties (i.e. it comes from equivariant integration on a compact symplectic manifold or orbifold).

	\begin{lemma}\label{Lem-I-1}
Let $0$ be generic for $F=\sum_{I}\frac{P_{I}e^{\lambda_{I}}}{\prod_{i\in I}\alpha_{i}}$. If $F(x)$ is analytic and $0$ is not contained in the convex hull $conv(\lambda_{I}\, |\, I)$ then for all ordered bases $x$ we have
\begin{equation}\label{Eq-I-16}
\textnormal{Res}^{+}_{x}F(x)dx = 0.
\end{equation}
	\end{lemma}

	\begin{proof}
By Proposition \ref{Prop-I-2}, it is enough to show (\ref{Eq-I-16}) for a particular ordered bases $x$. Since $0 \notin conv(\lambda_{I}\, |\, I )$, there is hyperplane $\mathcal{H}$ containing the set $conv(\lambda_{I}\, |\, I)$ in one of its open half spaces. Choose the ordered bases $x=\{ x_{1},\ldots,x_{r} \}$ such that $\mathcal{H} = \langle x_{2},\ldots,x_{r} \rangle$ and $conv(\lambda_{I}\, |\, I)\subset \{\sum_{k=1}^{r}a_{k}x_{k}\, |\, a_{1}<0\}$. But all polarized vectors lie in $\{\sum_{k=1}^{r}a_{k}x_{k}\, |\, a_{1}\geq0 \}$, therefore from Corollary \ref{Cor-I-1} follows (\ref{Eq-I-16}).
	\end{proof}

	\begin{lemma}\label{Lem-I-2}
Let $x=\{ x_{1},\ldots,x_{r} \}$ be an ordered bases and let $\alpha_{1},\ldots,\alpha_{k}$ satisfying (\ref{Eq-I-4}). For $i\leq k$ define $v_{i}$ to be the projection of $x_{i}$ to $\langle \alpha_{1},\ldots,\alpha_{k} \rangle$ along $\langle x_{k+1}, \ldots, x_{r} \rangle$ and for $i>k$ let $v_{i}= x_{i}$.
Write $F(x)$ in bases $v=\{v_{1},\ldots,v_{r}\}$ and denote it by $F(v)$. Then
\begin{equation}\label{Eq-I-17}
\textnormal{Res}^{+}_{x_{k}|\alpha_{k}} \ldots \textnormal{Res}^{+}_{x_{1}|\alpha_{1}} F(x)dx_{1}\ldots dx_{k} 
=
\textnormal{Res}^{+}_{v_{k}|\alpha_{k}} \ldots \textnormal{Res}^{+}_{v_{1}|\alpha_{1}} F(v)dv_{1}\ldots dv_{k}. 
\end{equation}
	\end{lemma}

	\begin{proof}
We emphasis that $\pi_{\langle \alpha_{1},\ldots, \alpha_{l-1} \rangle} \alpha_{l}(x)$ denotes the projection of $\alpha_{l}$ to $\langle x_{l}, \ldots, x_{r} \rangle$ along $\langle \alpha_{1}, \ldots, \alpha_{l-1} \rangle$ and similarly, $\pi_{\langle \alpha_{1},\ldots, \alpha_{l-1} \rangle} \alpha_{l}(v)$ denotes the projection of $\alpha_{l}$ to $\langle v_{l}, \ldots, v_{r} \rangle$ along $\langle \alpha_{1}, \ldots, \alpha_{l-1} \rangle$. We have $\pi_{\langle \alpha_{1},\ldots, \alpha_{l-1} \rangle} \alpha_{l}(x) \in \langle \alpha_{1},\ldots,\alpha_{l} \rangle \cap \langle x_{l},\ldots, x_{r} \rangle$ and $\pi_{\langle \alpha_{1},\ldots, \alpha_{l-1} \rangle} \alpha_{l}(v) \in \langle \alpha_{1},\ldots,\alpha_{l} \rangle \cap \langle v_{l},\ldots, v_{r} \rangle,$
but 
	\begin{equation}\label{Eq-I-18}
\langle x_{l},\ldots, x_{r} \rangle = \langle v_{l},\ldots, v_{r} \rangle.
	\end{equation}
Therefore
$$
\pi_{\langle \alpha_{1},\ldots, \alpha_{l-1} \rangle} \alpha_{l}(x) 
=
\pi_{\langle \alpha_{1},\ldots, \alpha_{l-1} \rangle} \alpha_{l}(v),
$$
because they are the projection of $\alpha_{l-1}$ along $\langle \alpha_{1},\ldots,\alpha_{l-1} \rangle$ to (\ref{Eq-I-18}).
Consider the systems of equations
$$
u_{1} = \widetilde{\alpha}_{1}(x),\ \ldots,\ u_{k} = \widetilde{\pi_{\langle \alpha_{1},\ldots,\alpha_{k-1} \rangle}\alpha_{k}(x)},\ u_{k+1} = x_{k+1}, \ \ldots,\ u_{r}=x_{r} 
$$
and
$$
u_{1} = \widetilde{\alpha}_{1}(v),
\ \ldots,\ 
u_{k} = \widetilde{\pi_{\langle \alpha_{1},\ldots,\alpha_{k-1} \rangle}\alpha_{k}(v)},
\ 
u_{k+1} = v_{k+1},
\
\ldots,
\
u_{r}=v_{r}. 
$$
Write then in matrix form 
$$
(u_{1},\ldots,u_{r}) = B'\cdot (x_{1},\ldots,x_{r})^{t},
\qquad
(u_{1},\ldots,u_{r}) = B''\cdot (v_{1},\ldots,v_{r})^{t}.
$$
Remark that the first $r\times r$ minors of $B'$ and $B''$ agree, and denote it by $B_{r}$.
 Write $F(x)$ and $F(v)$ in bases $u$ and expand it with respect to $u_{1}\ll u_{2}\ll \ldots \ll u_{k} \ll u_{k+1},\ldots,u_{r}$ and denote the result by $\mathcal{F}'(u)$ and $\mathcal{F}''(u)$, respectively. We have $\mathcal{F}'(u)=\mathcal{F}''(u)$, thus
\begin{multline*}
\textnormal{Res}^{+}_{x_{k}|\alpha_{k}} \ldots \textnormal{Res}^{+}_{x_{1}|\alpha_{1}} F(x)dx_{1}\ldots dx_{k} 
= 
\det(B_{r})^{-1}
\textnormal{Res}^{+}_{u_{k}=0} \ldots \textnormal{Res}^{+}_{u_{1}=0} \mathcal{F}'(u)du_{1}\ldots du_{k}
\\
= \det(B_{r})^{-1}
\textnormal{Res}^{+}_{u_{k}=0} \ldots \textnormal{Res}^{+}_{u_{1}=0} \mathcal{F}''(u)du_{1}\ldots du_{k}
= 
\textnormal{Res}^{+}_{v_{k}|\alpha_{k}} \ldots \textnormal{Res}^{+}_{v_{1}|\alpha_{1}} F(v)dv_{1}\ldots dv_{k}. 
\end{multline*}
	\end{proof}

	\begin{definition}
Let $x$ be a generic bases with respect to $F = \sum_{I}\frac{ P_{I}e^{\lambda_{I}} }{ \prod_{i\in I}\alpha_{i} }$. The order of 
$$
\textnormal{Res}^{+}_{x_{r}|\alpha_{i_{r}}} \ldots \textnormal{Res}^{+}_{x_{1}|\alpha_{i_{1}}} \frac{P_{I}(x) e^{\lambda_{I}(x)} }{ \prod_{i\in I} \alpha_{i}(x) }dx,
$$ 
$(i_{1},\ldots, i_{r} \in I)$ is equal to $\textnormal{ord}(i_{1},\ldots,i_{r};I)=k$ if 
$$
\lambda_{I}(x) = \sum_{l=1}^{r} (\lambda'_{I})_{l} \widetilde{ \pi_{\langle \alpha_{i_{1}},\ldots, \alpha_{i_{l-1}} \rangle }\alpha_{i_{l}}}
$$ 
with $(\lambda'_{I})_{1},\ldots,(\lambda'_{I})_{k}\neq 0$ and $(\lambda'_{I})_{k+1} = \ldots = (\lambda'_{I})_{r} = 0$.
	\end{definition}

	\begin{remark}\label{Rk-I-1}
If $\textnormal{Res}^{+}_{x_{r}|\alpha_{i_{r}}} \ldots \textnormal{Res}^{+}_{x_{1}|\alpha_{i_{1}}} \frac{P_{I}(x) e^{\lambda_{I}(x)} }{ \prod_{i\in I} \alpha_{i}(x) }dx$ is of order $k$, then
\begin{equation}\label{Eq-I-2a}
\textnormal{Res}^{+}_{x_{r}|\alpha_{i_{r}}} \ldots \textnormal{Res}^{+}_{x_{1}|\alpha_{i_{1}}} \frac{P_{I}(x) e^{\lambda_{I}(x) + \rho(x)} }{ \prod_{i\in I} \alpha_{i}(x) }dx
\end{equation}
depends continuously on small $\rho \in \langle \alpha_{i_{1}}, \ldots, \alpha_{i_{k}} \rangle$ (small means that it stays of order $k$).
	\end{remark}

Let $(M,\omega)$ be a compact hamiltonian $T$-manifold with moment map $\mu:M\to \mathfrak{t}^{*}$. We choose a compatible triple $(\omega,g,\mathcal{I})$ on $M$, where $g$ Riemannian metric, $\mathcal{I}$ almost complex structure.
Consider
\begin{equation}\label{Eq-I-19}
\oint_{M} \eta e^{\omega - \mu} := \sum_{X_{I}\subset M^{T}} \int_{X_{I}}\frac{i^{*}_{X_{I}}(\eta e^{\omega - \mu}) }{e_{T}\mathcal{N}(X_{I}|M)} =: \sum_{I}\frac{P_{I} e^{\lambda_{I}}}{ \prod_{i\in I}\alpha_{i}}=F
\end{equation}
where $I$ indexes the fixed point component $X_{I}\subset M^{T}$ and $\alpha_{i}$, $i \in I$ are isotropy weights of $T$-action on the normal bundle $\mathcal{N}(X_{I}|M)$ of $X_{I}$ in $M$, and $\lambda_{I} = - \mu(X_{I})$.

\begin{definition}
Let $S\subset T$ be a subtorus and let $N\subset M^{S}$ be a connected component. The convex subpolytope $\mu(N)$ is called wall of the moment polytope $\mu(M)$.
\end{definition}

Let $\alpha_{i_{1}},\ldots,\alpha_{i_{r}}$ satisfy (\ref{Eq-I-4}). To any $\textnormal{Res}^{+}_{x_{r}|\alpha_{i_{r}}} \ldots \textnormal{Res}^{+}_{x_{1}|\alpha_{i_{1}}}\frac{ P_{I}(x)e^{ \lambda_{I}(x) } }{ \prod_{i\in I} \alpha_{i}(x) } dx$, $(i_{1},\ldots,i_{r}\in I)$ we associate a series of walls of $\mu(M)$ as follows. For any $l<r$ let $S_{l}\subset T$ subtorus such that $Lie(S_{l}) = \mathfrak{s}_{l}:=\cap_{j=1}^{l}\ker \alpha_{i_{j}}$. We can identify $Lie(T/S_{l})^{*} = \ker(\mathfrak{t}^{*}\to \mathfrak{s}^{*}_{l}) = \langle \alpha_{i_{1}},\ldots,\alpha_{i_{l}} \rangle$. Let $N_{l}$ be a the connected component of $M^{S_{l}}$ containing the fixed point component $X_{I}\subset M^{T}$. Then $N_{l}^{T} = N_{l}^{T/S_{l}}$ and $\mu(N_{l})\subset \lambda_{I} + \langle \alpha_{i_{1}},\ldots,\alpha_{i_{l}} \rangle$ having dimension $\dim \mu(N_{l}) = l$. We call $\lambda_{I} + \langle \alpha_{i_{1}},\ldots,\alpha_{i_{l}} \rangle$ the plane of $\mu(N_{l})$. Moreover, if $\textnormal{ord}(i_{1},\ldots,i_{r};I) \leq l$ then $0$ is in the plane of $\mu(N_{l})$.

	\begin{lemma}\label{Lem-I-3}
Let $x$ be a generic bases with respect to $F$ as in (\ref{Eq-I-19}). Denote $\mathcal{W}_{k}(M)$ the set of $k$-dimensional walls $\mu(N)$ of $\mu(M)$ containing $0$ in their plane. Then
$$
\sum_{\textnormal{ord}(i_{1},\ldots,i_{r};I)\leq k} 
\textnormal{Res}^{+}_{x_{r}|\alpha_{i_{r}}} \ldots \textnormal{Res}^{+}_{x_{1}|\alpha_{i_{1}}} \frac{P_{I}(x) e^{\lambda_{I}(x)} }{ \prod_{i\in I} \alpha_{i}(x) }dx
=
\sum_{ \mu(N) \in \mathcal{W}_{k}(M) }
\textnormal{Res}^{+}_{w} \textnormal{Res}^{+}_{v(N)} F_{N}(v(N),w) dv(N)dw,
$$
where $v_{1}(N),\ldots v_{k}(N)$ are the projections of $x_{1},\ldots,x_{k}$ to the plane $\mu(N)$ along $\langle x_{k+1}, \ldots, x_{r} \rangle$, $w_{i} = x_{i}$ for $i>k$, and $F_{N}(v(N),w)$ is equal to $\oint_{N}\frac{i^{*}_{N}(\eta e^{\omega - \mu})}{e_{T}(N|M)} = \sum_{X_{I}\subset N^{T}} \int_{X_{I}} \frac{ i^{*}_{X_{I}} ( \eta e^{ \omega - \mu } ) }{ e_{T}(X_{I}|M) } $ written in bases $\{v(N),w\}$ and expanded with respect to $v_{i}(N) \ll w_{j}$ for all $i \leq k < j$.
	\end{lemma}

	\begin{proof}
Let 
	\begin{equation}\label{Eq-I-20}
\textnormal{Res}^{+}_{x_{r}|\alpha_{i_{r}}} \ldots \textnormal{Res}^{+}_{x_{1}|\alpha_{i_{1}}} \frac{P_{I}(x) e^{\lambda_{I}(x)} }{ \prod_{i\in I} \alpha_{i}(x) }dx
	\end{equation}
 be of order $l\leq k$, and let $N_{k}$ be its associated $k$-dimensional wall in $\mu(M)$ and $0$ is contained in the plane of $\mu(N_{k})$. Then 
$\frac{P_{I} e^{\lambda_{I}} }{ \prod_{i\in I} \alpha_{i} } = \int_{X_{I}} \frac{ i^{*}_{X_{I}}(\eta e^{\omega - \mu}) }{ e_{T}\mathcal{N}(X_{I}|M) }$, $X_{I}\subset N_{k}^{T}$ is the summand of $\oint_{N_{k}}\frac{i^{*}_{N_{k}}(\eta e^{\omega - \mu}) }{e_{T}\mathcal{N}(N_{k}|M) }$. By Lemma \ref{Lem-I-2}, (\ref{Eq-I-20}) is equal to
$$
\textnormal{Res}^{+}_{w_{r}|\alpha_{i_{r}}} \ldots \textnormal{Res}^{+}_{w_{k+1}|\alpha_{i_{k+1}}} \textnormal{Res}^{+}_{v_{k}(N_{k})|\alpha_{i_{k}}} \ldots  \textnormal{Res}^{+}_{v_{k}(N_{k})|\alpha_{i_{1}}} \frac{P_{I}(v(N_{k}),w) e^{\lambda_{I}(v(N_{k}),w)} }{ \prod_{i\in I} \alpha_{i}(v(N_{k}),w) }dv(N_{k})dw, 
$$
which is a summand of $\textnormal{Res}^{+}_{w}\textnormal{Res}^{+}_{v(N_{k})}F_{N_{k}}(v(N_{k}),w)dv(N_{k})dw$.

Conversely, if $\mu(N)$ is a $k$-dimensional wall of $\mu(M)$ with plane containing $0$, then any summand of $\textnormal{Res}^{+}_{w}\textnormal{Res}^{+}_{v(N)} F_{N}(v(N),w)dv(N)dw$ is of form (\ref{Eq-I-20}) having order $l\leq k$.
	\end{proof}

\begin{remark}\label{Rk-I-3}
Let $S\subset T$ be a subtorus and let $N \subset M^{S}$ be a fixed point component. Recall that we can identify $Lie( T/S )^{*} = \ker( \mathfrak{t}^{*} \to \mathfrak{s}^{*} )$. $\mathcal{N}( N|M )$ $S$-equivariantly splits to sum of line bundles $ \oplus_{j} L_{j} $. Then $ e_{T} ( L_{j} ) = \gamma_{j} + e_{T/S} ( L_{j} ) $, where $0 \neq \gamma_{j} \in \mathfrak{s}^{*} $. We may identify $\mathfrak{s}^{*}$ to $ \langle w_{k+1}, \ldots, w_{r} \rangle \subset \mathfrak{t}^{*} $. Hence, if we expand $ e_{T}( L_{j} )( v(N), w ) = \gamma_{j}( w ) + e_{T/S} ( L_{j} )( v(N) )$ as $ v_{i}(N) \ll w_{j}$, $i \leq k < j$, then $\frac{ 1 }{ e_{T} \mathcal{N}( N|M ) } $ is analytic in $v(N)$. Therefore $F_{N}( v(N), W ) $ defined in Lemma \ref{Lem-I-3} is analytic in $v(N)$ by compactness of $N$. 
\end{remark}

We have the following vanishing result.

	\begin{prop}\label{Prop-I-3}
Let $x$ be a generic bases with respect to $F$ as in (\ref{Eq-I-19}). Suppose that $0$ is not on any wall of $\mu(M)$, i.e $0$ is regular value of $\mu$.  Then for any $k<r$
	\begin{equation}\label{Eq-I-21}
\sum_{\deg(i_{1},\ldots,i_{r};I) = k}
\textnormal{Res}^{+}_{x_{r}|\alpha_{i_{r}}} \ldots \textnormal{Res}^{+}_{x_{1}|\alpha_{i_{1}}} \frac{P_{I}(x) e^{\lambda_{I}(x) + \rho(x)} }{ \prod_{i\in I} \alpha_{i}(x) }dx =0
	\end{equation}
for any $\rho$ in a small neighborhood of $0$.
	\end{prop}
	\begin{proof}
First, we will show by induction on $k$ that for any $\mu(N) \in \mathcal{W}_{k}(M)$, we have
	\begin{equation}\label{Eq-I-22}
\textnormal{Res}^{+}_{v(N)} F_{N}(v(N),w) e^{\rho(v(N),w)}dv(N) = 0
	\end{equation}
with $F_{N}(v(N),w)$ defined in Lemma \ref{Lem-I-3} for any $\rho$ in a small neighborhood of $0$.

Let $k_{0}$ be the smallest number such that $\mathcal{W}_{ k_{0} }( M ) \neq \emptyset$. By Lemma \ref{Lem-I-3} this is equal to the smallest order. Since $x$ is generic, we have $k_{0}>0$. Consider $w$ as a fixed parameter and remark that $0$ is generic for $F_{N}(v(N),w)$ as fraction in $v(N)$. Furthermore, $0$ is generic for $F_{N}(v(N),w) e^{\rho'(v(N))}$ for small $\rho'\in \langle v_{1}(N),\ldots,v_{k_{0}}(N) \rangle$. Since $0$ is not contained in the convex polytope $-\mu(N) + \rho'(v(N))$ and $F_{N}( v(N), w )$ is analytic in $v(N)$, Lemma \ref{Lem-I-1} we have $\textnormal{Res}^{+}_{v(N)}F_{N}(v(N),w) e^{\rho'(v(N))} dv(N) = 0$. Moreover, we can write any small $\rho$ as $\rho(v(N),w) = \rho'(v(N)) + \rho''(w)$, thus
$$
\textnormal{Res}^{+}_{v(N)}F_{N}(v(N),w) e^{\rho(v(N),w)} dv(N) 
= 
 \textnormal{Res}^{+}_{v(N)}F_{N}(v(N),w) e^{\rho'(v(N),w)} dv(N) e^{\rho''(w)} 
= 
0.
$$
Thus we have showed (\ref{Eq-I-22}) for $k=k_{0}$. For general $k<r$
$$
\textnormal{Res}^{+}_{v(N)} F_{N}(v(N),w)dv(N)
$$
can be written as sum of degree $k$ terms and lower degree terms. The sum of lower degree terms is equal to
$\sum_{\substack{\mu(N')\in \mathcal{W}_{l}(N),\, l<k}} \textnormal{Res}^{+}_{w'}\textnormal{Res}^{+}_{v(N')} F_{N'}(v(N'),w',w) dv(N') dw'$,
where $w'=\{v_{l+1}(N),\ldots,v_{k}(N)\}$.
By induction hypothesis,
	\begin{equation}\label{Eq-I-23}
\sum_{\substack{\mu(N')\in \mathcal{W}_{l}(N),\, l<k}} \textnormal{Res}^{+}_{w'}\textnormal{Res}^{+}_{v(N')} F_{N'}(v(N'),w',w) e^{\theta(v(N'),w')} dv(N') dw' 
= 
0
	\end{equation}
for all $\theta \in \langle v_{1}(N),\ldots,v_{k}(N) \rangle$ small. From (\ref{Eq-I-23}) and Remark \ref{Rk-I-1} follows that 
$$
\textnormal{Res}^{+}_{v(N)}F_{N}(v(N),w) e^{\theta(v(N))}dv(N)
$$ depends continuously on small $\theta$. Fix $\theta$ small and let $\varrho_{N}(v(N))$ such that $0$ is generic for $F_{N}(v(N),w) e^{\theta(v(N)) + \varrho_{N}(v(N))}$ as fraction in $v(N)$. Then
	\begin{equation}\label{Eq-I-24}
\textnormal{Res}^{+}_{v(N)} F_{N}(v(N),w)e^{\theta(v(N))}dv(N) 
=
\lim_{s\to 0} \textnormal{Res}^{+}_{v(N)}F_{N}(v(N),w) e^{\theta(v(N)) + s \varrho_{N}(v(N))} dv(N)
=
0,  
	\end{equation}  
by Lemma \ref{Lem-I-1}. We can write any small $\rho$ as $\rho(v(N),w) = \theta(v(N)) + \vartheta(w)$ and we have
$$
\textnormal{Res}^{+}_{v(N)} F_{N}(v(N),w) e^{\rho(v(N),w)} dv(N) 
=
\textnormal{Res}^{+}_{v(N)} F_{N}(v(N),w)e^{\theta(v(N))} dv(N) e^{\vartheta(w)}
=
0, 
$$
by (\ref{Eq-I-24}). In particular, if $\mu(N)\in \mathcal{W}_{k}(M)$, $k<r$ then for all $\rho$ small we have
	\begin{equation}\label{Eq-I-25}
\textnormal{Res}^{+}_{w} \textnormal{Res}^{+}_{v(N)} F_{N}(v(N),w) e^{\rho(v(N),w)} dv(N) dw 
=
0. 
	\end{equation}
From Lemma \ref{Lem-I-3} follows that for small $\rho$ we have
	\begin{multline*}
\sum_{\deg(i_{1},\ldots,i_{r};I)=k} 
\textnormal{Res}^{+}_{x_{r}|\alpha_{i_{r}}} \ldots \textnormal{Res}^{+}_{x_{1}|\alpha_{i_{1}}} \frac{P_{I}(x) e^{\lambda_{I}(x) + \rho(x) }}{ \prod_{i\in I} \alpha_{i}(x) } dx 
\\
=
\sum_{\mu(N)\in \mathcal{W}_{k}(M)}
\textnormal{Res}^{+}_{w} \textnormal{Res}^{+}_{v(N)} F_{N}(v(N),w) e^{\rho(v(N),w)} dv(N) dw
\\
-
\sum_{\mu(N)\in \mathcal{W}_{k-1}(M)}
\textnormal{Res}^{+}_{w} \textnormal{Res}^{+}_{v(N)} F_{N}(v(N),w) e^{\rho(v(N),w)} dv(N) dw
=0,
	\end{multline*}
by (\ref{Eq-I-25}). 
	\end{proof}

	\begin{prop}\label{Prop-I-4}
Let $x$ be a generic bases with respect to $F$ as in (\ref{Eq-I-19}). Suppose that $0$ is not on any wall of $\mu(M)$. Then
\begin{enumerate}
\item[(a)] $\textnormal{JKRes}_{x} F(x) e^{ \rho(x) } dx$ depends continuously on $\rho$ small.

\item[(b)] $\textnormal{JKRes}_{x}F(x)dx$ does not depend on the choice of generic bases $x$, i.e. if $y$ is another generic bases with respect to $F$, then $\textnormal{JKRes}_{x} F(x) dx = \textnormal{JKRes}_{y} F(y) dy$.
\end{enumerate}
	\end{prop}

	\begin{proof}
\begin{enumerate}
\item[(a)] By Remark \ref{Rk-I-1} and Proposition \ref{Prop-I-3} we have that $\textnormal{Res}^{+}_{x_{r}} \ldots \textnormal{Res}^{+}_{x_{1}} F(x) e^{\rho(x)} dx$ depends continuously on small $\rho$.

\item[(b)] Let $\rho$ be small such that $0$ is generic for $F(x) e^{\rho(x)}$. Then 
	\begin{multline*}
\textnormal{JKRes}_{x} F(x) dx 
= 
\lim_{s\to 0} \textnormal{JKRes}_{x} F(x) e^{ s \rho(x) } dx
=
\lim_{s\to 0} \textnormal{JKRes}_{y} F(x) e^{ s \rho(y) } dy
=
\textnormal{JKRes}_{y} F(y) dy. 
\end{multline*}
\end{enumerate}
	\end{proof}

\subsection{Equivariant Jeffrey-Kirwan residue}\label{secEqRes}

The equivariant Jeffrey-Kirwan residue can be thought as a parametric version of the usual one, but the additional freedom in the choice of polarization gives more flexibility.

Let $\mathfrak{k}^{*}$ and $\mathfrak{s}^{*}$ be real vector spaces of dimension $q$ and $r-q$, respectively. Set $\mathfrak{t}^{*} = \mathfrak{k}^{*} \times \mathfrak{s}^{*}$. Let $s$ be a bases of $\mathfrak{s}^{*}$.

\begin{definition}
A $\mathfrak{k}^{*}$-pole (or simply a pole) in $\mathfrak{t}^{*}$ is a $q$-dimensional vector space $V$ such that $V\oplus \mathfrak{s}^{*} = \mathfrak{t}^{*}$.
\end{definition}

Let $x = \{ x_{1}, \ldots, x_{r} \}$ be an ordered bases of $\mathfrak{t}^{*}$. It induces polarization on each $V$ as follows. Let $\{x^{1}, \ldots, x^{r}\} \subset  \mathfrak{t}$ be its dual bases. Denote $\nu : \mathfrak{t} \to V^{*}$ the adjoint of the inclusion $V \hookrightarrow \mathfrak{t}^{*}$. Let $\{ v^{1} = \nu(x^{i_{1}}), \ldots, v^{q} = \nu (x^{i_{q}})\}$ be the bases of $V^{*}$ such that $i_{1} + \ldots + i_{q}$ is minimal and let $v = \{v_{1}, \ldots , v_{q}\} \subset  V$ be its dual bases. We call $v$ the bases of $V$ induced by $x$.

\begin{lemma*}
Let $\alpha\in V$ be a non-zero vector. Then $\alpha$ is polarized with respect to $x$ if and only if it is polarized with respect to $v$.
\end{lemma*}

\begin{proof}
Recall that $\alpha\in V$ is polarized with respect to $x$ if $\alpha(x^{1}) = \ldots = \alpha(x^{k-1}) = 0$ and $\alpha(x^{k})>0$ for some $k$. Then we have $\alpha( \nu(x^{1}) ) = \ldots = \alpha( \nu(x^{k-1}) ) = 0$ and $\alpha( \nu(x^{k}) )>0$.  Moreover, $\nu(x^{1}), \ldots, \nu(x^{k-1})$ cannot span $V^{*}$, hence by minimality condition $\nu(x^{k}) = v^{l}$ for some $l$ and $v^{1},\ldots,v^{l-1} \in \{ \nu(x^{1}), \ldots, \nu(x^{k-1}) \}$. Thus $\alpha$ is polarized by $v$.

Conversely, let $\alpha \in V$ be polarized by $v$, i.e. $\alpha(v^{1}) = \ldots = \alpha(v^{l-1}) = 0$ and $\alpha(v^{l})>0$. We have $\alpha(x^{i_{1}}) = \ldots = \alpha(x^{i_{l-1}}) = 0$ and $\alpha(x^{i_{l}}) > 0$. By minimality, for all $j<i_{l}$ we have $\nu(x^{j}) \in \langle \nu(x^{i_{1}}), \ldots, \nu(x^{i_{l-1}}) \rangle$, therefore $\alpha( \nu(x^{j}) ) = \alpha( x^{j} ) = 0$. Thus $\alpha$ is also polarized by $x$.
\end{proof}

\begin{definition}
Fix a scalar product on $\mathfrak{k}^{*}$ and on each $k^{*}$-pole $V$ we consider the pull-back scalar product via isomorphism $\textnormal{pr}_{\mathfrak{k}^{*}}|_{V}:V \to \mathfrak{k}^{*}$. Let $F = \sum_{I} \frac{ P_{I} e^{ \lambda_{I} } }{ \prod_{i\in I} \alpha_{i} }$ and we define its equivariant Jeffrey-Kirwan residue as
$$
\textnormal{EqRes}_{x}F(x) = \sum_{ V\ \mathfrak{k}^{*}\textnormal{-pole} } \textnormal{JKRes}_{v} \mathcal{F}(v,s) dv,
$$
where $v$ is the induced bases on $V$ by $x$ and $\mathcal{F}(v,s)$ is the expansion of $F(v,s)$ as $v \ll s$. 
\end{definition}

\begin{definition}
A fraction $\frac{P_{I}e^{\lambda_{I}}}{\prod_{i\in I}\alpha_{i}}$ is called $\mathfrak{k}^{*}$-generating if $\textnormal{pr}_{\mathfrak{k}^{*}}(\alpha_{i})$, $i\in I$ generate $\mathfrak{k}^{*}$.
\end{definition}

Any fraction $\frac{P_{I}e^{\lambda_{I}}}{\prod_{i\in I}}$ can be decomposed to sum of $\mathfrak{k}^{*}$-generating fractions of form $\frac{e^{\lambda_{I}}}{\prod\limits_{j\in J}\beta_{j}\prod\limits_{i=1}^{q}\alpha_{i}^{n_{i}+1}}$ with $\beta_{j}\in \mathfrak{s}^{*}$, and non-$\mathfrak{k}^{*}$-generating fractions. Remark that $\textnormal{EqRes}_{x}$ vanishes on non-$\mathfrak{k}^{*}$-generating fractions.

\begin{remark*}\label{Rk-I-2}
In the definition of $\textnormal{EqRes}$ it is enough to consider poles $V = \langle \alpha_{i_{1}}, \ldots, \alpha_{i_{q}} \rangle$ with $i_{1}, \ldots, i_{q} \in I$ since for other poles $\mathcal{F}(v,s)$ will be non-generating in $v$.
\end{remark*}

\begin{definition}
$0$ is generic with respect to $F$ if $0$ is generic with respect to $\mathcal{F}(v,s)$ in $v$ ($s$ is considered fixed parameter), where $v$ is the induced bases on $V$ for all poles $V = \langle \alpha_{i_{1}}, \ldots, \alpha_{i_{q}} \rangle$, $i_{1}, \ldots, i_{q} \in I$.
\end{definition}

\begin{definition}
$x$ is a generic bases with respect to $F$ if all induced bases $v$ on poles $V=\langle \alpha_{i_{1}}, \ldots, \alpha_{i_{q}} \rangle$, $i_{1}, \ldots, i_{q} \in I$, are generic with respect to $\mathcal{F}(v,s)$ ($s$ considered as fixed parameter). 
\end{definition}

\begin{lemma*} Let $x$ be an ordered bases of $\mathfrak{t}^{*}$,  $F = \sum_{I} \frac{P_{I} e^{\lambda_{I}}}{\prod_{i\in I} \alpha_{i}}$ and let $\mathcal{A} = \cup_{I} \{ \alpha_{i}\,|\, i\in I \}$. There exits bases $y$ which is generic with respect to $F$ and it induces the same polarization on $\mathcal{A}$ as $x$.
\end{lemma*}

\begin{proof}
By Remark \ref{Rk-I-4}, the polarization on $\mathcal{A}$ induced by $x$ corresponds to a connected component $\Lambda\subset \{ t \in \mathfrak{t} \,|\, \alpha(t)>0,\,\forall \alpha \in \mathcal{A} \}$. Let $\xi \in \Lambda$. We can choose a generic bases $y$ with respect to $F$ such that $y_{1}(\xi)>0$ and $y_{2}(\xi) = \ldots = y_{r}(\xi) = 0$. It will induce the same polarization as $x$.  
\end{proof}

\begin{prop}\label{Prop-I-6}
Let $(M, \omega)$ be a compact symplectic manifold with Hamiltonian $K \times S$-torus action and $\mu = \mu_{K} \times \mu_{S} : M \to \mathfrak{k}^{*} \times  \mathfrak{s}^{*}$ moment map. We suppose that $0\in \mathfrak{k}^{*}$ is a regular value of $\mu_{K}$ and let $F$ be as in (\ref{Eq-I-19}).
Let $x$ and $y$ be two ordered bases of $\mathfrak{t}^{*}$, generic with respect to $F$. Then
\begin{enumerate}
\item[(a)] $\textnormal{EqRes}_{x}F(x) = \textnormal{EqRes}_{y}F(y)$ (independence on polarization),
\item[(b)] $\textnormal{EqRes}_{x} F(x) e^{\rho(x)}$ depends continuously on $\rho\in \mathfrak{t}^{*}$ in a small neighborhood of $0$. 
\end{enumerate}
\end{prop}

\begin{proof}
Let $V$ be a pole and let $v$ and $w$ be bases of $V$ induced by $x$ and $y$, respectively. Denote $\mathcal{F}(v,s)$ and $\mathcal{F}(w,s)$ respectively the expansions of $F(v,s)$ and $F(w,s)$ with respect to $v \ll s$ and $w \ll s$. It is enough to show that $\textnormal{JKRes}_{v} \mathcal{F}(v,s) dv = \textnormal{JKRes}_{w} \mathcal{F}(w,s) dw$.

Let $\mathcal{J}$ be  the set of those $I$ such that there are $i_{1}, \ldots, i_{q} \in I$ with $V = \langle \alpha_{i_{1}}, \ldots, \alpha_{i_{q}} \rangle$. Let $\mathcal{F}_{I}(v,s)$ be the expansion of $\frac{P_{I}(v,s) e^{\lambda_{I}(v,s)} }{ \prod_{i\in I}\alpha_{i}(v,s) }$ with respect to $v \ll s$. If $I \notin \mathcal{J}$ then 
$$
\textnormal{JKRes}_{v} \sum_{I\in \mathcal{J}} \mathcal{F}_{I}(v,s) dv = 0,$$ since $\mathcal{F}_{I}(v,s)$ is non-generating in $v$. Thus 
$$
\textnormal{JKRes}_{v} \mathcal{F}(v,s) ds = \textnormal{JKRes}_{v} \sum_{I \in \mathcal{J} } \mathcal{F}_{I}(v,s) dv.
$$

Let $H \subset K$ be the subtorus such that $Lie(H) = V^{\perp} = \{ t \in \mathfrak{t} \,|\, \vartheta(t) = 0,\ \forall \vartheta\in V  \}$. Denote $\mathcal{V}$ the set of $N \subset M^{H}$ fixed point components such that $\dim \mu(N) = q$. 
Remark that $\mu_{T/H}(N) \subset V = Lie(T/H)^{*}$ and $\mu_{K}(N) = \textnormal{pr}_{\mathfrak{k}^{*}}|_{V}(\mu_{T/H} (N))$.  $0$ is not on any proper wall of $\mu_{T/H}(N)$ since $0$ is a regular value of $\mu_{K}$.
If $I \in \mathcal{J}$ then $\frac{P_{I} e^{\lambda_{I}}}{\prod_{i\in I} \alpha_{i}}$ is a summand of 
\begin{equation}\label{Eq-I-26}
\sum_{N\in \mathcal{V}} \oint_{N} \frac{i^{*}_{N}(\eta e^{\omega - \mu})}{e_{T}\mathcal{N}(N|M)} = \sum_{N\in \mathcal{V}} \sum_{X_{I} \subset N^{T}} \int_{X_{I}} \frac{i^{*}_{X_{I}}(\eta e^{\omega - \mu})}{e_{T}\mathcal{N}(X_{I}|M)}.
\end{equation}
Moreover, if $X_{I} \subset N^{T} = N^{T/H}$ for some $N\in \mathcal{V}$ then $I\in \mathcal{J}$. By Remark \ref{Rk-I-3}, $e_{T}\mathcal{N}(N|M)$ is invertible if $v \ll s$. Hence  the expansion  $\mathcal{G}(v,s)$ of (\ref{Eq-I-26}) with respect to $v \ll s$  is analytic in $v$. Thus
\begin{multline*}
\textnormal{JKRes}_{v}\mathcal{F}(v,s) dv
= \textnormal{JKRes}_{v}\sum_{I\in \mathcal{J}} \mathcal{F}_{I}(v,s) dv 
= \textnormal{JKRes}_{v} \mathcal{G}(v,s) dv
=
\\
\textnormal{JKRes}_{w} \mathcal{G}(w,s) dw
= \textnormal{JKRes}_{w} \sum_{I\in \mathcal{J}} \mathcal{F}_{I}(w,s) dw
= \textnormal{JKRes}_{w}  \mathcal{F}(w,s) dw,
\end{multline*}
by Proposition \ref{Prop-I-4}.

By Proposition \ref{Prop-I-4} also follows that 
$$\textnormal{EqRes}_{x} F(x)e^{\rho(x)} = \sum_{V pole} \textnormal{JKRes}_{v} \mathcal{F}(v,s)e^{\rho(v,s)} dv = \sum_{V pole} \textnormal{JKRes}_{v} \mathcal{G}(v,s) e^{\rho(v,s)} dv$$
depends continuously on $\rho$ in a small neighborhood of $0$. 
\end{proof}

We have the following analogue of Proposition \ref{Prop-I-1}.

\begin{prop}\label{Prop-I-5}
Let $x$ be an ordered bases of $\mathfrak{t}^{*}$ and let $z = \{ z_{1}, \ldots, z_{q} \}$ be an orthornormal bases of $\mathfrak{k}^{*}$. Consider $\alpha_{1},\ldots,\alpha_{q}\in \mathfrak{t}^{*}$ and let $\lambda = \lambda_{0}(s) + \lambda_{1}\widetilde{\alpha}_{1} + \ldots + \lambda_{q} \widetilde{\alpha}_{q}$. Then
$$
\textnormal{EqRes}_{x} \frac{e^{\lambda(x)}}{\prod_{i=1}^{q} \alpha_{i}(x)^{n_{i}+1}} 
= 
\frac{ e^{\lambda_{0}(s)} }{ \left| \det \left( \frac{ \partial\alpha_{i}(z,s) }{ \partial z_{j} } \right) \right|} \prod_{i=1}^{q}\frac{ \varepsilon(\alpha_{i})^{n_{i}+1} \lambda_{i}^{n_{i}} }{ n_{i}! },
$$
if $\det \left( \frac{ \partial\alpha_{i}(z,s) }{ \partial z_{j} } \right) \neq 0$ and $\lambda_{1}, \ldots, \lambda_{q} > 0$, otherwise it yields zero.
\end{prop}

\begin{proof}
Let $V = \langle \alpha_{1}, \ldots, \alpha_{q} \rangle$. $V$ is a pole if and only if $\det \left( \frac{ \partial\alpha_{i}(z,s) }{ \partial z_{j} } \right) \neq 0$. Suppose that this is the case, thus we have 
\begin{equation}\label{Eq-I-27}
\textnormal{EqRes}_{x} \frac{e^{\lambda(x)}}{\prod_{i=1}^{q} \alpha_{i}(x)^{n_{i}+1}} 
= \textnormal{JKRes}_{v} \frac{e^{\lambda(v,s)}}{\prod_{i=1}^{q} \alpha_{i}(v)^{n_{i}+1}} 
 dv 
= \textnormal{JKRes}_{v} \frac{ e^{ \lambda_{0}(s) + \sum_{i=1}^{q} \lambda_{i} \widetilde{\alpha}_{i} } }{ \prod_{i=1}^{q} \alpha_{i}(v)^{n_{i}+1} } dv,
\end{equation}
where $v$ is the induced bases on $V$ by $x$. 
Denote $z'_{i}$ the projection of $z_{i}$ to $V$ along $\mathfrak{s}^{*}$. Remark that $z'$ is an orthonormal bases of $V$ and  $\det \left( \frac{ \partial\alpha_{i}(z,s) }{ \partial z_{j} } \right) = \det \left( \frac{ \partial\alpha_{i}(z',s) }{ \partial z'_{j} } \right) $. By Proposition \ref{Prop-I-1}, (\ref{Eq-I-27}) is equal to 
$
\left|\det \left( \frac{ \partial\alpha_{i}(z',s) }{ \partial z'_{j} } \right) \right| \prod_{i=1}^{q} \frac{\varepsilon(\alpha_{i})^{n_{i}+1} \lambda_{i}^{n_{i}+1} }{ n_{i}! },
$
if $\lambda_{1}, \ldots, \lambda_{q} > 0$, otherwise (\ref{Eq-I-27}) is zero.  
\end{proof}

We also have an analogue of Corollary \ref{Cor-I-1}.

	\begin{corollary}\label{Cor-I-2}
Let $x$ be an ordered bases of $\mathfrak{t}^{*}$ and let $0$ be generic with respect to $F=\frac{ P e^{\lambda} }{ \prod_{i\in I} \alpha_{i} }$, where $\lambda \neq 0$. If $\textnormal{pr}_{\mathfrak{k}^{*}}(\lambda) \notin Cone( \textnormal{pr}_{\mathfrak{k}^{*}}( \widetilde{\alpha}_{i} ) \,|\, i\in I )$ then $\textnormal{EqRes}_{x} F(x) = 0$.
	\end{corollary}

	\begin{proof}
With notations of Proposition \ref{Prop-I-5} we have $\textnormal{pr}_{\mathfrak{k}^{*}}(\lambda) = \lambda_{1} \textnormal{pr}_{\mathfrak{k}^{*}}( \widetilde{\alpha}_{1} ) + \ldots + \lambda_{q} \textnormal{pr}_{\mathfrak{k}^{*}}( \widetilde{\alpha}_{q} )$. Then the statement follows from decomposition to $\mathfrak{k}^{*}$-generating and non-generating fractions, and from Proposition \ref{Prop-I-5}.
	\end{proof}

	\begin{corollary}\label{Cor-I-3}
Let $x=\{x_{1},\ldots,x_{r}\}$ be an ordered bases of $\mathfrak{t}^{*} = \mathfrak{k}^{*} \times \mathfrak{s}^{*}$ such that $x_{1},\ldots,x_{q} \in \mathfrak{k^{*}}$ and $x_{q+1}, \ldots, x_{r} \in \mathfrak{s}^{*}$.
Let $0$ be generic with respect to $F=\frac{ P e^{\lambda} }{ \prod_{i\in I} \alpha_{i} }$, where $\lambda \notin \mathfrak{s}^{*}$. 
If $\lambda$ is not polarized with respect to $x$, i.e. $\widetilde{\lambda} = -\lambda$ 
 then $\textnormal{EqRes}_{x} F(x) = 0$.
	\end{corollary}

	\begin{proof}
Remark that for any $\beta \in \mathfrak{t}^{*}$ we have $\textnormal{pr}_{\mathfrak{k}^{*}}(\widetilde{\beta}) = \widetilde{ \textnormal{pr}_{ \mathfrak{k}^{*} }(\beta) }$. Hence $Cone( \textnormal{pr}_{\mathfrak{k}^{*}}( \widetilde{\alpha}_{i} )  \,|\, i \in I) = Cone( \widetilde{ \textnormal{pr}_{\mathfrak{k}^{*}}( \alpha_{i} ) } \,|\, i \in I )$ is contained in the polarized cone of $\mathfrak{t}^{*}$. Since $\lambda\notin \mathfrak{s}^{*}$ and it is not polarized, hence  $\textnormal{pr}_{\mathfrak{k}^{*}}(\lambda) = - \textnormal{pr}_{\mathfrak{k}^{*}}(\widetilde{\lambda}) = - \widetilde{\textnormal{pr}_{\mathfrak{k}^{*}}(\lambda)} \neq 0$ is not contained in the polarized cone of $\mathfrak{t}^{*}$ and from Corollary \ref{Cor-I-2} follows that $\textnormal{EqRes}_{x}F(x) = 0$.
	\end{proof}


\section{Symplectic cut}

In this section we review in details the symplectic cut technique (\cite{Ler}, \cite{LMTW}, \cite{JKo}). We compute new fixed point data arising on the symplectic cut space: fixed point sets, Euler classes, orbifold multiplicities. Our setup is only a slightly different from the one in \cite{JKo}. The results of the section are summarized in the Atiyah-Bott-Berline-Vergne formula on the symplectic cut space (Theorem \ref{Thm-SC-B}) and they will be used in the subsequent section.

Consider $\mathbb{C}^{q}$ with  the standard symplectic form $\omega_{\mathbb{C}^{q}} = \frac{\sqrt{-1}}{2}\sum_{i=1}^{q} dz_{i} d\bar{z}_{i}$ and let $K=U(1)^{q}$ $q$-dimensional torus which act on $\mathbb{C}^{q}$ by weights $\gamma_{1},\ldots,\gamma_{q}\in \mathbb{Z}^{q}\simeq \mathfrak{k}^{*}_{\mathbb{Z}}$, i.e. $t\cdot (z_{1},\ldots, z_{q}) = (t^{\gamma_{1}}z_{1}, \ldots, t^{\gamma_{q}} z_{q})$. 
It is a Hamiltonian action with moment map $\psi:\mathbb{C}^{q} \to \mathfrak{k}^{*}$, $\psi(z) = \sum_{i=1}^{q} \gamma_{i}\frac{|z_{i}|^{2}}{2}$. 
Suppose that  $\gamma_{1},\ldots, \gamma_{q}$ are linearly independent.

Let $(M,\omega)$ be a symplectic manifold with Hamiltonian action of the $n$-dimensional torus $T$ with moment map $\mu_{T}:M\to \mathfrak{t}^{*}$. Let $K\subset T$ be a $q$-dimensional subtorus and denote its moment map by $\mu_{K}$. The product space $M\times \mathbb{C}^{q}$ is symplectic with symplectic form $\omega + \omega_{\mathbb{C}^{q}}$ and it admits Hamiltonian $T\times K$-action $(t,k)\cdot (m,z) := (t m, k^{-1} z)$ with moment map $(\mu_{T}, -\psi): M\times \mathbb{C}^{q} \to \mathfrak{t}^{*} \times \mathfrak{k}^{*}$. We consider the embedding $i_{diag}:K\to T\times K$, $k\mapsto (k,k^{-1})$ and we denote its image by $K_{diag}$, its Lie algebra by $\mathfrak{k}_{diag}$. The $K_{diag}$-moment map is equal to
$\Psi: M\times \mathbb{C}^{q} \to \mathfrak{k}^{*}_{diag}$, $\Psi= \mu_{K} - \psi$. Introduce notations
$\Gamma = \big\{ \sum_{i=1}^{q}a_{i}\gamma_{i}\in\mathfrak{k}^{*}\, |\, a_{i}\in \mathbb{R}_{\geq 0} \big\}$.

\begin{lemma*}
Suppose that $0\in \mathfrak{k}^{*}$ is a regular value of $\mu_{K}$. Then $0\in\mathfrak{k}_{diag}^{*}$ is a regular value of $\Psi$ if $\Gamma$ intersects the moment polytope $\mu_{K}(M)$ transversally, i.e. any face of $\Gamma$ intersects any wall of $\mu_{K}(M)$ transversally (\cite{JKo}).
\end{lemma*}

Remark that if any face of $\Gamma$ intersects any wall of $\mu_{K}(M)$ transversally then any face of $\textnormal{pr}^{-1}_{\mathfrak{k}^{*}}(\Gamma)$  intersects any wall of $\mu_{T}(M)$ transversally, where $\textnormal{pr}_{\mathfrak{k}^{*}}:\mathfrak{t}^{*} \to \mathfrak{k}^{*}$ canonical projection. From now on we suppose that the conditions of the lemma are fulfilled. Let $M_{0}=\mu^{-1}_{K}(0)/K$ be the symplectic quotient and denote the image of $m\in \mu_{K}^{-1}(0)$ under the quotient map by $[m]$.

\begin{definition}
The symplectic cut of $M$ with respect to the simplicial cone $\Gamma$ is defined as  
$$
M_{\Gamma} = \Psi^{-1}(0)/K_{diag}.
$$ 
We also denote the image of $(m,z)\in \Psi^{-1}(0)$ under the quotient map $\Psi^{-1}(0)\to M_{\Gamma}$ by $[m,z]$.
\end{definition}

\begin{remark*}
In general, $M_{\Gamma}$ is an orbifold even if $K$ acts freely on $\mu^{-1}(0)$, i.e $M_{0}$ is a manifold.	
\end{remark*}

The $T$-action on $M\times \mathbb{C}^{q}$ descends to $M_{\Gamma}$. It is Hamiltonian with moment map $M_{\Gamma}\to \mathfrak{t}^{*}$, $[m,z]\mapsto \mu_{T}(m)$ which we will also denote by $\mu_{T}$ (and $\mu_{K}$ for $K\subset T$). The image of moment maps equal to $\mu_{K}(M_{\Gamma}) = \mu_{K}(M)\cap \Gamma$ and $\mu_{T}(M_{\Gamma}) = \mu_{T}(M) \cap \textnormal{pr}_{\mathfrak{k}^{*}}^{-1}(\Gamma)$.

\subsection{$T$-fixed components on $M_{\Gamma}$}
\label{SC-Sec-1}	

	\begin{prop}\label{Prop-SC-A}
Let $(m,z) \in \Psi^{-1}(0)$, $m\in M$ and $z\in \mathbb{C}^{q}$. Denote $T_{m}^{\circ}$ and $K^{\circ}_{z}$ the unit components of isotropy groups of $m$ and $z$, respectively. Let $F_{m} \subset M^{T^{\circ}_{m}}$ and $F_{z} \subset (\mathbb{C}^{q})^{K^{\circ}_{z}}$ the connected components of fixed point sets containing $m$ and $z$, respectively. Then $[m,z] \in M_{\Gamma}^{T}$ if and only if $\mathfrak{t}_{m} \oplus \mathfrak{k}_{z} = \mathfrak{t}$, where $\mathfrak{t}_{m}$ and $\mathfrak{k}_{z}$ are the Lie algebras of the corresponding isotropy groups. Moreover, the connected component of $M_{\Gamma}^{T}$ containing $[m,z]$ is equal to 
$$
F_{[m,z]} = (F_{m} \times F_{z})/\!\!/K_{diag} = (F_{m} \times F_{z}) \cap \Psi^{-1}(0)/K_{diag}.
$$
\end{prop}

\begin{proof}
A point $[m,z]\in M_{\Gamma}$ is a $T$-fixed point if and only if for any $t\in T$ there is $k\in K$ such that $(tk\cdot m, k^{-1}\cdot z) = (m,z)$. That's it, if $T_{m} \subset T$ is the isotropy group of $m\in M$ and $K_{z}\subset K$ is the isotropy group of $z\in\mathbb{C}^{q}$ then $[m,z]\in M_{\Gamma}$ is a fixed point if and only if for all $t\in T$ there is $k\in K_{z}$ such that $s:=tk\in T_{m}$ or equivalently  $T_{m}\times K_{z} \to T$, $(s,k)\mapsto sk^{-1}$ is surjective or equivalently $\mathfrak{t}_{m}+ \mathfrak{k}_{z} = \mathfrak{t}$. Since $K_{diag}$ acts locally freely on $\Psi^{-1}(0)$, thus $\mathfrak{t}_{m}\cap \mathfrak{k}_{z} = \{ 0\}$.

Denote $\pi:\Psi^{-1}(0) \to M_{\Gamma}$ the quotient map. To check the inclusion $(F_{m}\times F_{z}) \cap \Psi^{-1}(0) \subset \pi^{-1}(F_{[m,z]})$, let $(m_{1},z_{1}) \in (F_{m}\times F_{z}) \cap \Psi^{-1}(0)$. Since the homomorphism $T^{\circ}_{m}\times K^{\circ}_{z}\to T$, $(s,k)\mapsto sk^{-1}$ is surjective, we can decompose any $t\in T$ as $t=sk^{-1}$ with $s\in T^{\circ}_{m}$ and $k\in K^{\circ}_{z}$ and we have
$
t\cdot [m_{1},z_{1}] = [sk^{-1}\cdot m_{1},z_{1}] = [s\cdot m_{1}, k^{-1}\cdot z_{1}] = [m_{1},z_{1}].
$
Therefore, 
$(F_{m}\times F_{z})/\!\!/K_{diag} \subset M_{\Gamma}^{T}$ and moreover by \cite{A} it is connected, containing $[m,z]$, hence $(F_{m}\times F_{z})/\!\!/K_{diag} \subset F_{[m,z]}$.

To show the reverse inclusion, let $(m_{2},z_{2})\in \pi^{-1}(F_{[m,z]})$ be in a small neighborhood of $(m_{1},z_{1})$ in $(F_{m}\times F_{z})\cap \Psi^{-1}(0)$. The inclusion $[m_{2},z_{2}]\in M_{\Gamma}^{T}$ implies that $\mathfrak{t}_{m_{2}} \oplus \mathfrak{k}_{z_{2}} = \mathfrak{t}$. Similarly, $(m_{1},z_{1})\in F_{m}\times F_{z}$ implies that $T^{\circ}_{m}\subset T_{m_{1}}$ and $K^{\circ}_{z}\subset K_{z_{1}}$, thus $\mathfrak{t} = \mathfrak{t}_{m} \oplus \mathfrak{k}_{z} \subset \mathfrak{t}_{m_{1}}\oplus \mathfrak{k}_{z_{1}} = \mathfrak{t}$, therefore $\mathfrak{t}_{m}=\mathfrak{t}_{m_{1}}$ and $\mathfrak{k}_{z}=\mathfrak{k}_{z_{1}}$. The isotropy groups locally decrease, thus we have $T_{m_{2}}\subset T_{m_{1}}$ and $K_{z_{2}}\subset K_{z_{1}}$.  For Lie algebras
$\mathfrak{t} = \mathfrak{t}_{m_{2}} \oplus \mathfrak{k}_{z_{2}} \subset \mathfrak{t}_{m_{1}} \oplus \mathfrak{k}_{z_{1}} = \mathfrak{t}$, hence $\mathfrak{t}_{m_{2}} = \mathfrak{t}_{m_{1}} = \mathfrak{t}_{m}$ and $\mathfrak{k}_{z_{2}} = \mathfrak{k}_{z_{1}} = \mathfrak{k}_{z}$ which yields $T^{\circ}_{m}\subset T_{m_{2}}$ and $K^{\circ}_{z}\subset K_{z_{2}}$, hence $(m_{2},z_{2})\in (F_{m}\times F_{z})\cap \Psi^{-1}(0)$. This shows that $(F_{m}\times F_{z})\cap \Psi^{-1}(0)$ is open in $\pi^{-1}(F_{[m,z]})$, but it is also closed being a $T^{\circ}_{m} \times K_{z}^{\circ}$ fixed point component of $\Psi^{-1}(0)$. From the connectedness of $\pi^{-1}(F_{[m,z]})$ we have that $(F_{m} \times F_{z}) \cap \Psi^{-1}(0) = \pi^{-1}(F_{[m,z]})$, i.e. $(F_{m}\times F_{z})/\!\!/K_{diag} = F_{[m,z]}$. 
\end{proof}

We introduce notations
$M_{int} = \{ m\in M\, |\, \mu_{K}(m) \in \textnormal{int}\,\Gamma\}$,
$M_{\Gamma,int} = \{ [m,z]\in M_{\Gamma}\, |\, \mu_{K}([m,z])\in \textnormal{int}\,\Gamma \}$,
$M_{\Gamma,0} = \{[m,z]\in M_{\Gamma}\, |\, z=0 \}.$  
We sort the fixed point components in three groups:
	\begin{enumerate}
	\item[(0)] fixed point components $H_{0} \subset M_{\Gamma,0}^{T}$. They are characterized by $\mu_{K}(H_{0})=0$ and they can be identified to $T$-fixed point components $F_{0}$ in $M_{0}$ via the $T$-equivariant diffeomorphism $\Xi_{0}:M_{0}\to M_{\Gamma,0}$, $[m]\mapsto [m,0]$.

	\item[(1)] fixed point components $H_{1} \subset M_{\Gamma,int}^{T}$, characterized by $\mu_{K}(H_{1})\in \textnormal{int}\,\Gamma$ and they corresponds to fixed point components of $M_{int}$ as follows. Let $K'\subset K$ be the maximal subgroup of $K$ acting trivially on $\mathbb{C}^{q}$. For any $m\in M_{int}$ we can write $\mu_{K}(m) = \frac{1}{2} \sum_{i=1}^{q}\mu_{K}(m)_{i}\gamma_{i}$ with $\mu_{K}(m)_{i}>0$ for all $i=1,\ldots,q$. The map $M_{int} \to M_{\Gamma,int}$, $m \mapsto \big[ m, \big(\sqrt{\mu_{K}(m)_{1}},\ldots, \sqrt{\mu_{K}(m)_{q}}\,\big) \big]$ induces $T$-equivariant isomorphism $\Xi_{int}:M_{int}/K' \to M_{\Gamma,int}$ of orbifolds. Under this map the fixed point components $H_{1}$ corresponds to suborbifolds $F_{1}/K$ of $M_{int}/K'$ where $F_{1}$ is a $T$-fixed point component of $M_{int}$ (remark that $K'$ acts trivially on $F_{1}$).

	\item[(2)] Other fixed point components $H \subset M_{\Gamma}^{T}$ characterized by $\mu_{K}(H)$ being non-zero and it lies on the boundary of $\Gamma$. If $K$ is $1$-dimensional then this type of components do not occur. 
	\end{enumerate}

\subsection{Choice of cohomology classes and their restrictions}

The projection $M\times \mathbb{C}^{q} \to M$ induces an isomorphism $H_{T}(M)\stackrel{\sim}{\longrightarrow} H_{T}(M\times \mathbb{C}^{q})$ and the homomorphism $T\times K_{diag}\to T$, $(t,k)\mapsto tk$ induces a homomorphism $H_{T}(M\times \mathbb{C}^{q}) \to H_{T\times K_{diag}}(M\times \mathbb{C}^{q})$. Their composition we denote by $\delta: H_{T}(M)\to H_{T\times K_{diag}}(M\times \mathbb{C}^{q}).$

\begin{definition}
Define the cut homomorphism $\Delta:H_{T}(M) \to H_{T}(M_{\Gamma})$, $\Delta= \kappa_{T}\circ \delta,$
where $\kappa_{T}: H_{T\times K_{diag}}(M\times \mathbb{C}^{q}) \to H_{T}(M_{\Gamma})$ is the $T$-equivariant Kirwan map. 
More explicitly, if $\alpha$ is an equivariantly closed form on $M$ then
$
\Delta(\alpha)(u) = \textnormal{Hor}_{\theta}(j^{*}\alpha(u+\Theta(u))),
$
(cf. \cite{BT}) where $u\in \mathfrak{t}$, $\theta$ is a $T$-invariant connection form on the principal $K_{diag}$-bundle $\Psi^{-1}(0)\to M_{\Gamma}$ with equivariant curvarture form $\Theta$ and $j:\Psi^{-1}(0)\to M \times \mathbb{C}^{q}$ is the inclusion.
\end{definition}

	\begin{prop*}\label{Prop-SC-B}
$\qquad$
{\renewcommand{\labelenumi}{(\arabic{enumi})}
\begin{enumerate}
\item We have $\Xi_{0}^{*}(\Delta(\alpha)|_{M_{\Gamma,0}}) = \kappa_{T/K}(\alpha)$ where $\kappa_{T/K}:H_{T}(M)\to H_{T/K}(M_{0})$ is the equivariant Kirwan map.
In particular, if $F_{0}\subset M_{0}^{T}$ and $H_{0}\subset M_{\Gamma,0}^{T}$ are fixed point components such that $\Xi_{0}(F_{0})=H_{0}$ then 
$
\Xi_{0}^{*}(\Delta(\alpha)|_{H_{0}}) = \kappa_{T/K}(\alpha)|_{F_{0}}.
$

\item As equivariant differential forms 
$\Xi_{0}^{*}(\Delta(\omega - \mu_{T})|_{M_{\Gamma,0}})$ is the reduced equivariant differential form on $M_{0}$.
\item\label{SC-Ass3} Let $F_{1}\subset M_{int}^{T}$ and $H_{1}\subset M_{\Gamma,int}^{T}$ be fixed components such that $\Xi_{int}(F_{1})=H_{1}$. Then
$
\Xi_{int}^{*}(\Delta(\alpha)|_{H_{1}}) = \alpha|_{F_{1}}.
$ In particular, $\Xi_{int}^{*}(\Delta(\omega - \mu_{T})|_{H_{1}}) = (\omega - \mu_{T})|_{F_{1}}$.
\end{enumerate}}	
	\end{prop*}
	\begin{proof}
{\renewcommand{\labelenumi}{(\arabic{enumi})}
\begin{enumerate}

\item We have the following isomomorphism of principal $K$-bundles
$$
\xymatrix{
\mu^{-1}_{K}(0) \ar[r]^-{\sim} \ar[d] & \mu^{-1}_{K}(0)\times \{0\} \ar[d]
\\
M_{0} \ar[r]^{\sim}_{\Xi_{0}} & M_{\Gamma,0}. 
}
$$	
Let $j:\mu_{K}^{-1}(0) \hookrightarrow M$ and $j_{0}:\mu_{K}^{-1}\times \{0\} \hookrightarrow M\times \mathbb{C}^{q}$ inclusions. Let $\theta$ be a connection form on principal $K$-bundles $\mu_{K}^{-1}(0)\times \{0\}\to M_{\Gamma,0}$ and $\Xi_{0}^{*}\,\theta$ its pull-back to $\mu_{K}^{-1}(0)\to M_{0}$. From the definitions follows that
$$
\Xi_{0}^{*}(\Delta(\alpha)|_{M_{\Gamma,0}}) 
= \Xi_{0}^{*}\textnormal{Hor}_{\theta}(j_{0}^{*}\alpha(u+\Theta(u))) 
= \textnormal{Hor}_{\Xi^{*}_{0}\theta}(j_{0}^{*}\alpha(u+\Xi_{0}^{*}\Theta(u)))
= \kappa_{T/K}(\alpha).
$$

\item In particular, for $\alpha = \omega - \mu_{T}$ equivariant symplectic form we have equality of equivariant forms
$\Xi_{0}^{*}(\Delta(\omega -\mu_{T})|_{M_{\Gamma}}) = \kappa_{T/K}(\omega -\mu_{T})$.

\item We have the following commutative diagram
$$
\xymatrix{
F_{1} \ar[r]^-{\sigma} \ar[d] & \Psi^{-1}(0)\cap (F_{1}\times \mathbb{C}^{q}) \ar[d]
\\
F_{1}/K' \ar[r]_-{\Xi_{int}}^-{\sim} & H_{1}
}
$$
where $\sigma(m) = \big(m, (\sqrt{\mu_{K}(m)_{1}}, \ldots ,\sqrt{\mu_{K}(m)_{q}}\, )\big)$ is a global section of the principal $K_{diag}$-bundle $\Psi^{-1}(0)\cap (F_{1} \times \mathbb{C}^{q}) \to F_{1}$ and yields isomorphism $F_{1}\times K_{diag}/K' \simeq \Psi^{-1}(0)\cap (F_{1}\times\mathbb{C}^{q})$. 
Let $j_{1}$ be the composition of inclusion $\Psi^{-1}(0)\cap (F_{1}\times\mathbb{C}^{q}) \hookrightarrow M\times \mathbb{C}^{q}$ and projection $M\times \mathbb{C}^{q} \to M$. Let $\theta$ be the pull-back of the Maurer-Cartan form to the principal $K_{diag}$-bundle $\Psi^{-1}(0)\cap (F_{1}\times\mathbb{C}^{q}) \simeq F_{1}\times K_{diag}/K'$ and $\Theta$ be its equivariant curvature form. Then the restriction of the class $\Delta(\alpha)$ to $H_{1}$ is equal to
$$
\Delta(\alpha)(u)\big|_{H_{1}} = \textnormal{Hor}_{\theta}(j^{*}_{1}\alpha(u + \Theta(u)))
$$
Moreover, if $\{\xi^{1}, \ldots, \xi^{n}\}$ is a bases of $\t$ then $\Theta(u) = - \sum_{i=1}^{q}u_{i}\iota_{\underline{\xi}^{i}}\theta + \Theta = 0$ and $\sigma^{*}\theta =0$. Therefore $\sigma^{*}(\Delta(\alpha)|_{H_{1}}) = \alpha|_{F_{1}}$ implying the assertion (\ref{SC-Ass3}). Remark this is an equality of equivariant differential forms. In particular, for $\alpha = \omega - \mu_{T}$ equivariant symplectic form on $M$ we have the equality 
$\Xi_{int}^{*}(\Delta(\omega - \mu_{T})|_{H_{1}}) = (\omega - \mu_{T})|_{F_{1}}.$
\end{enumerate}}
	\end{proof}

\subsection{$T$-equivariant Euler classes of normal bundles of fixed point components}

	\begin{lemma}[cf. \cite{JKo} Proposition 2.2]\label{Lem-SC-3}
Let $R$ be a $K$-invariant symplectic submanifold of a Hamiltonian $K$-manifold $M$ with moment map $\mu:M\to\mathfrak{k}^{*}$. Suppose that $0$ is a regular value of $\mu$ and denote $M_{0}=\mu^{-1}(0)/K$. If $R\cap\mu^{-1}(0)\neq\emptyset$ then the symplectic quotient $R_{0}=(R\cap\mu^{-1}(0))/K$ exists and we have isomorphism of normal bundles 
$$
\N(R_{0}|M_{0})\simeq\N(R|M)/\!\!/K.$$
If there is a second $T$-action on $M$	which commutes with $K$ and $R$ is $T$-invariant, then the isomorphism of normal bundle is $T$-equivariant.
	\end{lemma}

	\begin{proof}
$0$ is a regular value of $\mu$ is equivalent to the locally free action of $K$ on  $\mu^{-1}(0)$. Hence $K$ acts locally freely on the set $R\cap\mu^{-1}(0)=(\mu|_{R})^{-1}(0)$. Therefore, $0$ is a regular value of $\mu|_{R}$ and the symplectic quotient $R_{0}=(\mu|_{R})^{-1}(0)/K=(R\cap\mu^{-1}(0))/K$ exists.

Let $\pi:\mu^{-1}(0)\to M_{0}$ the quotient map. We have the following commutative diagram of vector bundles over $\pi^{-1}R_{0}$ with short exact sequences  in rows and in the first two columns
$$
\xymatrix{
		  & 0 \ar[d] & 0 \ar[d] & & 
 \\ 
0\ar[r] & \pi^{-1}(R_{0})\times \mathfrak{k} \ar[r]\ar[d] &  (\pi^{-1}(M_{0})\times \mathfrak{k})|_{\pi^{-1}(R_{0})}\ar[r]\ar[d] & 0 \ar[d] &
\\
0\ar[r] & TR|_{\pi^{-1}(R_{0})} \ar[r]\ar[d]& TM|_{\pi^{-1}(R_{0})} \ar[r]\ar[d]& \N(R|M)|_{\pi^{-1}(R_{0})}\ar[r]\ar[d] & 0
\\
0\ar[r] & \pi^{-1}(TR_{0}) \ar[r]\ar[d]& \pi^{-1}(TM_{0}|_{R_{0}}) \ar[r]\ar[d]& \pi^{-1}\N(R_{0}|M_{0}) \ar[r]\ar[d]& 0
\\
 & 0 & 0 & 0 &
}
$$
The $9$-lemma implies the exactness of the last column and the lemma follows. In the case of the second action the diagram is $T$-equivariantly commutative, hence the isomorphism of bundles will be also $T$-equivariant.
	\end{proof}

We have a finite cover $T_{m}\times K_{z} \to T$, $(s,k) \mapsto sk$, since $\mathfrak{t} = \mathfrak{t}_{m} \oplus \mathfrak{k}_{z}$. Consequently, for computation of $T$-equivariant Euler classes we may suppose that $T=T_{m}\times K_{z}$. 
From Proposition \ref{Prop-SC-A} and Lemma \ref{Lem-SC-3} follows that
\begin{equation}\label{Eq-SC-1}
\mathcal{N}(F_{[m,z]}|M_{\Gamma}) 
\simeq 
\mathcal{N}(F_{m}\times F_{z}|M\times \mathbb{C}^{q})/\!\!/_{K_{diag}} 
\simeq 
\big( \textnormal{pr}_{1}^{*} \mathcal{N}(F_{m}|M) \oplus \textnormal{pr}_{2} ^{*}\mathcal{N}(F_{z}|\mathbb{C}^{q}) \big) /\!\!/ K_{diag}.
\end{equation}
 Choosing a $T$-invariant compatible triplet $(\omega,g,\mathcal{I})$ on $M$ the normal bundles $\mathcal{N}(F_{m}|M)$ and $\mathcal{N}(F_{m}\times F_{z}|M\times\mathbb{C}^{q})$ become complex vector bundles. 

The normal bundle $\mathcal{N}(F_{m}|M)$ splits $T_{m}$-equivariantly to sum of complex line bundles $L'_{i} \to F_{m}$, i.e. $\mathcal{N}(F_{m}|M) = \bigoplus_{i\in I_{m}} L'_{i}$. This splitting is  $T=T_{m}\times K_{z}$-equivariant since the actions of $T_{m}$ and $K_{z}$ commute. Denote $L_{i}$ the pull-back of $L'_{i}$ along the projection $\textnormal{pr}_{1}:F_{m}\times F_{z}\to F_{m}$ which is $\phi_{1}:T\times K_{diag}\to T$, $(t,k)\mapsto tk$ intertwining. Therefore, $L_{i}\to F_{m}\times F_{z}$ is a complex $T\times K_{diag}$-line bundle and we have $T\times K_{diag}$-equivariant isomorphism of vector bundles 
\begin{equation}\label{Eq-SC-2}
\textnormal{pr}_{1}^{*}\mathcal{N}(F_{m}|M) = \bigoplus_{i\in I_{m}} L_{i}.
\end{equation}

For a fixed $z = ( z_{1}, \ldots, z_{q} ) \in \mathbb{C}^{q}$ we set $J_{z} = \{ j \,|\, z_{j} = 0 \}$.
 We have $K$-equivariant splitting $\mathcal{N}(F_{z}|\mathbb{C}^{q})=\bigoplus_{j\in J_{z}} F_{z}\times \mathbb{C}_{\gamma_{j}}$, where $K$ acts on $\mathbb{C}_{\gamma_{j}}$ with weight $\gamma_{j}$. 
The pull-back of the line bundle $F_{z}\times \mathbb{C}_{\gamma_{j}}\to F_{z}$ along the $\phi_{2}:T\times K_{diag}\to K$, $(t,k)\mapsto k^{-1}$ intertwining projection $\textnormal{pr}_{2}:F_{m}\times F_{z}\to F_{z}$
is equal to the $T\times K_{diag}$-line bundle $\mathcal{L}_{j} := (F_{m}\times F_{z})\times \mathbb{C}_{\gamma_{j}}\to F_{m}\times F_{z}$ 
and we have $T\times K_{diag}$-isomorphism of vector bundles
\begin{equation}\label{Eq-SC-3}
\textnormal{pr}_{2}^{*}\mathcal{N}(F_{z}|\mathbb{C}^{q}) = \bigoplus_{j\in J_{z}}\mathcal{L}_{j}.
\end{equation}

The direct sum decomposition $\mathfrak{t}_{m} \oplus \mathfrak{k}_{z} = \mathfrak{t}$ induces isomorphism $\mathfrak{t}^{*} \simeq \mathfrak{t}^{*}_{m} \times \mathfrak{k}^{*}_{z}$. Moreover, with identifications $
\mathfrak{t}^{*}_{m}=\{\lambda \in \mathfrak{t}^{*} \,|\, \lambda(\mathfrak{k}_{z})=0\}\subset \mathfrak{t}^{*}$ and 
$\mathfrak{k}^{*}_{z} = \{ \lambda \in \mathfrak{t}^{*} \,|\, \lambda(\mathfrak{t}_{m})=0 \}\subset \mathfrak{t}^{*}$ we can write $\mathfrak{t}^{*} = \mathfrak{t}^{*}_{m} \oplus \mathfrak{k}^{*}_{z}$. 
Denote $\varrho_{m} : \mathfrak{t}^{*}_{m} \hookrightarrow \mathfrak{t}^{*}$ and $\varrho_{z} : \mathfrak{k}^{*}_{z} \hookrightarrow \mathfrak{t}^{*}$ inclusions.
Similarly, from direct sum decomposition $\mathfrak{k}=(\mathfrak{k}\cap \mathfrak{t}_{m}) \oplus \mathfrak{k}_{z}$ we get $\mathfrak{k}^{*} = (\mathfrak{k}\cap \mathfrak{t}_{m})^{*} \oplus \mathfrak{k}^{*}_{z}$ with identifications
$
(\mathfrak{k}\cap \mathfrak{t}_{m})^{*} = \{ \beta \in \mathfrak{k}^{*}\ |\ \beta(\mathfrak{k}_{z}) = 0 \} \subset \mathfrak{k}^{*}$
and
$\mathfrak{k}_{z}^{*} = \{ \beta \in \mathfrak{k}^{*} \ |\ \beta(\mathfrak{k} \cap \mathfrak{t}_{m}) = 0 \} \subset \mathfrak{k}^{*}.
$ 
We consider $\rho_{m} : \mathfrak{t}^{*}_{m} \to \mathfrak{k}^{*}$ and the inclusion $\rho_{z}: \mathfrak{k}^{*}_{z} \to \mathfrak{k}^{*}$ satisfying commutative diagrams
$$
\xymatrix{ \mathfrak{t}_{m}^{*} \ar[d] \ar[r]^{\varrho_{m}} \ar[rd]^{\rho_{m}} & \mathfrak{t}^{*} \ar[d] 
\\ 
(\mathfrak{k} \cap \mathfrak{t}_{m})^{*} \ar[r]  & \mathfrak{k}^{*}}
\qquad\textnormal{and}\qquad
\xymatrix{ \mathfrak{k}_{z}^{*}  \ar[r]^{\varrho_{z}} \ar[rd]^{\rho_{z}} & \mathfrak{t}^{*} \ar[d] 
\\ 
 & \mathfrak{k}^{*}}.
$$

The $T_{m}\times K_{z}$-equivariant Euler class of $L_{i}$ has of form $e_{T_{m}
\times K_{z}}(L_{i}) = \alpha_{i} + e_{K_{z}}(L_{i})$ for some $\alpha_{i}\in \mathfrak{t}_{m}$ and therefore the $T\times K_{diag}$-equivariant Euler class is equal to
$$
e_{T\times K_{diag}}(L_{i}) = \phi_{1}^{*}(e_{T}(L_{i})) = \varrho_{m}(\alpha_{i}) + \rho_{m}(\alpha_{i}) + (\varrho_{z}\times\rho_{z})
(e_{K_{z}}(L_{i})).$$ 
More explicitly, for any $u \in \mathfrak{t}$ and $v \in \mathfrak{k}$ we have
$$
e_{T\times K_{diag}}(L_{i})(u,v) = \alpha_{i} \big( \textnormal{pr}_{\mathfrak{t}_{m}}(u +v) \big) + e_{K_{z}}(L_{i}) \big( \textnormal{pr}_{\mathfrak{k}_{z}}(u + v) \big), 
$$
where $\textnormal{pr}_{\mathfrak{t}_{m}}:\mathfrak{t} \to \mathfrak{t}_{m}$ and $\textnormal{pr}_{\mathfrak{k}_{z}}:\mathfrak{t} \to \mathfrak{k}_{z}$ are projection along $\mathfrak{k}_{z}$ and  $\mathfrak{t}_{m}$, respectively. Moreover, we have
$$
e_{T\times K_{diag}}(\mathcal{L}_{j}) = \phi_{2}^{*}e_{K}(\mathbb{C}_{\gamma_{j}}) = -\gamma_{j}.
$$

Let $\kappa: H_{K_{diag}}(F_{m}\times F_{z}) \to H(F_{[m,z]})$ and $\kappa_{T}: H_{T\times K_{diag}}(F_{m}\times F_{z}) \to H_{T}(F_{[m,z]})$ be the usual and $T$-equivariant Kirwan maps, respectively. Let $\theta$ be a $T$-invariant connection form on the principal $K_{diag}$-bundle $(F_{m} \times F_{z})\cap \Psi^{-1}(0) \to F_{[m,z]}$. Consider the map $\mathfrak{t} \to \mathfrak{k}$ given by $u \mapsto \iota_{\underline{u}}\theta$.
	\begin{remark*}
	{\renewcommand{\labelenumi}{(\roman{enumi})}
	\begin{enumerate}
\item If $u\in \mathfrak{t}_{m}$ then the fundamental vector field on $F_{m}\times F_{z}$ equals $\underline{u}=0$.
\item If $u\in \mathfrak{k}_{z}$ then $\iota_{\underline{u}}\theta = u$ because $K_{z}\subset T$ acts on $F_{m}\times F_{z}$ as $K_{z}\subset K_{diag}$.

\item For all $u\in \mathfrak{t}$ we have $\iota_{\underline{u}}\theta = \textnormal{pr}_{k_{z}}u$.
	\end{enumerate}}
	\end{remark*}

\noindent We compute $\kappa(e_{K_{diag}}(L_{i})) = \textnormal{Hor}_{\theta}[ \alpha_{i}(\textnormal{pr}_{\mathfrak{t}_{m}}d\theta) + e_{K_{z}}(L_{i})(\textnormal{pr}_{\mathfrak{k}_{z}}d\theta)]$ and 
	\begin{align*}
\kappa_{T}(e_{T\times K_{diag}}(L_{i}))(u) 
&{}= 
\textnormal{Hor}_{\theta} \big[ \alpha_{i} \big( \textnormal{pr}_{\mathfrak{t}_{m}} ( u - \iota_{\underline{u}}\theta + d\theta ) \big) + e_{K_{z}}(L_{i}) \big( \textnormal{pr}_{\mathfrak{k}_{z}}(u-\iota_{\underline{u}}\theta + d\theta)\big) \big]
\\
&{}=
\textnormal{Hor}_{\theta} \big[ \alpha_{i} \big( \textnormal{pr}_{\mathfrak{t}_{m}}(u - \textnormal{pr}_{\mathfrak{k}_{z}}u + d\theta ) \big) + e_{K_{z}}(L_{i}) \big( \textnormal{pr}_{\mathfrak{k}_{z}}(u - \textnormal{pr}_{\mathfrak{k}_{z}}u + d\theta ) \big) \big]
\\
&{}= 
\alpha_{i}(\textnormal{pr}_{\mathfrak{t}_{m}}u) + \textnormal{Hor}_{\theta} \big[ \alpha_{i}(\textnormal{pr}_{\mathfrak{t}_{m}}d\theta) + e_{K_{z}}(L_{i})(\textnormal{pr}_{\mathfrak{k}_{z}} d\theta) \big]
\\
&{}=
\alpha_{i}(\textnormal{pr}_{\mathfrak{t}_{m}}u) + \kappa \big( e_{K_{diag}}(L_{i} ) \big),
	\end{align*}	
thus
\begin{equation}
\kappa_{T}(e_{T\times K_{diag}}(L_{i})) = \varrho_{m}(\alpha_{i}) + \kappa(e_{K_{diag}}(L_{i})).
\end{equation}
Similarly,
$$
\kappa_{T}(e_{T\times K_{diag}}(\mathcal{L}_{j}))(u) =
\textnormal{Hor}_{\theta} \big[ -\gamma_{j}(-\iota_{\underline{u}}\theta + d\theta ) \big]
=
\gamma_{j}(\textnormal{pr}_{\mathfrak{k}_{z}}u) - \textnormal{Hor}_{\theta} \big[ \gamma_{j}(d\theta) \big]
=
\gamma_{j}(\textnormal{pr}_{\mathfrak{k}_{z}}u) - \kappa(\gamma_{j}),
$$
thus
\begin{equation}\label{Eq-SC-5}
\kappa_{T}(e_{T\times K_{diag}}(\mathcal{L}_{j})) = \varrho_{z}(\gamma_{j}) - \kappa(\gamma_{j}).
\end{equation}
From equations (\ref{Eq-SC-1}-\ref{Eq-SC-5}) we have
\begin{align}
e_{T}(\mathcal{N}(F_{[m,z]}|M_{\Gamma})) 
&{}=\prod_{i\in I_{m}}\kappa_{T}(e_{T\times K_{diag}}(L_{i}))\prod_{j\in J_{z}}(e_{T\times K_{diag}}(\mathcal{L}_{j})) 
\notag
\\
&{}= 
\prod_{i\in I_{m}} \big[ \varrho_{m}(\alpha_{i}) + \kappa(e_{K_{diag}}(L_{i})) \big] \prod_{j\in J_{z}} \big[ \varrho_{z}(\gamma_{j}) - \kappa(\gamma_{j}) \big].
\label{SC-Eq-5b}
\end{align}

\begin{remark*}
\begin{enumerate}
\item  We have $\mathfrak{k}_{z} = \cap_{j \notin J_{z}} \ker \gamma_{j}$ and $F_{z} = \{ w \in \mathbb{C}^{q} \,|\,  w_{j} =0,\,\forall j \in J_{z} \}$. Moreover, $\Gamma_{z}: = Cone( \gamma_{j} \,|\, j \notin J_{z} ) = \psi(F_{z})$, hence  $\mu_{T}(F_{[m,z]}) = \mu_{T}(F_{m}) \cap \chi^{-1}(\Gamma_{z})$.

\item The set $\{\varrho_{z}(\gamma_{j}) \,|\, j\in J_{z}\}$ is a bases of $\mathfrak{k}^{*}_{z}$ and $\mu_{T}(F_{m}) \subset \mu_{T}(m) + \mathfrak{k}_{z}^{*} \subset \mathfrak{t}^{*}$. 

\item The underlying vector space of $\chi^{-1}(\Gamma_{z})$ is equal to $\mathfrak{t}^{*}_{m}$.
\end{enumerate}
\end{remark*}

	\begin{remark*}
In case of $T=K\times S$ denote $\sigma:\mathfrak{t}^{*}\to \mathfrak{s}^{*}$ the  restriction and introduce maps $\sigma_{m}:\mathfrak{t}^{*}_{m} \to \mathfrak{s}^{*}$, $\sigma_{m}=\sigma\circ\varrho_{m}$ and $\sigma_{z}:\mathfrak{k}^{*}_{z} \to \mathfrak{s}^{*}$, $\sigma_{z}=\sigma\circ\varrho_{z}$. Then $\varrho_{m} = \rho_{m}+\sigma_{m}$ and $\varrho_{z} = \rho_{z} + \sigma_{z}$. Hence,
	\begin{align}
e_{K\times S}\mathcal{N}(F_{[m,z]}|M_{\Gamma}) 
&= 
\prod_{i\in I_{m}} \big[ \rho_{m}(\alpha_{i}) + \sigma_{m}(\alpha_{i}) + \kappa(e_{K_{diag}}(L_{i})) \big] \prod_{j\in J_{z}} \big[ \rho_{z}(\gamma_{j}) + \sigma_{z}(\gamma_{j}) - \kappa(\gamma_{j}) \big]
\notag
\\
&{}=
\prod_{i\in I_{m}} \big[ \rho_{m}(\alpha_{i}) + \kappa_{S}(e_{K_{diag}}(L_{i})) \big] \prod_{j\in J_{z}} \big[ \rho_{z}(\gamma_{j}) - \kappa_{S}(\gamma_{j}) \big].
	\end{align}
Moreover, the splitting $\mathfrak{t} = \mathfrak{k} \oplus \mathfrak{s}$ induces isomorphism $\mathfrak{t}^{*} \simeq \mathfrak{k}^{*} \times \mathfrak{s}^{*}$ and we have identification 
$
\mathfrak{s}^{*} = \{ \lambda\in \mathfrak{t}^{*}\ |\ \lambda(\mathfrak{k})=0 \} \subset \mathfrak{t}^{*}.
$
For $z=0$ we have $\mathfrak{k}_{0} = \mathfrak{k}$, thus $\mathfrak{t}_{m}^{*} = \mathfrak{s}^{*}$, hence $\rho_{m}=0$, $\rho_{z} = 1_{\mathfrak{k}^{*}}$ and
\begin{align}
e_{K\times S}(\mathcal{N}(F_{[m,0]}|M_{\Gamma})) 
&{}= 
\prod_{i\in I_{m}} \big[ \sigma_{m}(\alpha_{i}) + \kappa(e_{K_{diag}}(L_{i})) \big] \prod_{j=1}^{q} \big[ \gamma_{j} -\kappa_{S}(\gamma_{j}) \big] 
\notag
\\
&{}=
e_{S} \big( F_{[m,0]}|M_{0} \big)
\prod_{j=1}^{q} \big[ \gamma_{j} - \kappa_{S}(\gamma_{j}) \big]. 
\label{SC-Eq-4b}
\end{align}
\end{remark*}

\subsection{Orbifold multiplicities}

As remarked before $M_{\Gamma}$ is an orbifold in general. Next we will compute the orbifold multiplicities  of its fixed point components as suborbifolds.

	\begin{prop}\label{SC-Prop-A}
If $M$ is compact and $K$ acts effectively on it then $\textnormal{mult}(M_{\Gamma})=1$ and $\textnormal{mult}(M_{\Gamma,0})=1$. Moreover, let $\tau_{1},\ldots, \tau_{q}$ be a $\mathbb{Z}$-bases of $\mathfrak{k}^{*}_{\mathbb{Z}}$. Set $\delta_{\gamma} = |\det([\gamma_{ij}]_{i,j=1}^{q})|$, where $\gamma_{i} = \sum_{j=1}^{q}\gamma_{ij}\tau_{j}$. Then $\textnormal{mult}(F'_{1}) = |K'| = \delta_{\gamma}$.
	\end{prop}

\begin{proof}
Denote $M_{H} = \{ m\in M \ |\ K_{m} = H \}$ and $M_{(H)} = \{ m\in H \ |\ K_{m}\textnormal{ is conjugate to $H$ in $K$} \}$. $M_{(H)}$ is a submanifold for all $H\subset K$ and $M_{H}=M_{(H)}$ since $K$ is abelian. 
The compactness of $M$ implies that only finite number of subgroups of $K$ may appear as isotropy subgroups, hence we have a finite stratification $M=\biguplus_{H} M_{(H)}$. The inclusion defines a partial order on the set of isotropy groups. Moreover, it has a unique minimal element $H_{min}$. The corresponding stratum $M_{(H_{min})}$ is open and dense submanifold of $M$, called the principal stratum. Remark that $H_{min}$ being the unique minimal isotropy group, it acts trivially on $M$. (cf. \cite{GKG})
Therefore, if the action of $K$ is effective then $H_{min}$ is the trivial subgroup and $K$ acts freely on the principal stratum. In particular $K$ acts freely on $M_{int}\cap M_{(H_{min})}\neq\emptyset$ implying that $M_{\Gamma,int} \simeq M_{int}/K'$ is an effective orbifold. $M_{\Gamma,int} \subset M_{\Gamma}$ is open thus $M_{\Gamma}$ is also effective and therefore $\textnormal{mult}(M_{\Gamma}) = 1$.

There is an open neighborhood $U$ of $0\in\mathfrak{k}^{*}$ such that $\mu^{-1}(U)$ is $K$-equivariantly diffeomorphic to $\mu^{-1}(0)\times U$. Since $\mu^{-1}(U)$ is open and $M_{(H_{min})}$ is dense we have that $\mu^{-1}(U)\cap M_{(H_{min})}\neq\emptyset$ implying that $\mu^{-1}(0)\cap M_{(H_{min})}\neq\emptyset$ and consequently $M_{0}$ is also an effective orbifold. 
We emphasis that this does not mean that the orbifold multiplicities of $T$-fixed point components $H_{0}$ in $M_{\Gamma,0}$ are one.

The fixed point components $H_{1}$ of $M_{\Gamma,int}$ as orbifolds are isomorphic to $F_{1}/K' \subset M_{int}/K'$ for some fixed point component $F_{1}$ of $M_{int}$, thus $\textnormal{mult}(H_{1}) = \textnormal{mult}(F_{1}/K')= |K'|$. 
The action of $K\simeq \mathbb{R}^{q}/\mathbb{Z}^{q}$ on $\mathbb{C}^{q}$ gives a homomorphism $\mathbb{R}^{q}/\mathbb{Z}^{q} \to \mathbb{R}^{q}/\mathbb{Z}^{q}$, $(u_{1}, \ldots, u_{q}) \mapsto u_{1}\gamma_{1} + \ldots + u_{q}\gamma_{q}$, which has degree $|\det([\gamma_{ij}]_{i,j=1}^{q})|=K'$.
\end{proof}

\subsection{Atiyah-Bott-Berline-Vergne theorem on $M_{\Sigma}$}

The orbifold version of the Atiyah-Bott-Berline-Vergne localization is as follows \cite{Me}. 
	\begin{thm*}\label{SC-Thm-A}
Let $X$ be an orbifold with a $T$-action. For $\beta \in H_{T}(X)$ we have
$$\frac{1}{\textnormal{mult}(X)}\int_{X}\beta 
= \sum_{F\subset X^{T}}\frac{1}{\textnormal{mult}(F)}\int_{F}\frac{i^{*}_{F}\beta}{e_{T}\N(F|X)},$$
where $\textnormal{mult}(X)$ and $\textnormal{mult}(F)$ are the orbifold multiplicities of $X$ and $F$, and $\mathcal{N}(F|X)$ is the equivariant normal orbibundle of $F$ in $X$. 
	\end{thm*}

We apply the above theorem for $M_{\Gamma}$. It will summarize the results of this section.

	\begin{thm}\label{Thm-SC-B}
Suppose that $T=K\times S$ and $M_{\Sigma}$ compact. For any $\beta\in H_{T}(M)$ we have
\begin{align*}
\int_{M_{\Gamma}} \Delta(\beta e^{\omega - \mu_{T}}) 
={}& 
\oint_{M_{0}} \frac{ \kappa_{S}(\beta e^{\omega - \mu_{T}}) }{ \prod_{j=1}^{q} [\gamma_{j} - \kappa_{S}(\gamma_{j}) ] }
+ 
\sum_{\substack{F \subset M^{T} \\ \mu_{K}(F)\in \Gamma}} \frac{1}{\delta_{\gamma}} \int_{F} \frac{i^{*}_{F}(\beta e^{\omega - \mu_{T}})}{e_{T}\mathcal{N}(F|M)} 
\\&{}
+ \sum_{\substack{H \subset M_{\Gamma}^{T} \\ \mu_{K}(H)\in \partial\Gamma\setminus\{0\} } } \frac{1}{\textnormal{mult}(H)}\int_{H} \frac{i^{*}_{H}\Delta(\beta e^{\omega - \mu_{T}})}{e_{T}\mathcal{N}(H|M_{\Gamma})},
\end{align*}
where the number $\delta_{\gamma}$ is defined in Proposition \ref{SC-Prop-A} and we used notation
$$
\oint_{M_{0}} \frac{\kappa_{S}(\beta e^{\omega - \mu_{T}}) }{\prod_{j=1}^{q}[\gamma_{j} - \kappa_{S}(\gamma_{j})]} 
:= 
\sum_{F_{0}\subset M_{0}^{S}} \frac{1}{\textnormal{mult}(F_{0})} \int_{F_{0}} \frac{i^{*}_{F_{0}}\kappa_{S}(\beta e^{\omega - \mu_{T}}) }{e_{S}(F_{0}|M_{0})\prod_{j=1}^{q} \big[\gamma_{j} - i^{*}_{F_{0}}\kappa_{S}(\gamma_{j}) \big]},
$$
 $\textnormal{mult}(F_{0})$ is the multiplicity of $F_{0}$ as suborbifold of $M_{0}$.
	\end{thm}

	\begin{corollary*}\label{Cor-SC-A}
Let $K$ be one dimensional, then $\Gamma=\mathbb{R}_{\geq0}\gamma$ with $\gamma\in \mathfrak{k}_{\mathbb{Z}}\simeq \mathbb{Z}$. If $M_{\Gamma}$ is compact then for all $\beta\in H_{T}(M)$ we have
$$
\int_{M_{\Gamma}}\Delta(\beta e^{\omega - \mu_{T}}) 
=
\oint_{M_{0}} \frac{\kappa_{T/K}(\beta e^{\omega - \mu_{T}})}{\gamma - \kappa_{T/K}(\gamma)} 
+ 
\frac{1}{|\gamma|}\sum_{\substack{F\subset M^{T} \\ \mu_{K}(F)\in\Gamma }} \int_{F} \frac{i^{*}_{F}(\beta e^{\omega - \mu_{T}})}{e_{T}\mathcal{N}(F|M)}.
$$
	\end{corollary*}


\section{Equivariant Jeffrey-Kirwan theorem}

Motivated by \cite{PW} and \cite{HP2} we define integrations of equivariant cohomology classes on non-compact manifolds formaly by the Atiyah-Bott-Berline-Vergne localization formula provided that the fixed point locus is compact. 
By the same symplectic cut technique as in \cite{JKo} we prove our main result: an equivariant analog of the Jeffrey-Kirwan theorem for non-compact symplectic quotients under assumption the our spaces admit an auxiliary proper and bounded below moment map (cf. \cite{PW}, \cite{HP1}, \cite{Martens}). First we prove an equivariant version of Jeffrey-Kirwan theorem for compact torus quotients. The way we defined the polarization and equivariant residue it makes the theorem more general than the straightforward generalization of the ordinary Jeffrey-Kirwan theorem (cf. \cite{Martens} Theorem 3). For non-compact spaces we first apply symplectic cut with respect to the proper, bounded below moment map, making them compact. Then we apply our theorem for particular polarization and we show that the new fixed point loci introduced by the cut will not contribute.  We use Martin's method \cite{Martin} to proceed from torus quotient to non-abelian quotients. We conclude the section with hyperK\"ahler versions of the main result.

Let $G$ be a compact Lie group  and let $S$ be a torus of rank $q$. Let $(M,\omega)$ be a symplectic (non-compact) manifold with  Hamiltonian $G\times S$-action and denote by $\mu_{G\times S}:M \to \mathfrak{g}^{*}\times \mathfrak{s}^{*}$ its moment map.
 For any subgroup $H\subset G\times S$ the corresponding moment map is $\mu_{H} = (\mathfrak{g}^{*}\times \mathfrak{s}^{*}\to \mathfrak{h}^{*})\circ \mu_{G\times S}$. In particular, $\mu_{G}$ and $\mu_{S}$ denote the $G$- and $S$-moment maps.

We assume that there is an one dimensional torus $K$ in the center of $G\times S$  with non-surjective proper moment map $\mu_{K}$.  For any $\gamma \in \mathfrak{k}^{*}_{\mathbb{Z}}$ we can write $\mu_{K}=\varphi\cdot \gamma$ with $\varphi:M \to \mathbb{R}$. By \cite{PW} we have either $\textnormal{Im}\,\varphi = (-\infty, \eta]$ or $\textnormal{Im}\,\varphi = [\eta,+\infty)$ for some $\eta\in \mathbb{R}$. In the first case $\varphi$ is proper and bounded above, while in the second it is proper and bounded below.
\vskip1ex
\noindent {\it Proper and bounded below assumption.} We choose $\gamma$ such that the corresponding $\varphi$ to be proper and bounded below. 
\vskip1ex

Let $T$ be a maximal torus of $G$ of rank $r$ such that $K\subset T\times S$. 
Moreover, suppose that  $0\in (\mathfrak{g}^{*})^{G}$ is a regular value of $\mu_{G}$. Let $M/\!\!/G = \mu^{-1}_{G}(0)/G$ be the symplectic quotient and we assume that it is non-compact.

	\begin{remark*}
	\begin{enumerate}
	
\item Properness of $\mu_{K}$ implies that $\mu_{T\times K}$ and $\mu_{T\times S}$ are also proper.

\item $K\nsubseteq G$ otherwise $\mu_{G}$ would be proper and the symplectic quotient $M/\!\!/G$ would be compact.

\item For any $F\subset M^{T\times S}$ fixed point component $\mu_{K}(F)$ is a point, therefore $F$ is compact since $\mu_{K}$ is proper.
	\end{enumerate}
	\end{remark*}
\vskip1ex
\noindent {\it Compactness assumption.} We suppose that $M^{T\times S}$ is compact or equivalently, there are finite number of fixed point components $F\subset M^{T\times S}$. 	
\vskip1ex
 We recall the following result (\cite{MS} Proposition 5.6.)
	\begin{prop*}\label{EJK-Prop-1}
If $M$ is a compact $T\times S$-Hamiltonian manifold then $M^{T\times S}$ is non-empty.
	\end{prop*}

We also recall the following proposition from \cite{PW}.

	\begin{prop*}
Let $T$ be a torus and $K\subset T$ be an one dimensional subtorus. If $M$ is a $T$-Hamiltonian manifold such that the $K$-moment map $\varphi:M\to \mathbb{R}$ proper and bounded below then	$M^{T}$ is non-empty. 
	\end{prop*}

\begin{definition}
Let $M$ be a non-compact manifold with a $G$-action. Let $T\subset G$ be its maximal torus. For $\alpha \in H_{G}(M)$ we define
$$
\oint_{M} \alpha := \sum_{F\subset M^{T}} \int_{M}\frac{i^{*}_{F}\alpha}{ e_{T}\mathcal{N}(F|M)}, 
$$
under assumption that $M^{T}$ is compact (cf. \cite{HP2}). 
\end{definition}

	\begin{prop}\label{Prop-JK-2}
If $M^{T\times S}$ compact then $(M/\!\!/G)^{S}$ is also compact.	
	\end{prop}
\begin{proof} Since $(M/\!\!/G)^{S} \subset (M/\!\!/G)^{K}$ it suffices to show that $(M/\!\!/G)^{K}$ is compact.
Every connected component of $M^{K}$ is compact since $\varphi$ is proper. Moreover, $M^{K}$ has finitely many connected components because each of them contains a connected component of $M^{T\times S}$ and $M^{T\times S}$ has finitely many component by assumption. In particular, $M^{T\times K}\subset M^{K}$ is also compact.

Denote $\pi:\mu^{-1}_{G}(0) \to M/\!\!/G$ the quotient map. If $m\in (M/\!\!/G)^{K}$ and $p\in \pi^{-1}(m)\subset M$ then the isotropy group $(G\times K)_{p}$ is one dimensional. Its unit component $(G\times K)^{\circ}_{p}$ is an one dimensional torus. Recall that isotropy groups in the same orbit are conjugate. Thus, there is $q\in (G\times K)\cdot p$ such that $(G\times K)^{\circ}_{q} \subset T\times K$.  We remark that $(G\times K)^{\circ}_{q}\nsubseteq T$ and 
$\pi^{-1}(m) = G\cdot p \subset G\cdot M^{(G\times K)^{\circ}_{q}}.$ 
Let $\mathcal{T}$ be the set of all one dimensional subtori $T'\subset T\times K$ with properties
{\renewcommand{\labelenumi}{(\alph{enumi})}
\begin{enumerate}
\item $T'\nsubseteq T$ and
\item there is $F'\subset M^{T'}$ connected component such that for all $T''$ torus with $T'\subset T''\subset T\times K$, $T''\nsubseteq T$ we have that $(F')^{T''}$ is strictly smaller than $F'$.
\end{enumerate}}
	
The set $\mathcal{T}$ is finite. Indeed, let $F'$ be as in $(b)$. $F'$ contains a $T\times K$-fixed point. Choosing $T\times K$-invariant compatible triple we can consider $T\times K$-weights $\alpha_{i}$ on $T_{x}M$ for all $T\times K$-fixed points $x$. Recall that we have the same weight on $T_{x}M$ for all $x$ in a connected component of $M^{T\times K}$. The Lie algebra $\mathfrak{t}'$ must be of form $\cap_{\textnormal{some }i}\ker \alpha_{i}$, thus there is finite number of possible Lie subalgebras of $\mathfrak{t}\oplus \mathfrak{k}$ as $\mathfrak{t}'$.

We have inclusions of closed subsets
$$
\pi^{-1}\left((M/\!\!/G)^{K}\right) \subset G\cdot \left( \bigcup_{T'\in \mathcal{T}} M^{T'}\cap \mu^{-1}_{G}(0) \right) \subset G\cdot \left( \bigcup_{T'\in\mathcal{T}} M^{T'} \cap \mu^{-1}_{T}(0) \right).
$$
We conclude our proof by showing that $M^{T'} \cap \mu^{-1}_{T}(0)$ is compact for any $T'\in \mathcal{T}$. Let $F'\subset M^{T'}$ be a connected component. Recall that $(\mu_{T\times K})(F')$ lies in an affine hyperplane $\mathcal{H}$ of $\mathfrak{t}^{*}\times \mathfrak{k}^{*}$, where $\mathcal{H}$ is the inverse image of the point $\mu_{T'}(F')$ under $\mathfrak{t}^{*}\times  \mathfrak{k}^{*} \to (\mathfrak{t'})^{*}$. $\mathcal{H}\cap (0\times \mathfrak{k}^{*})$ is finite since $T'\nsubseteq T$. Finally,
$$
F'\cap \mu^{-1}_{T}(0)\ \subset\ F'\cap \mu_{T\times K}^{-1}(0\times \mathfrak{k}^ {*})\ \subset\ \mu_{T\times K}^{-1}(\mathcal{H}\cap (0\times \mathfrak{k}^{*}))
$$
and latter set is compact because $\mu_{T\times K}$ is proper. $M^{T'}$ has finitely many connected components since each of them contains a connected component of $M^{T\times S}$, hence $M^{T'}\cap \mu_{T}^{-1}(0)$ is compact.
	\end{proof}

The invariant function $\varphi:M\to \mathbb{R}$ descends to $M/\!\!/G$ and we denote it by $\varphi'$.
Consider the projection  $\textnormal{pr}_{\mathfrak{k}^{*}}:\mathfrak{t}^{*}\times\mathfrak{s}^{*}\to \mathfrak{k}^{*} = \mathbb{R}\gamma$ and define 
\begin{equation}\label{Eq-JK-14}
\pi:\mathfrak{t}^{*}\times \mathfrak{s}^{*} \to \mathbb{R},
\quad
\textnormal{pr}_{\mathfrak{k}^{*}}(x) = \pi(x)\cdot \gamma.
\end{equation}
Let $T' \subset T\times S$ be a subtorus and let $N \subset M^{T'}$ be a fixed point component. Recall that $\mu_{T\times S}(N)$ is a wall of $\mu_{T\times S}(M)$ and it lies on the affine subspace $W_{\mu_{T\times S}(N)} = \mu_{T\times S}(F) + (\mathfrak{t} \oplus \mathfrak{s}/\mathfrak{t}' )^{*}$ of $\mathfrak{t}^{*}\times \mathfrak{s}^{*}$, where $F\subset N^{T\times S}$ and we identify $(\mathfrak{t} \oplus \mathfrak{s}/\mathfrak{t}' )^{*}$ with the subspace $\{ \tau \in \mathfrak{t}^{*} \times \mathfrak{s}^{*}  \,|\, \tau(\mathfrak{t}')=0 \}$. Moreover, $(\mathfrak{t} \oplus \mathfrak{s}/\mathfrak{t}' )^{*}$ is spanned by the isotropy $T\times S$-weight vectors of $N$ at $F$.

	\begin{lemma}\label{Lem-JK-1}
There exists regular value $\varepsilon\in \mathbb{R}$ of $\varphi$ such that
\begin{enumerate}
\item[(a)] For all $\mu_{T\times S}(N)$ wall of $\mu_{T\times S}(M)$ such that $W_{\mu_{T\times S}(N)} \cap 0\times \mathfrak{s}^{*} = \{p\}$ we have $\pi(p)<\varepsilon$. 

\item[(b)] $M^{T\times S} \subset \varphi^{-1}(-\infty,\varepsilon]$.

\item[(c)] $\varepsilon$ is a regular value of $\varphi':M/\!\!/G \to \mathbb{R}$. 

\item[(d)] $(M/\!\!/G)^{S} \subset (\varphi')^{-1}(-\infty,\varepsilon]$.
\end{enumerate}
	\end{lemma}

	\begin{proof}
$\varphi(M^{T\times S})$ and $\varphi'((M/\!\!/G)^{S})$ are finite by compactness assumption and Proposition \ref{Prop-JK-2}. It follows that $\varphi$ and $\varphi'$ have only finitely many critical values. Moreover, $N$ contains a fixed point component of $M^{T\times S}$ for any wall $\mu_{T\times S}(N)$ of $\mu_{T\times S}(M)$. Thus we have only finitely many $W_{\mu_{T\times S}(N)}$ affine subspaces, hence it yields finitely many values $\pi(p)$ with $W_{\mu_{T\times S}(N)} \cap 0\times \mathfrak{s}^{*} = \{p\}$. Any value $\varepsilon$ bigger than all above will satisfy the required properties.
	\end{proof}

Let $\varepsilon$ be as in Lemma \ref{Lem-JK-1}. Consider the symplectic cut $X:=M_{\Gamma}$, where $\Gamma = \varepsilon - \mathbb{R}\gamma \in \mathfrak{k}^{*}$. Remark that as set $X = \varphi^{-1}(-\infty,\varepsilon)\uplus \varphi^{-1}(\varepsilon)/K$.
	Since $\varphi$ is proper and bounded below $X$ is a compact Hamiltonian $G\times S$-manifold. Let $\Delta_{X}:H_{T\times S}(M) \to H_{T\times S}(X)$ be the cut homomorphism. 
Denote by $\omega_{X}$ and $\phi_{G\times S}$, respectively the induced symplectic form and $G\times S$-moment map on $X$. We also use notation $\phi_{H}$ for moment maps corresponding to subgroups $H\subset G\times S$.
	\begin{lemma*}
$0 \in \mathfrak{g}^{*}$ is a regular value of $\phi_{G}:X\to \mathfrak{g}^{*}$.	
	\end{lemma*}
	\begin{proof}
It is enough to check that $0 \in \mathfrak{g}^{*}$ is a regular value of $\phi_{G}$ on $\varphi^{-1}(\varepsilon)/K$ which is equivalent to $G$ acting locally freely on $\phi^{-1}_{G}(0)\cap \varphi^{-1}(\varepsilon)/K $. This latter holds if and only if $G\times K$ acts locally freely on $\varphi^{-1}(\varepsilon) \cap \mu^{-1}_{G}(0)$	, i.e. when $(0,\, \varepsilon) \in  \mathfrak{g}^{*}\times \mathbb{R}$ is a regular value of $\mu_{G}\times \varphi:M \to \mathfrak{g}^{*}\times \mathbb{R}$. By a similar argument this holds exactly when $\varepsilon$ is a regular value of $\varphi':M/\!\!/G \to \mathbb{R}$, which holds by Lemma \ref{Lem-JK-1}(c).
	\end{proof}

\subsection{Abelian version}
First we consider the $G=T$ abelian case. We have the following generalization of the compact Jeffrey-Kirwan theorem.

	\begin{thm}\label{Thm-JK-1}
Let $X$ be a compact Hamiltonian $T\times S$-manifold with moment map $\phi_{T \times S} = \phi_{T}\times \phi_{S}:X \to \mathfrak{t}^{*} \times \mathfrak{s}^{*}$. Suppose that $0\in \mathfrak{t}^{*}$ is a regular value of $\phi_{T}$ and let $X/\!\!/T = \phi_{T}^{-1}(0)/T$ be the symplectic quotient. 
Then for any $\alpha_{X} \in H_{T\times S}(X)$ and any generic ordered bases $x$ with respect to $\oint_{X} \alpha_{X} e^{\omega_{X} - \phi_{T} - \phi_{S}}$ we have
$$
\oint_{X/\!\!/T} \kappa_{S}(\alpha_{X} e^{\omega_{X} - \phi_{T} -\phi_{S}}) = \textnormal{EqRes}_{x}\, \left( \frac{1}{vol(T)} \oint_{X} \alpha_{X} e^{\omega_{X} - \phi_{T} - \phi_{S}} \right)(x),
$$
where $\kappa_{S}:H_{T\times S}(X) \to H_{S}(X/\!\!/T)$ is the $S$-equivariant Kirwan map and $vol(T)$ is computed with scalar product used for $\textnormal{EqRes}$.
	\end{thm}

	\begin{proof}
The proof goes the same way as in \cite{JKo}, only the very end is different.
Let $y = \{ t_{1}, \ldots, t_{r}, s_{1}, \ldots, s_{q}\}$ be a generic bases such that $\{t_{1},\ldots,t_{r}\}$ and $\{s_{1},\ldots,s_{q}\}$ are bases of $\mathfrak{t}^{*}$ and $\mathfrak{s}^{*}$, respectively. Denote $\{t^{1}, \ldots, t^{r}, s^{1}, \ldots, s^{q}\} \subset \mathfrak{t} \oplus \mathfrak{s}$ its dual bases. 
Introduce notations $(t^{1})^{<0} = \{ a_{1}t_{1}+ \ldots + a_{r+q}s_{q} \in \mathfrak{t}^{*} \times \mathfrak{s}^{*} \,|\, a_{1}<0 \}$. We define similarly the sets $(t^{1})^{\leq0}$, $(t^{1})^{\geq0}$ and $(t^{1})^{>0}$.
Since $y$ is generic we have $\langle t_{2}, \ldots, t_{r}, s_{1}, \ldots, s_{q} \rangle \cap \phi_{T \times S}(X^{T \times S}) = \emptyset$.

Let $\Gamma = Cone(\gamma_{1},\ldots,\gamma_{r}) \subset \mathfrak{t}^{*}$ be a rational simplicial cone such that
\begin{enumerate}

\item[($\Gamma1$)] $\gamma_{1}, \ldots, \gamma_{r} \in (t^{1})^{<0} \cap \mathfrak{t}^{*}_{\mathbb{Z}}$

\item[($\Gamma2$)] $\Gamma$ intersects any wall of $\phi_{T}(X)$ transversally, 

\item[($\Gamma3$)] $(t^{1})^{<0} \cap \phi_{T}(X^{T}) \subset \Gamma$, hence $(t^{1})^{<0} \cap \phi_{T \times S}(X^{T \times S}) \subset \chi^{-1}(\Gamma)$, where $\chi:\mathfrak{t}^{*} \times \mathfrak{s}^{*} \to \mathfrak{t}^{*}$ projection.

\end{enumerate}
We make symplectic cut with respect to $\Gamma$ and let  $X_{\Gamma}$. We also denote the $T$- and $S$-moment maps by $\phi_{T}$ and $\phi_{S}$, respectively.

We apply Theorem \ref{Thm-SC-B} on $X_{\Gamma}$:
\begin{align}
\oint_{X_{\Gamma}} \Delta (\alpha_{X} e^{\omega_{X} - \phi_{T \times  S}})
={}&
\oint_{X_{0}} \frac{\kappa_{S}(\alpha_{X} e^{\omega_{X} - \phi_{T \times S}}) }{\prod_{j=1}^{r}(\gamma_{j} - \kappa_{S}(\gamma_{j}))}
\label{Eq-JK-1a}
\\{}&
+
\sum_{\substack{F\subset X^{T \times S} \\ \phi_{T}(F)\in \Gamma}} \frac{1}{\delta_{\gamma}} \int_{F} \frac{i^{*}_{F}(\alpha_{X} e^{\omega_{X} - \phi_{T \times S}})}{e_{T\times S}\mathcal{N}(F|X)} 
\label{Eq-JK-1b}
\\&{}
+ 
\sum_{\substack{H \subset X_{\Gamma}^{T \times S} \\ \phi_{T}(H)\in \partial\Gamma\setminus\{0\} } } \frac{1}{\textnormal{mult}(H)}\int_{H} \frac{i^{*}_{H}\Delta(\alpha_{X} e^{\omega_{X} - \phi_{T \times S}})}{e_{T\times S}\mathcal{N}(H | X_{\Gamma})},
\label{Eq-JK-1c}
\end{align}
where $\Delta: H_{T\times S}(X) \to H_{T\times S}(X_{\Gamma})$ cut homomorphism.
Denote the left hand side of (\ref{Eq-JK-1a}) by $I$ and the right hand side by $I_{red}$. Moreover, denote the sum in (\ref{Eq-JK-1b}) and (\ref{Eq-JK-1c}) by $I_{old}$ and $I_{new}$, respectively. The short version of the above equality is $I= I_{red} + I_{old} + I_{new}$. 

Let $\rho \in \textnormal{int}\, \Gamma$ be generic in a small neighborhood of $0$ such that 
\begin{equation}\label{Eq-JK-2}
\langle \rho, -t^{1} \rangle < \langle \phi_{T} (F'), -t^{1} \rangle 
\end{equation}
for all $F' \subset X_{\Gamma}^{T \times S}$ with $\phi_{T}(F') \neq 0$.
Denote $z$ the ordered bases $\{-t_{1},\ldots,-t_{r},-s_{1},\ldots,-s_{q}\}$.

\begin{lemma}\label{Lem-JK-2}
$
\textnormal{EqRes}_{z}I_{old}(z) e^{\rho(z)} + \textnormal{EqRes}_{z} I_{new}(z) e^{\rho(z)} = 0.
$
\end{lemma}

\begin{proof}
Any summand of $I_{old} e^{\rho} + I_{new} e^{\rho}$ is of form $\frac{P e^{\lambda}}{\prod_{i}\alpha_{i}}$ with $\lambda = -\phi_{T \times S}(F') + \rho$ for some $F' \subset X_{\Gamma}^{T \times S}$ such that $\phi_{T}(F') \in \Gamma\setminus \{0\}$. From (\ref{Eq-JK-2}) and $(\Gamma1)$ follows that $\lambda \in (t^{1})^{>0}$, thus it is not polarized with respect to $z$ and by Corollary \ref{Cor-I-3} we have  $\textnormal{EqRes}_{z}\frac{P(z)e^{\lambda(z)}}{\prod_{i} \alpha_{i}(z)}=0$.
\end{proof}

\begin{lemma}\label{Lem-JK-3}
\begin{enumerate}

\item[(a)] $\textnormal{EqRes}_{y} I_{red}(y) e^{\rho(y)} = 0$. 

\item[(b)] $\textnormal{EqRes}_{y} I_{old}(y) e^{\rho(y)} = \frac{1}{\delta_{\gamma}} \textnormal{EqRes}_{y} \big( \oint_{X} \alpha_{X} e^{\omega_{X} - \phi_{T \times S} + \rho} \big)(y)$.

\end{enumerate}
\end{lemma}

\begin{proof} The proof goes similarly as for the previous lemma.
\begin{enumerate}

\item[(a)] Any summand of $I_{red}e^{\rho}$ is of form $\frac{P e^{\rho}}{\prod_{i} \alpha_{i}}$ and recall that $\rho \in (t^{1})^{<0}$ by $(\Gamma1)$. Hence $\rho$ is not polarized with respect to $y$ and from Corollary \ref{Cor-I-3} follows the first part of the lemma.

\item[(b)] If $F \subset X^{T \times S}$ such that $\langle \phi_{T}(F), t^{1} \rangle >0$ then its contribution $\int_{F} \frac{ i^{*}_{F}(\alpha_{X} e^{\omega_{X} - \phi_{T \times S}}) }{ e_{T \times S} \mathcal{N}(F | X) } = \frac{P e^{\lambda}}{\prod_{i} \alpha_{i}}$ with $\lambda = - \phi_{T \times S}(F) \in (t^{1})^{<0}$. Thus $(\Gamma1)$ and  $\rho \in \Gamma$ implies that $\lambda + \rho \in (t^{1})^{<0}$, which is not polarized with respect to $y$. By Corollary \ref{Cor-I-3} and $(\Gamma3)$ we get
$$
\textnormal{EqRes}_{y} \left( \frac{1}{\delta_{\gamma}} \oint_{X}\alpha_{X} e^{\omega_{X} - \phi_{T \times S} + \rho} \right)(y) 
= 
\textnormal{EqRes}_{y} \bigg( \sum_{ \substack{F \subset X^{T \times S}\\ \phi_{T}(F)\in \Gamma }} \frac{1}{\delta_{\gamma}} \int_{F} \frac{i^{*}_{F}( \alpha_{X} e^{\omega_{X} - \phi_{T \times S} + \rho})}{e_{T \times S} \mathcal{N}(F|X)} \bigg)(y).
$$

\end{enumerate}
\end{proof}

\begin{lemma}\label{Lem-JK-4}
$\textnormal{EqRes}_{y} I_{new}(y) e^{\rho(y)} = 0$.
\end{lemma}

\begin{proof}
Denote $U=T \times S$. Any summand of $I_{new}$ is of form $\frac{P e^{\lambda}}{\prod_{i\in I'} \alpha'_{i} \prod_{j \in J_{z}} \gamma'_{j}}$, where $\lambda = -\phi_{T \times S}(H)$ for some $H = F_{[x,z]} \subset X^{T \times S}_{\Gamma}$ with $\phi_{T}(H) \in \partial \Gamma \setminus \{0\}$ and $\alpha'_{i} \in \mathfrak{u}^{*}_{x}$, $\gamma'_{j} \in \mathfrak{t}^{*}_{z}$ is the projection of $\gamma_{j}$ to $\mathfrak{t}^{*}_{z}$ along $\mathfrak{u}^{*}_{x}$ for all $j\in J_{z}$. Recall that $\mathfrak{u}^{*} = \mathfrak{u}^{*}_{x} \oplus \mathfrak{t}^{*}_{z}$. Denote $\widetilde{\alpha}'_{i}$ and $\widetilde{\gamma}'_{j}$ the polarizations of $\alpha'_{i}$ and $\gamma'_{j}$ with respect to $y$. We will show that $\lambda + \rho \notin Cone(\widetilde{\alpha}'_{i},\, \widetilde{\gamma}'_{j} \,|\, i\in I',\,j\in J_{z})$ 
or equivalently
\begin{equation}\label{Eq-JK-3}
0 \notin \phi_{T\times S}(H) - \rho + Cone(\widetilde{\alpha}'_{i},\, \widetilde{\gamma}'_{j} \,|\, i\in I',\,j\in J_{z})
\end{equation}
and the lemma will follow from Corollary \ref{Cor-I-2}. We have $\rho \in \chi^{-1}(\textnormal{int}\, \Gamma)$, $\phi_{T\times S}(H) \in \mathfrak{u}^{*}_{x}$ and $\chi^{-1}(\textnormal{int}\,\Gamma) \subset \mathfrak{u}^{*}_{x} + \textnormal{int}\, Cone(\gamma'_{j}  \,|\, j\in J_{z})$, hence
$$
0 \in \phi_{T \times S}(H) -\rho + \mathfrak{u}^{*}_{x} + \textnormal{int}\,Cone(\gamma'_{j} \,|\, j\in J_{z}).
$$
Moreover,
$Cone(\widetilde{\alpha}'_{i},\,\widetilde{\gamma}'_{j}  \,|\, i\in I', j\in J_{z})\subset u^{*}_{x} + Cone(\widetilde{\gamma}'_{j}  \,|\, Ęj\in J_{z})$, hence it is enough to show that $Cone(\widetilde{\gamma}'_{j}  \,|\, Ęj\in J_{z}) \cap\textnormal{int}\, Cone(\gamma'_{j}  \,|\, Ęj\in J_{z})  = \emptyset$. These cones are simplicial, thus enough to show that $\widetilde{\gamma}'_{j} = - \gamma'_{j}$ for some $j \in J_{z}$. We consider the functional $\tau:\phi_{T\times S}(F_{x}) \to \mathbb{R}$, $p \mapsto \langle p, t^{1} \rangle$. Since the intersection of the vector space $\mathfrak{u}^{*}_{x}$ and the convex polytope $\phi_{T \times S}(F_{x})$ is transversal in $\phi_{T \times S}(H)$, hence $\phi_{T \times S}(H)$ is in the interior of $\phi_{T \times S}(F_{x})$. By $(\Gamma3)$ and convexity of $\phi_{T \times S}(F_{x})$ all minimal points of $\tau$ lie in $\textnormal{int}\, \Gamma$, thus $\phi_{T \times S}(H)$ cannot be a minimal point of $\tau|_{\phi_{T \times S}(F_{x})\cap \Gamma}$ and since $\phi_{T \times S}(F_{x}) \cap \Gamma \subset \phi_{T \times S}(H) + Cone(\gamma'_{j} \,|\, j\in J_{z})$ there must be $j \in J_{z}$ such that $\langle \gamma'_{j},t^{1} \rangle < 0$, implying $\widetilde{\gamma}'_{j} = - \gamma'_{j}$.
\end{proof}

\begin{lemma}\label{Lem-JK-5}
$\displaystyle\lim_{\varepsilon \to 0} \textnormal{EqRes}_{z} I_{red}(z) e^{\varepsilon \rho(z)} = \frac{vol(T)}{\delta_{\gamma}} \oint_{X_{0}} \kappa_{S}(\alpha_{X} e^{\omega_{X} - \phi_{T \times S}})$.
\end{lemma}

\begin{proof}
If $D \subset X_{0}^{S}$ is a fixed point component, then the corresponding contribution to $I_{red}$ equals 
\begin{multline}\label{Eq-JK-4}
\frac{1}{\textnormal{mult}(D)} \int_{D} \frac{i^{*}_{D}\kappa_{S}(\alpha_{X} e^{\omega_{X} -\phi_{T\times S}})}{e_{S}\mathcal{N}(D | X_{0}) \prod_{j=1}^{r}(\gamma_{j} - i^{*}_{D}\kappa_{S}(\gamma_{j}))}
\\
=
\frac{1}{\textnormal{mult}(D)}
\int_{D} \frac{i^{*}_{D}\kappa_{S}(\alpha_{X} e^{\omega_{X} -\phi_{T\times S}})}{e_{S}\mathcal{N}(D | X_{0}) \prod_{j=1}^{r}(\gamma_{j} + \eta_{j} - i^{*}_{D}\kappa(\gamma_{j}))}
\\
=
\frac{1}{\textnormal{mult}(D)}
\sum_{k_{1},\ldots,k_{r}\geq0}
\int_{D} \frac{i^{*}_{D}\kappa_{S}(\alpha_{X} e^{\omega_{X} -\phi_{T\times S}})}{e_{S}\mathcal{N}(D | X_{0})} \prod_{j=1}^{r} \frac{i^{*}_{D}\kappa(\gamma_{j})^{k_{j}}}{(\gamma_{j} + \eta_{j})^{k_{j}+1}}
\\
=
\sum_{k_{1},\ldots,k_{r}\geq0}
\frac{ P_{ k_{1},\ldots,k_{r} } e^{ -\phi_{S}(D)}  }{ \prod_{j=1}( \gamma_{j} + \eta_{j} )^{k_{j}+1} },
\end{multline}
where $\textnormal{mult}(D)$ is the multiplicity of $D$ as suborbifold of $X_{0}$, 
$\eta_{j} \in \mathfrak{s}^{*}$ such that $i^{*}_{D}\kappa_{S}(\gamma_{j}) = i^{*}_{D}\kappa(\gamma_{j}) - \eta_{j}$ and $P_{k_{1},\ldots,k_{r}}$ is a rational fraction in $s$ such that
$$
P_{ k_{1},\ldots k_{r} } e^{- \phi_{S}(D)} 
= 
\frac{1}{\textnormal{mult}(D)}
\int_{D} \frac{i^{*}_{D}\kappa_{S}(\alpha_{X} e^{\omega_{X} -\phi_{T\times S}})}{e_{S}\mathcal{N}(D | X_{0})} \prod_{j=1}^{r} i^{*}_{D}\kappa(\gamma_{j})^{k_{j}}.
$$ 
By Proposition \ref{Prop-I-5} and $(\Gamma 1)$ we have
\begin{equation}\label{Eq-JK-5}
\textnormal{EqRes}_{z} \frac{ P_{ k_{1},\ldots,k_{r} }(s)  e^{ -\phi_{S}(D)(s) + \rho(z)} }{ \prod_{j=1}( \gamma_{j}(z) + \eta_{j}(z) )^{k_{j}+1} }
=
\frac{P_{k_{1},\ldots,k_{r}}(s)  e^{ -\phi_{S}(D)(s) - \sum_{j=1}^{r}\rho_{j}\eta_{j}(s) } }{ \sqrt{\det[(t_{a},t_{b})]_{a,b}} \cdot \big| \det \big( \frac{ \partial \gamma_{j} }{ \partial t_{i} } \big) \big| } 
\cdot 
\frac{ \rho_{1}^{k_{1}} \cdots \rho_{r}^{k_{r}} }{ k_{1}!\cdots k_{r}! },
\end{equation}
where $\rho =  \rho_{1}(\gamma_{1} + \eta_{1}) + \ldots \rho_{r}(\gamma_{r} + \eta_{r}) - \sum_{j=1}^{r} \rho_{j}\eta_{j}$. 
Let $\{ \tau_{1}, \ldots, \tau_{r} \}$ be a bases of $\mathfrak{t}^{*}_{\mathbb{Z}}$ and $\{ \nu_{1}, \ldots, \nu_{r} \}$ be an orthonormal bases of $\mathfrak{t}^{*}$. Then by Remark \ref{Rk-I-5} we have
\begin{equation}\label{Eq-JK-6}
\sqrt{\det[(t_{a},t_{b})]_{a,b}} \cdot \left| \det \left( \frac{ \partial \gamma_{j} }{ \partial t_{i} } \right) \right| 
= 
\left| \det \left( \frac{ \partial \gamma_{j} }{ \partial \nu_{i} } \right) \right| 
= 
\left| \det \left( \frac{ \partial \gamma_{j} }{ \partial \tau_{l} } \right) \right| \cdot \left| \det \left( \frac{ \partial \tau_{l} }{ \partial \nu_{i} } \right) \right| 
= 
\frac{\delta_{\gamma}}{ vol(T)}.
\end{equation}
Remark that $P_{0,\ldots,0}e^{-\phi_{S}(D)} = \frac{1}{\textnormal{mult}(D)} \int_{D} \frac{i^{*}_{D} \kappa_{S}(\alpha_{X} e^{\omega_{X} -\phi_{T \times S} }) }{ e_{S}\mathcal{N}(D|X_{0})}$, thus by equations (\ref{Eq-JK-4}),  (\ref{Eq-JK-5}) and  (\ref{Eq-JK-6}) we have 
\begin{multline*}
\lim_{\varepsilon \to 0} \textnormal{EqRes}_{z} I_{red}(z) e^{\varepsilon\rho(z)} 
= 
\\
\frac{vol(T)}{\delta_{\gamma}}
\lim_{ \varepsilon \to 0 }   \sum_{ k_{1}, \ldots, k_{r} \geq 0 }  
P_{ k_{1}, \ldots, k_{r} }(s) \frac{ (\varepsilon \rho_{1})^{k_{1}} \cdots (\varepsilon \rho_{r})^{k_{r}} }{ k_{1}! \cdots k_{r}! } e^{ -\phi_{S}(D)(s) - \sum_{j=1}^{r} \varepsilon \rho_{j} \eta_{j}(s) }  
\\
=
\frac{vol(T)}{\delta_{\gamma}}
P_{0,\ldots,0} e^{-\phi_{S}(D)} 
=
\frac{vol(T)}{\delta_{\gamma}} \oint_{X_{0}} \kappa_{S}(\alpha_{X} e^{\omega_{X} - \phi_{T \times S}}).
\end{multline*}
\end{proof}

From Proposition \ref{Prop-I-6}(a) we have $\textnormal{EqRes}_{y} I(y) e^{\rho(y)} = \textnormal{EqRes}_{z} I(z) e^{\rho(z)}$ and by Lemmas \ref{Lem-JK-2}, \ref{Lem-JK-3}(a) and \ref{Lem-JK-4} it yields
$\textnormal{EqRes}_{y}I_{old}(y) e^{\rho}(y) = \textnormal{EqRes}_{z} I_{red}(z) e^{\rho(z)}$.   Taking limit as $\rho \to 0$ we get 
$$
\lim_{\varepsilon \to 0} \frac{1}{\delta_{\gamma}} \textnormal{EqRes}_{y} \left( \oint_{X} \alpha_{X} e^{ \omega_{X} - \phi_{T\times S} +\varepsilon \rho } \right)(y) = \lim_{\varepsilon \to 0} \textnormal{EqRes}_{z} I_{red}(z) e^{\varepsilon \rho(z)}
$$
by Lemma \ref{Lem-JK-3}(b), hence we arrive to
$$
\frac{1}{vol(T)} \textnormal{EqRes}_{y} \left( \oint _{X} \alpha_{X} e^{\omega_{X} - \phi_{T \times S}} \right)(y) 
= 
\oint_{X_{0}} \kappa_{S}(\alpha_{X} e^{\omega_{X} - \phi_{T \times S}}),
$$ 
by Lemma \ref{Lem-JK-5} and Proposition \ref{Prop-I-6}(b).
We use again Proposition \ref{Prop-I-6}(a) to get
\begin{align*}
\frac{1}{vol(T)} \textnormal{EqRes}_{x} \left( \oint_{X} \alpha_{X} e^{\omega_{X} - \phi_{T \times S}} \right) (x)
&{}=
\frac{1}{vol(T)} \textnormal{EqRes}_{y} \left( \oint_{X} \alpha_{X} e^{\omega_{X} - \phi_{T \times S}} \right) (y) 
\\&{}
= 
\oint_{X_{0}} \kappa_{S}(\alpha_{X} e^{\omega_{X} - \phi_{T \times S}}).
\end{align*}
	\end{proof}

The abelian version of our main theorem is as follows.

	\begin{thm}\label{Thm-JK-2}
Let $y =\{y_{1}, \ldots, y_{r+q}\}$ be an ordered bases of $\mathfrak{t}^{*} \times \mathfrak{s}^{*}$ such that $\pi(y_{1})>0$ and $y_{2},\ldots,y_{r+q} \in \ker \pi$. Then
\begin{equation}\label{Eq-JK-7}
\oint_{M/\!\!/T} \kappa_{S}(\alpha e^{ \omega - \mu_{ T\times S } }) = \textnormal{EqRes}_{x} \left( \frac{1}{vol(T)} \oint_{M} \alpha e^{ \omega - \mu_{T\times S} } \right)(x),
\end{equation}
where $x$ is a generic bases with respect to $\oint_{M} \alpha e^{\omega - \mu_{T \times S}}$, inducing the same polarization as $y$ on isotropy $T\times S$-weights of $M$.
	\end{thm}

	\begin{proof}
By Lemma \ref{Lem-JK-1}(b) we have $X^{T\times S} = M^{T\times S} \uplus M_{\varepsilon}^{T\times S}$, where $M_{\varepsilon} = \varphi^{-1}(\varepsilon)/K$. Theorem \ref{Thm-JK-1} on $X$ yields
\begin{align}\label{Eq-JK-8}
\oint_{X_{0}} \kappa_{S}\Delta_{X}(\alpha e^{\omega - \mu_{T\times S}}) 
={}& 
\textnormal{EqRes}_{x} \frac{1}{vol(T)} \left( \oint_{X} \Delta_{X}(\alpha e^{\omega - \mu_{T \times S}}) \right)(x)
\notag
\\
={}&
\textnormal{EqRes}_{x} \frac{1}{vol(T)} \left( \oint_{M} \alpha e^{\omega - \mu_{T \times S}} \right)(x) 
\\
&{}+
\textnormal{EqRes}_{x} \frac{1}{vol(T)} \left( \oint_{M_{\varepsilon}} \frac{ \kappa_{T\times S/K}( \alpha e^{\omega - \mu_{T \times S}} ) }{ e_{T\times S}\mathcal{N}(M_{\varepsilon}|X) } \right)(x), 
\notag
\end{align}
where $\kappa_{T\times S/K}:H_{T\times S}(M) \to H_{T\times S/K}(M_{\varepsilon}) \subset H_{T\times S}(M_{\varepsilon})$.  We can write $X_{0}$ as symplectic cut with respect to the cone $\Gamma = \varepsilon - \mathbb{R}\gamma \subset \mathfrak{k}^{*}$ on $M_{0}$, i.e. $X_{0} = (M_{0})_{\Gamma}$ and $X_{0} = (\varphi')^{-1}(-\infty,\varepsilon)\uplus (\varphi')^{-1}(\varepsilon)/K$ as set.
By Lemma \ref{Lem-JK-1}(c) and (d) we have $X_{0}^{T\times S} = M_{0}^{T\times S} \uplus M_{0,\varepsilon}^{T\times S}$, where $M_{0} = M/\!\!/T$ and $M_{0,\varepsilon} = M_{0}/\!\!/_{\varepsilon}K = (\varphi')^{-1}(\varepsilon)/K$. The Atiyah-Bott-Berline-Vergne theorem on $X_{0}$ gives
\begin{equation}\label{Eq-JK-9}
\oint_{X_{0}} \kappa_{S}\Delta_{X}(\alpha e^{\omega - \mu_{T \times S}}) 
= 
\oint_{M_{0}} \kappa_{S}(\alpha e^{\omega - \mu_{T\times S}}) 
+ 
\oint_{M_{0,\varepsilon}} \frac{ i^{*}_{M_{0,\varepsilon}}\kappa_{S}\Delta_{X}(\alpha e^{\omega - \mu_{T \times S}}) }{ e_{S}\mathcal{N}(M_{0,\varepsilon}|X_{0}) },
\end{equation}
and combining with equation (\ref{Eq-JK-8}) we get
\begin{align}\label{Eq-JK-13}
\oint_{M_{0}} \kappa_{S}(\alpha e^{\omega - \mu_{T\times S}})  
={}& 
\textnormal{EqRes}_{x} \frac{1}{vol(T)} \left( \oint_{M} \alpha e^{\omega - \mu_{T \times S}} \right)(x) 
\notag
\\
&{}+
\textnormal{EqRes}_{x} \frac{1}{vol(T)} \left( \oint_{M_{\varepsilon}} \frac{ \kappa_{T\times S/K}( \alpha e^{\omega - \mu_{T \times S}} ) }{ e_{T\times S}\mathcal{N}(M_{\varepsilon}|X) } \right)(x)
\\
&{}-
\oint_{M_{0,\varepsilon}} \frac{ i^{*}_{M_{0,\varepsilon}}\kappa_{S}\Delta_{X}(\alpha e^{\omega - \mu_{T \times S}}) }{ e_{S}\mathcal{N}(M_{0,\varepsilon}|X_{0}) }.
\notag
\end{align}
We conclude the proof by the following lemma.
\end{proof}

	\begin{lemma*}\label{Lem-JK-6}
\begin{equation}\label{Eq-JK-12}
\oint_{M_{\varepsilon,0}} \frac{i^{*}_{0,\varepsilon} \kappa_{S}\Delta_{X}(\alpha e^{\omega - \mu_{T\times S}}) }{ e_{S}\mathcal{N}(M_{0,\varepsilon}|X_{0}) }
=
\textnormal{EqRes}_{x} \frac{1}{vol(T)} \left( \oint_{M_{\varepsilon}} \frac{\kappa_{T\times S/K}(\alpha e^{\omega - \mu_{T\times S}}) }{ e_{T\times S}\mathcal{N}(M_{\varepsilon}|X)}  \right)(x).
\end{equation}
	\end{lemma*}

	\begin{proof}
By symplectic cut construction of $X_{0}$ we have
\begin{equation}\label{Eq-JK-10}
i^{*}_{M_{0,\varepsilon}} \kappa_{S} \Delta_{X} (\alpha e^{\omega - \mu_{T \times S}}) = \kappa_{T\times S/T\times K}(\kappa_{S} (\alpha e^{\omega - \mu_{T\times S}})),
\end{equation}
where $\kappa_{T\times S/T\times K} : H_{S}(M_{0}) = H_{T\times S/T}(M_{0}) \to H_{T\times S/T\times K}( M_{0,\varepsilon} )$. Moreover $M_{\varepsilon}/\!\!/_{0}T$ and $M_{0,\varepsilon} = M_{0}/\!\!/_{\varepsilon}T$ are naturally diffeomorphic to $\mu_{T}^{-1}(0)\cap \varphi^{-1}(\varepsilon)/T\times K$ and the following diagram commutes
$$
\xymatrix{
H_{T\times S}(M) \ar[d]_{\kappa_{S}} \ar[rr]^{\kappa_{T\times S/K}} 
&& 
H_{T\times S/K}(M_{\varepsilon}) \ar[d]^{\kappa'_{T\times S/T\times K}}
\\
H_{S}(M_{0}) \ar[rr]_{\kappa_{T\times S/T\times K}} 
&&
H_{T\times S/T\times K}(M_{0,\varepsilon})
}.
$$
Thus (\ref{Eq-JK-10}) equals to $\kappa'_{T\times S/T\times K}(\kappa_{T \times S/K}(\alpha e^{\omega - \mu_{T\times S}}))$. Furthermore, by Lemma \ref{Lem-SC-3} we have
$$
\kappa'_{T\times S/T\times K}(e_{T\times S}\mathcal{N}(M_{\varepsilon}|X))
=
e_{T\times S/T\times K}\mathcal{N}(M_{0,\varepsilon}|X_{0}) 
=
e_{S}\mathcal{N}(M_{0,\varepsilon}|X_{0}). 
$$ 
Hence the statement of the lemma is equivalent to
\begin{equation}\label{Eq-JK-11}
\oint_{M_{\varepsilon,0}} \kappa'_{T\times S/T\times K} \left( \frac{ \kappa_{T\times S/K} (\alpha e^{\omega - \mu_{T\times S}}) }{ e_{T\times S}\mathcal{N}(M_{\varepsilon}|X) } \right)
=
\textnormal{EqRes}_{x} \frac{1}{vol(T)} \left( \oint_{M_{\varepsilon}} \frac{\kappa_{T\times S/K} (\alpha e^{\omega -\mu_{T\times S}}) }{ e_{T\times S}\mathcal{N}(M_{\varepsilon}|X) } \right)(x)
\end{equation}
By Remark \ref{Rk-I-6} it is enough to consider poles spanned by isotropy $T\times S$-weight vectors at a fixed point component. Thus poles of $\oint_{X} \Delta_{X}(\alpha e^{\omega - \mu_{T\times S}})$ may be divided in two groups:
\begin{enumerate}
\item[(i)] poles contained in $\ker \pi$,

\item[(ii)] poles of $\oint_{M} \alpha e^{\omega - \mu_{T\times S}}$ not in $\ker \pi$.
\end{enumerate}
We conclude the proof of the lemma with the following two lemmas. 
	\end{proof}

	\begin{lemma*}\label{Lem-JK-7}
$$
\sum_{ \textnormal{poles }V\nsubseteq \ker \pi } \textnormal{JKRes}_{v}\frac{1}{vol(T)} \left( \oint_{M_{\varepsilon}} \frac{\kappa_{T\times S/K}(\alpha e^{\omega - \mu_{T\times S}}) }{e_{T\times S}\mathcal{N}(M_{\varepsilon}|X)} \right)(v,s) = 0,
$$
where $v$ is the bases induced by $x$ on $V$.
	\end{lemma*}

	\begin{proof}
$V$ is a pole of $\oint_{M}\alpha e^{\omega - \mu_{T\times S}}$ since $V\nsubseteq \ker \pi$. Let $H \subset M_{\varepsilon}^{T\times S}$ fixed point component. Recall $H$ is of form $F_{[m,0]}$ for some $m\in M$. We may suppose that $V$ is spanned by isotropy $T\times S$-weight vectors $\beta_{1},\ldots,\beta_{r}$ at $H$. Since $V= \langle \beta_{1}, \ldots, \beta_{r} \rangle \nsubseteq \ker\pi$ we may also suppose that $\beta_{1} = \varrho_{0}(\gamma) \in \mathfrak{t}^{*} \times \mathfrak{s}^{*}$ and $\beta_{i} = \varrho_{m}(\eta_{i})$ where $\eta_{i} \in (\mathfrak{t}\oplus\mathfrak{s})^{*}_{m}$ for all $i=2,\ldots,r$. Let $U \subset T\times S$ be the subtorus such that $Lie(U) = \cap_{i=2}^{r}\ker \eta_{i} \subset \mathfrak{t}\oplus \mathfrak{s}$. Let $Y \subset M^{U}$ be the fixed component containing $F_{m}$. Recall that $\mu_{T\times S}(F_{m}) = \mu_{T\times S}(H) + \langle \varrho_{0}(\gamma) \rangle$ and $Y^{T\times S} \neq \emptyset$ since $\varphi|_{Y}$ is proper and bounded below. Hence the supporting affine plane of $\mu_{T\times S}(Y)$ is equal to 
$$
\mu_{T\times S}(H) + \langle \varrho_{0}(\gamma), \eta_{2},\ldots, \eta_{r} \rangle 
=
\mu_{T\times S}(Y^{T\times S}) + \langle \varrho_{0}(\gamma), \varrho_{m}(\eta_{2}),\ldots, \varrho_{m}(\eta_{r}) \rangle
=
\mu_{T\times S}(H) + V, 
$$
because $\beta_{i} = \varrho_{m}(\eta_{i})$ equals $\eta_{i}$ modulo $\langle \varrho_{0}(\gamma) \rangle$. By Lemma \ref{Lem-JK-1}(a) if 
$$
\{p\}:= \big( \mu_{T\times S}(H) + V \big) \cap \big( \{0\} \times \mathfrak{s}^{*} \big)
$$
then $\pi(p) < \varepsilon = \pi(\mu_{T\times S}(H))$, i.e. $\pi(\mu_{T\times S}(H) - p) > 0$, therefore $\mu_{T\times S}(H) - p \in V$ is polarized. Hence
$$
-\mu_{T\times S}(H) = a_{1}v_{1}+ \ldots + a_{r}v_{r} - p
$$
with $p \in \mathfrak{s}^{*}$ and $a_{1}<0$ since $v$ is generic. Thus
$$
\textnormal{JKRes}_{v}\frac{1}{vol(T)} \left( \oint_{M_{\varepsilon}} \frac{\kappa_{T\times S/K} (\alpha e^{\omega - \mu_{T\times S}})}{e_{T \times  S}\mathcal{N}(M_{\varepsilon}|X)}\right)(v,s) = 0.
$$
	\end{proof}

	\begin{lemma*}\label{Lem-JK-8}
\begin{multline*}
\sum_{\textnormal{poles }V \subset \ker\pi} \textnormal{JKRes}_{v} \frac{1}{vol(T)} \left( \oint_{M_{\varepsilon}} \frac{\kappa_{T\times S/K}(\alpha e^{\omega-\mu_{T\times S}}) }{ e_{T \times  S}\mathcal{N}(M_{\varepsilon}|X)} \right)(v,s)
\\
=
\oint_{M_{\varepsilon,0}} \kappa'_{T\times S/T\times K} \left( \frac{\kappa_{T\times S/K}(\alpha e^{\omega - \mu_{T\times S}}) }{ e_{T\times S}\mathcal{N}(M_{\varepsilon}|X) } \right)
\end{multline*}
	\end{lemma*}

	\begin{proof}[Proof]
We decompose $\mathfrak{t}^{*}\times \mathfrak{s}^{*} = \langle \varrho_{0}(\gamma) \rangle \oplus ((\mathfrak{t}\oplus\mathfrak{s})/\mathfrak{k})^{*}$, 
hence we have 
$$
e_{T\times S}\mathcal{N}(M_{\varepsilon}|X) = - \varrho_{0}(\gamma)  + e_{T\times S/K}\mathcal{N}(M_{\varepsilon}|X).
$$
Moreover, on $M_{\varepsilon}$ only $T\times S/K$ acts effectively, hence $H_{T\times S}(M_{\varepsilon}) = H_{T\times S/K}(M_{\varepsilon}) \otimes H_{K}(pt)$. 
We expand fractions with respect to $v\ll \varrho_{0}(\gamma)$ and $v\ll s$, hence  considering $\varrho_{0}(\gamma)$ as non-zero real parameter then $\oint_{M_{\varepsilon}}\frac{\kappa_{T\times S/K}(\alpha e^{\omega - \mu_{T\times S}})}{ e_{T\times S}\mathcal{N}(M_{\varepsilon}|X) }$ is well-defined as integral of equivariant cohomology class in $H_{T\times S/K}(M_{\varepsilon})$.
Finally, the lemma follows from Theorem \ref{Thm-JK-1} (considering $\varrho_{0}(\gamma)$ as non-zero real parameter.)
\end{proof}


\subsection{Non-abelian version}

We deduce the non-abelian version of Theorem \ref{Thm-JK-2} following closely \cite{Martin}. We adapt their technique to the equivariant setting which are carried out in Lemma \ref{Lem-JK-9} and \ref{Lem-JK-10}. In this subsection we use a different notational system for Kirwan maps as before.

Suppose that $0\in \mathfrak{g}^{*}$ is a regular value of $\mu_{T}:M\to \mathfrak{t}^{*}$, hence it is also a regular value of $\phi_{T}:X \to \mathfrak{t}^{*}$. Denote $\mathfrak{v}^{*} = \ker(\mathfrak{g}^{*} \to \mathfrak{t}^{*})$. Choice of positive roots fixes the orientation of $\mathfrak{v}$ and $\mathfrak{v}^{*}$ such that 
$$\mathfrak{v}^{*} \simeq \oplus_{\beta}\mathbb{C}_{(\beta)},
$$ 
where the sum is over the negative roots and $\mathbb{C}_{(\beta)}$ is the one-dimensional representation of $T$ on which it acts by weight $\beta$. The restriction of $\phi_{G}$ to $\phi^{-1}_{T}(0)$ defines an $T\times S$-equivariant map $\sigma':\phi^{-1}_{T}(0) \to \mathfrak{v}^{*}$ ($S$ acts trivially on $\mathfrak{v}^{*}$) which induces an $S$-equivariant section $\sigma$ of the associated bundle $E^{-}:=(X\times \mathfrak{v}^{*})/\!\!/T \to X/\!\!/T$. $0\in \mathfrak{g}^{*}$ is regular  value of $\phi_{G}$ is equivalent to $\sigma$ being a section transverse to the zero section. Denote $Z:=\phi_{G}^{-1}(0) \subset X$, hence $Z/T$ is a submanifold of $X/\!\!/T$. Consider the following diagram
$$
\xymatrix{
Z/T \ar[d]^{\pi} \ar@{^{(}->}[r]^{i} & X/\!\!/T 
\\
Z/G = X/\!\!/G
}
$$
The vertical subbundle $\ker d\pi$ of $T(Z/T)$ is isomorphic to $E^{+}|_{Z/T}$, where $E^{+} = (X\times \mathfrak{v})/\!\!/T$. Denote $\kappa'_{T}:H_{T\times S}(X) \to H_{S}(X/\!\!/T)$ and $\kappa'_{G}:H_{G\times S}(X) \to H_{S}(X/\!\!/G)$  the abelian and non-abelian Kirwan maps. For $\alpha \in H_{G\times S}(X) \simeq H_{T\times S}(X)^{W}$ we compute
\begin{align*}
\int_{X/\!\!/G}\kappa'_{G}(\alpha) 
&{}=  
\frac{1}{|W|} \int_{Z/T}\pi^{*}\kappa_{G}(\alpha) e_{S}(E^{+}|_{Z/T})
&& \textnormal{by Lemma \ref{Lem-JK-9}}
\\
&{}=\frac{1}{|W|}\int_{Z/T} i^{*}(\kappa'_{T}(\alpha) e_{S}(E^{+}))
&& \pi^{*}\kappa'_{G} = i^{*}\kappa'_{T}
\\
&{}=
\frac{1}{|W|} \int_{X/\!\!/T}\kappa'_{T}(\alpha) e_{S}(E^{+}) e_{S}(E^{-})
&& \textnormal{by Lemma \ref{Lem-JK-10}}
\\ 
&{}=
\frac{1}{|W|} \int_{X/\!\!/T}\kappa'_{T} ( \alpha\, \varpi ),
&&
\end{align*}
where $\varpi$ is the product of all roots.

Recall that by symplectic cut with respect to $\varphi$ on $M/\!\!/T$ and $M/\!\!/G$ we have
$$
\oint_{M/\!\!/T} \kappa_{T}(\alpha  \varpi )
=
\int_{X/\!\!/T}\kappa'_{T}\Delta_{X}(\alpha  \varpi )
-
\int_{M_{\varepsilon}/\!\!/T} \frac{ \kappa''_{T\times K}(\alpha  \varpi )  }{ e_{S}\mathcal{N}(M_{\varepsilon}/\!\!/T \,|\, X/\!\!/T) }
$$
and
$$
\oint_{M/\!\!/G} \kappa_{T}(\alpha )
=
\int_{X/\!\!/G}\kappa'_{G}\Delta_{X}(\alpha )
-
\int_{M_{\varepsilon}/\!\!/G} \frac{ \kappa''_{G\times K}(\alpha  )  }{ e_{S}\mathcal{N}(M_{\varepsilon}/\!\!/G \,|\, X/\!\!/G) },
$$
where $\kappa''_{T\times K} : H_{T\times S}(M) \to H_{T\times S/T\times K}(M_{\varepsilon}/\!\!/T)$ and $\kappa''_{G\times K} : H_{G\times S}(M) \to H_{G\times S/G\times K}(M_{\varepsilon}/\!\!/G)$ are the Kirwan maps. Thus
\begin{align*}
\frac{1}{|W|} \oint_{M/\!\!/T} \kappa_{T}(\alpha \varpi)
&{}=
\frac{1}{|W|} \int_{X/\!\!/T}\kappa'_{T}\Delta_{X}(\alpha  \varpi )
-
\frac{1}{|W|} \int_{M_{\varepsilon}/\!\!/T} \frac{ \kappa''_{T\times K}(\alpha  \varpi)  }{ e_{S}\mathcal{N}(M_{\varepsilon}/\!\!/T \,|\, X/\!\!/T) }
\\
&{}=
\int_{X/\!\!/G}\kappa'_{G}\Delta_{X}(\alpha  )
-
 \int_{M_{\varepsilon}/\!\!/G} \frac{ \kappa''_{G\times K}(\alpha  )  }{ e_{S}\mathcal{N}(M_{\varepsilon}/\!\!/T \,|\, X/\!\!/T) }
\\
&{}= \oint_{M/\!\!/G}\kappa_{G}(\alpha ).
\end{align*}
If $0$ is not a regular value of $\mu_{T}$ then choose a regular value $\rho$ close $0$. As observed in \cite{Martin} we have
$$
\int_{X/\!\!/G} \kappa'_{G}(\alpha) 
= 
\lim_{s\to 0}\frac{1}{|W|} \int_{X/\!\!/_{s\rho} T} \kappa'_{T} ( \alpha\,\varpi)
$$
and similarly we have
$$
\oint_{M/\!\!/G}\kappa_{G}(\alpha e^{\omega - \mu_{G\times S}}) 
= 
\lim_{s\to 0} \frac{1}{|W|} \oint_{M/\!\!/_{s\rho}T} \kappa_{T}(\alpha e^{\omega - \mu_{T\times S}} \varpi).
$$

We change notation of the Kirwan map $\kappa_{G}$ to $\kappa_{S}:H_{G\times S}(M) \to H_{S}(M/\!\!/G)$ to make the statement of our main result compatible with Theorem \ref{Thm-JK-2}.

	\begin{thm}\label{Thm-JK-3}
Let $y =\{y_{1}, \ldots, y_{r+q}\}$ be an ordered bases of $\mathfrak{t}^{*} \times \mathfrak{s}^{*}$ such that $\pi(y_{1})>0$ and $y_{2},\ldots,y_{r+q} \in \ker \pi$. Then
\begin{equation}
\oint_{M/\!\!/G} \kappa_{S}(\alpha e^{ \omega - \mu_{ G\times S } }) 
= 
\lim_{s\to 0}
\textnormal{EqRes}_{x} \left( \frac{\varpi}{|W|vol(T)} \oint_{M} \alpha e^{ \omega - \mu_{T\times S} +s\rho } \right)(x),
\end{equation}
where $x$ is a generic bases with respect to $\oint_{M} \alpha e^{\omega - \mu_{T \times S}}$, inducing the same polarization as $y$ on isotropy $T\times S$-weights of $M$ and $\rho$ is a regular value of $\mu_{T}$ close to $0$.
	\end{thm}

We conclude this subsection by proving the following lemmas.

\begin{lemma}\label{Lem-JK-9}
$$
\int_{Z/G} \kappa'_{G}(\alpha) 
= 
\frac{1}{|W|} \int_{Z/T} \pi^{*}\kappa'_{G}(\alpha) e_{S}(\ker d\pi).
$$
\end{lemma}

\begin{proof}
$\pi:Z/T \to Z/G$ is fibration with fiber $G/T$ and $\pi$ is $S$-equivariant. For $H \subset (Z/G)^{S}$ fixed point component we have $\pi^{-1}(H) = F \subset (Z/T)^{S}$ is also a fixed point component. Moreover, we have $S$-equivariant isomorphism of normal bundles $\mathcal{N}(F|Z/T) \simeq \pi^{*}\mathcal{N}(H|Z/G)$. Finally, we compute
\begin{align*}
\int_{Z/G} \kappa'_{G}(\alpha) 
&{}= 
\sum_{H\subset (Z/G)^{S}} \int_{H} \frac{ i^{*}_{H}\kappa'_{G}(\alpha) }{ e_{S}\mathcal{N}(H|Z/G) }
\\
&{}=
\frac{1}{|W|} \sum_{F\subset (Z/T)^{S}} \int_{F} \frac{i_{F}^{*}\pi^{*}\kappa'_{G}(\alpha) e(\ker d\pi) }{ e_{S}\pi^{*}\mathcal{N}(H|Z/G) }
\end{align*}
\begin{align*}
&{}=
\frac{1}{|W|} \sum_{F\subset (Z/T)^{S}} \int_{F} \frac{i_{F}^{*}\pi^{*}\kappa'_{G}(\alpha) e_{S}(\ker d\pi) }{ e_{S}\mathcal{N}(F|Z/T) }
\\
&{}=
\frac{1}{|W|} \int_{Z/T} \pi^{*}\kappa'_{G}(\alpha) e_{S}(\ker d\pi).
\end{align*}
\end{proof}

\begin{lemma}\label{Lem-JK-10}
Let $E \to X$ be an $S$-equivariant vector bundle over a compact space $X$. Let $\sigma$ be an $S$-equivariant section, transverse to the zero section. Denote $Z = \sigma^{-1}(0)$ the zero set of the section and $i_{Z}:Z \hookrightarrow X$ the inclusion. For any  $\eta \in H_{S}(X)$ we have
$$
\int_{Z} i_{Z}^{*}\eta = \int_{X} \eta e_{S}(E).
$$
\end{lemma}

\begin{proof}
Let $F\subset X^{S}$ and $H \subset F\cap Z \subset Z^{S}$ be fixed point components. Denote $i_{F}:F \hookrightarrow X $, $i_{H}:H\hookrightarrow Z$, and $j:H\hookrightarrow F$ the inclusions.
By transversality of $\sigma$ we have equivariant isomorphism of vector bundles 
\begin{equation}\label{Eq-JK-15}
E|_{Z} \simeq \mathcal{N}(Z|X)
\end{equation}
We have equivariant decomposition $E|_{F} = (E|_{F})^{S} \oplus E'$, hence $E|_{H} = (E|_{H})^{S} \oplus E'|_{H}$. Moreover, by  (\ref{Eq-JK-15})
\begin{equation}\label{Eq-JK-16}
\mathcal{N}(Z|X)|_{H} \simeq (E|_{H})^{S} \oplus E'|_{H}
\end{equation}
The inclusions $H\subset Z \subset X$ gives
\begin{equation}\label{Eq-JK-17}
\mathcal{N}(H|X) \simeq \mathcal{N}(H|Z) \oplus \mathcal{N}(Z|X)|_{H},
\end{equation}
moreover the inclusions $H\subset F\subset X$ yields
\begin{equation}\label{Eq-JK-18}
\mathcal{N}(H|X) \simeq \mathcal{N}(H|F) \oplus \mathcal{N}(F|X)|_{H},
\end{equation}
and finally we also have a decomposition
\begin{equation}\label{Eq-JK-19}
\mathcal{N}(H|X) \simeq \mathcal{N}(H|X)^{S} \oplus \mathcal{N}(H|X)'.
\end{equation}
Since $\mathcal{N}(H|F) \simeq \mathcal{N}(H|X)^{S}$ by equations (\ref{Eq-JK-18}) and (\ref{Eq-JK-19}) we have
\begin{equation}\label{Eq-JK-20}
\mathcal{N}(F|X)|_{H} \simeq \mathcal{N}(H|X)',
\end{equation}
and equations (\ref{Eq-JK-16}), (\ref{Eq-JK-17}) and (\ref{Eq-JK-19}) give
\begin{equation}\label{Eq-JK-21}
\mathcal{N}(H|X)' \simeq \mathcal{N}(H|Z) \oplus E'|_{H}.
\end{equation}
The isomorphisms (\ref{Eq-JK-20}) and (\ref{Eq-JK-21}) implies 
\begin{equation}\label{Eq-JK-22}
e_{S}\mathcal{N}(H|Z) =  \frac{ j^{*} e_{S}\mathcal{N}(F|X) }{ j^{*}e_{S}(E') }
\end{equation}
Finally, we compute
\begin{align*}
\int_{Z} i^{*}_{Z}\eta 
&{}=
\sum_{H \subset Z^{S}} \int_{H} \frac{ i^{*}_{H}\eta }{ e_{S}\mathcal{N}(H|Z) }
&& \\
&{}=
\sum_{H \subset Z^{S}} \int_{H} \frac{ (i^{*}_{H}\eta) j^{*}e_{S}(E') }{ j^{*}e_{S}\mathcal{N}(F|X) }
&& \textnormal{by (\ref{Eq-JK-22})}
\\
&{}=
\sum_{F \subset X^{S}} \int_{F} \frac{ (i^{*}_{F}\eta) e_{S}(E') e((E|_{F})^{S}) }{e_{S}\mathcal{N}(F|X)}
&& \textnormal{by ordinary Poincar\'e duality}
\\
&{}=
\sum_{F \subset X^{S}} \int_{F} \frac{ (i^{*}_{F}\eta) e_{S}(E|_{F}) }{e_{S}\mathcal{N}(F|X)}
&& e((E|_{F})^{S}) = e_{S}((E|_{F})^{S})
\\
&{}=
\int_{X} \eta\, e_{S}(E).
&&
\end{align*}
\end{proof}

\subsection{HyperK\"ahler version} 

We formulate an analogue of Theorem \ref{Thm-JK-3} for hyperK\"ahler quotients. First we compare torus hyperK\"ahler quotients to symplectic quotients then by  \cite{HP2} we conclude the formula for general hyperK\"ahler quotients.

Let $M$ be a hyperK\"ahler manifold with real symplectic form $\omega_{\mathbb{R}}$ and complex symplectic form $\omega_{\mathbb{C}}$. Let $M$ admit an action of a compact Lie group $G$ which acts on it in a hyper-Hamiltonian manner with hyperK\"ahler moment map $\mu=(\mu_{\mathbb{R}}, \mu_{\mathbb{C}}): M \to \mathfrak{g}^{*}\oplus\mathfrak{g}^{*}_{\mathbb{C}}$. We also consider an additional Hamiltonian $S$-action on $(M,\omega_{\mathbb{R}})$ which commutes with the $G$-action and its moment map is denoted by $\mu_{S}:M\to \mathfrak{s}^{*}$. We assume that $\mathfrak{g}^{*}_{\mathbb{C}}$ is a $G\times S$-representation (coadjoint action for $G$) and $\mu_{\mathbb{C}}$ is $G\times S$-equivariant. We also need that for a maximal torus $T\subset G$ with Lie algebra $\mathfrak{t}$ we have $(\mathfrak{t}^{*}_{\mathbb{C}})^{S} = \{0\}$. Similarly to the symplectic case we suppose that there is an one dimensional subtorus $K$ in the center of $G\times S$ with moment map $\mu_{K}:(M,\omega_{\mathbb{R}}) \to \mathfrak{k}^{*}$ such that $\mu_{K} = \varphi\gamma$ where $\gamma \in \mathfrak{k}^{*}_{\mathbb{Z}}$ and $\varphi:M\to \mathbb{R}$ is proper and bounded below. Finally, let $(\xi,0)\in(\mathfrak{g}^{*}\oplus\mathfrak{g}^{*}_{\mathbb{C}})^{G}$ be a regular value of $\mu$.

First we discuss the abelian case, therefore let $G=T$.
As before we construct $X$ as symplectic cut of $M$ with respect to $\varphi$. Denote by $\phi_{T\times S}:X \to \mathfrak{t}^{*}\times \mathfrak{s}^{*}$ the $T\times S$-moment map induced by $\mu_{\mathbb{R}}\times \mu_{S}$
and by $\psi:X \to \mathfrak{t}^{*}_{\mathbb{C}}$ the equivariant map induced by $\mu_{\mathbb{C}}$. Suppose that $\xi$ is a regular value of $\mu_{\mathbb{R}}$, hence it is also a regular value of $\phi_{T}$. Let $\tilde{\psi}:X/\!\!/_{\xi}T \to \mathfrak{t}^{*}_{\mathbb{C}}$ the $S$-equivariant map induced by $\psi$ and denote $i:Z\hookrightarrow X/\!\!/_{\xi}T$ the inclusion where $Z:=\tilde{\psi}^{-1}(0)$. We have $Z^{S} = (X/\!\!/_{\xi}T)^{S}$ by $(\mathfrak{t}^{*}_{\mathbb{C}})^{S} = \{0\}$ and remark that $Z\times \mathfrak{t}^{*}_{\mathbb{C}} \simeq \mathcal{N}(Z|X/\!\!/_{\xi}T)$.  For $\alpha'\in H_{T\times S}(X)$ we compute
\begin{align*}
\int_{Z}i^{*}_{Z}\kappa'_{T}(\alpha') 
&{}= 
\sum_{F\subset Z^{S}} \int_{F} \frac{i^{*}_{F}\kappa'_{T}(\alpha')}{e_{S}\mathcal{N}(F|Z)}
\\
&{}= \sum_{F\subset Z^{S}} \int_{F} \frac{i^{*}_{F}\kappa'_{T}(\alpha') e_{S}\mathcal{N}(Z|X/\!\!/_{\xi}T)}{e_{S}\mathcal{N}(F|X/\!\!/_{\xi}T)}
\\
&{}=\sum_{F\subset Z^{S}} \int_{F} \frac{i^{*}_{F}\kappa'_{T}(\alpha') e_{S}(Z\times \mathfrak{t}^{*}_{\mathbb{C}})}{e_{S}\mathcal{N}(F|X/\!\!/_{\xi}T)}
\end{align*}
\begin{align*}
&{}=\int_{X/\!\!/_{\xi}T} \kappa_{T}'(\alpha') \vartheta
\\
&{}=\int_{X/\!\!/_{\xi}T} \kappa_{T}'(\alpha' \vartheta),
\end{align*}
where $\vartheta$ is the product of $S$-weight on $\mathfrak{t}^{*}_{\mathbb{C}}$ and remark that $\kappa'_{T} : H_{T\times S}(X) \to H_{S}(X/\!\!/_{\xi}T)$ is $H_{S}(pt)$-linear. 
As before for $\alpha \in H_{T\times S}(M)$ we have
\begin{align*}
\oint_{M/\!\!/\!\!/\!\!/_{(\xi,0)}T} \kappa_{T}(\alpha) 
&{}=
\int_{Z}i^{*}_{Z} \kappa'_{T}\Delta_{X}(\alpha) 
- 
\int_{Z\cap M_{\varepsilon}/\!\!/_{\xi}T} \frac{i^{*}_{Z\cap M_{\varepsilon}/\!\!/_{\xi}T}\kappa''_{T\times K}(\alpha) }{e_{S}\mathcal{N}(Z\cap M_{\varepsilon}/\!\!/_{\xi}T\,|\,Z)}
\\
&{}=
\int_{X/\!\!/_{\xi}T}\kappa'_{T}(\Delta_{X}(\alpha)\vartheta)
-
\int_{M_{\varepsilon}/\!\!/T} \frac{\kappa''_{T\times K}(\alpha \vartheta)}{e_{S}\mathcal{N}(M_{\varepsilon}/\!\!/_{\xi}T\,|\,X/\!\!/_{\xi}T)}
\\
&{}=
\oint_{M/\!\!/_{\xi}T}\kappa_{T}(\alpha\vartheta),
\end{align*}
where $\kappa''_{T\times K}:H_{T\times S}(M) \to H_{T\times S/T\times K}(M_{\varepsilon}/\!\!/_{\xi}T)$.
If $\xi$ is not a regular value of $\mu_{\mathbb{R}}$ then we perturb $\xi$ to a regular value and we take the limit. Let $\rho\in\mathfrak{t}^{*}$ be in a small neighborhood of $0$ such that $\xi + \rho$ is a regular value of $\mu_{\mathbb{R}}$. Then
$$
\oint_{M/\!\!/\!\!/\!\!/_{(\xi,0)} T} \kappa_{T}(\alpha) = \lim_{s\to 0} \oint_{M/\!\!/_{\xi}T} \kappa_{T}(\alpha \vartheta).
$$ 

The general case can be deduced as follows. We introduce notation $\mu^{T}_{\mathbb{R}}:M\to \mathfrak{t}^{*}$ for the abelian real moment map, it is the composition of $\mu_{\mathbb{R}}$ with the projection $\mathfrak{g}^{*}\to \mathfrak{t}^{*}$.
Suppose that $\xi$ is a regular value of $\mu^{T}_{\mathbb{R}}:M\to \mathfrak{t}^{*}$ the abelian real moment map, otherwise we perturb $\xi$ to a regular value of $\mu^{T}_{\mathbb{R}}$. Then by Theorem 2.2 of \cite{HP2} we have the following relation
$$
\oint_{M/\!\!/\!\!/\!\!/_{(\xi,0)}G} \kappa_{G}(\alpha) 
= 
\frac{1}{|W|} \oint_{M/\!\!/\!\!/\!\!/_{(\xi,0)}T} \kappa_{T}(\alpha\, \varpi_{\mathbb{R}}\, \varpi_{\mathbb{C}})
=
\frac{1}{|W|} \oint_{M/\!\!/_{\xi}T} \kappa_{T}(\alpha\, \vartheta\, \varpi_{\mathbb{R}}\, \varpi_{\mathbb{C}}),
$$
where $\vartheta$ is the product of $S$-weight on $\mathfrak{t}^{*}_{\mathbb{C}}$, $\varpi_{\mathbb{R}}=\varpi$ is the product of roots of $G$ and $\varpi_{\mathbb{C}}$ is the product of $T\times S$-weights on $\mathfrak{v}^{*}_{\mathbb{C}} = \ker(\mathfrak{g}^{*}_{\mathbb{C}} \to \mathfrak{t}^{*}_{\mathbb{C}})$. To make the notations compatible with Theorem \ref{Thm-JK-3} we change notation of the Kirwan map from $\kappa_{G}$ to $\kappa_{S}:H_{G\times S}(M) \to H_{S}(M/\!\!/\!\!/\!\!/_{(\xi,0)}G)$

	\begin{thm}\label{Thm-JK-4}
Let $y =\{y_{1}, \ldots, y_{r+q}\}$ be an ordered bases of $\mathfrak{t}^{*} \times \mathfrak{s}^{*}$ such that $\pi(y_{1})>0$ and $y_{2},\ldots,y_{r+q} \in \ker \pi$. Then
\begin{equation}
\oint_{M/\!\!/\!\!/\!\!/_{(\xi,0)}G} \kappa_{S}(\alpha e^{ \omega_{\mathbb{R}} - \mu_{\mathbb{R} } - \mu_{S} + \xi }) 
= 
\lim_{s\to 0}
\textnormal{EqRes}_{x} \left( \frac{\vartheta\,\varpi_{\mathbb{R}}\,\varpi_{\mathbb{C}}}{|W|vol(T)} \oint_{M} \alpha e^{ \omega_{\mathbb{R}} - \mu_{\mathbb{R}}^{T} - \mu_{S} + \xi + s\rho } \right)(x),
\end{equation}
where $x$ is a generic bases with respect to $\oint_{M} \alpha e^{\omega_{\mathbb{R}} - \mu_{\mathbb{R}}^{T} - \mu_{S} +\xi}$, inducing the same polarization as $y$ on isotropy $T\times S$-weights of $M$, $\varpi$ is the product of roots, $\vartheta\varpi$ is the product of $T\times S$-weights on $\mathfrak{g}^{*}_{\mathbb{C}}$, and $\rho\in \mathfrak{t}^{*}$ small such that $\xi+\rho$ is a regular value of $\mu_{\mathbb{R}}^{T}$. 
	\end{thm}

\section{Applications}

\subsection{Cohomology ring of Hilbert scheme of points in the plane} 
As an application of Theorem \ref{Thm-JK-4} we compute the (equivariant) cohomology ring of $\textnormal{Hilb}^{n}(\mathbb{C}^{2})$ the Hilbert scheme of point on the plane (cf. \cite{LS}, \cite{Va}) using the construction of \cite{Na} to get $\textnormal{Hilb}^{n}(\mathbb{C}^{2})$ as a hyperK\"ahler quotient. To compute the resulting residue we develop a language of diagrams which make the computation more intuitive.

Consider $\mathcal{A} = \End(\mathbb{C}^{n}) \oplus \Hom(\mathbb{C}, \mathbb{C}^{n})$ with $U(n)$ action given by 
$$
g\cdot (A,a) = (g^{-1}Ag, g^{-1}a), \qquad g\in U(n),\ A\in \End(\mathbb{C}^{n}),\ a\in \Hom(\mathbb{C},\mathbb{C}^{n}). 
$$
$\mathcal{M} = T^{*}\mathcal{A}$ is a hyperK\"ahler manifold with natural $U(n)$-action 
$$
g\cdot (A,a,B,b) = (g^{-1}Ag, g^{-1}a, g^{-1}Bg,bg),
$$
where $A,B\in \End(\mathbb{C}^{n})$, $a\in \Hom(\mathbb{C},\mathbb{C}^{n})$ and $b\in \Hom(\mathbb{C}^{n},\mathbb{C})$.
This action is hyperHamiltonian with real and complex moment maps 
\begin{align*}
\mu_{\mathbb{R}}:(\mathcal{M},\omega_{\mathbb{R}}) \to \mathfrak{u}(n)^{*},
&\quad 
(A,a,B,b)\mapsto \frac{\sqrt{-1}}{2} \big( [A,A^{*}] + [B,B^{*}] + aa^{*} -b^{*}b \big),
\\
\mu_{\mathbb{C}}:(\mathcal{M},\omega_{\mathbb{C}}) \to \mathfrak{u}(n)_{\mathbb{C}}^{*},
& 
\quad
(A,a,B,b) \mapsto [A,B] + ab,
\end{align*}
where we used identification $\mathfrak{u}(n)^{*} \simeq \mathfrak{u}(n)$. Let $\xi = \frac{\sqrt{-1}}{2}I_{n}$. Then $(\xi,0)$ is a regular value of $\mu=(\mu_{\mathbb{R}},\mu_{\mathbb{C}})$ and we have
$$
\textnormal{Hilb}^{n}(\mathbb{C}^{2}) \simeq \mathcal{M}/\!\!/\!\!/\!\!/_{(\xi,0)}U(n).
$$
We remark that $\xi$ is a regular value of $\mu_{\mathbb{R}}^{T}$ and $(\xi,0)$ is a regular value of $\mu^{T}=(\mu_{\mathbb{R}}^{T},\mu_{\mathbb{C}}^{T})$, where $\mu_{\mathbb{R}}^{T}:(\mathcal{M},\omega_{\mathbb{R}}) \to \mathfrak{t}^{*}$,
$$ 
\mu_{\mathbb{R}}^{T}(A,a,B,b)_{k}= \frac{\sqrt{-1}}{2} \Big( \sum_{j=1}^{n} \big( A_{kj}\bar{A}_{kj} - A_{jk}\bar{A}_{jk} + B_{kj}\bar{B}_{kj} - B_{jk}\bar{B}_{jk} \big) + a_{k}\bar{a}_{k} - \bar{b}_{k}b_{k} \Big). 
$$
is the real abelian moment map and
$\mu_{\mathbb{C}}^{T}:(\mathcal{M},\omega_{\mathbb{C}}) \to \mathfrak{t}^{*}_{\mathbb{C}}$,
$$ 
\mu_{\mathbb{C}}^{T}(A,a,B,b)_{k}= \sum_{j=1}^{n} A_{kj} B_{jk} + a_{k}b_{k}. 
$$
Hence the abelian symplectic and hyperK\"ahler quotients $\mathcal{M}/\!\!/_{\xi}T$ and $\mathcal{M}/\!\!/\!\!/\!\!/_{(\xi,0)}T$ exist. The latter is a hypertoric variety (cf. \cite{HP1}).

We consider an auxiliary circle action of $S=U(1)$ on $\mathcal{M}$ given by
$$
s\cdot (A,a,B,b) = (s^{N}A,s^{N}a,sB,sb),\qquad (N>n).
$$
It commutes with the $U(n)$-action and it admits moment map $\varphi:(\mathcal{M},\omega_{\mathbb{R}}) \to \mathbb{R}$,
$$
\varphi(A,a,B,b) = \frac{\sqrt{-1}}{2} \textnormal{Tr} \big( NAA^{*} + BB^{*} + Na^{*}a + bb^{*} \big), 
$$
which is proper and bounded below. 

The Chern classes of the vector bundle $\Xi_{n} = \mathcal{M}\times \mathbb{C}^{n}/\!\!/\!\!/\!\!/_{(\xi,0)}U(n)$ -- where $U(n)$ acts on $\mathbb{C}^{n}$ via standard representation -- generate the cohomology ring $H(\textnormal{Hilb}^{n}(\mathbb{C}^{n}))$ \cite{ES}. This is equivalent to the surjectivity of the Kirwan map $\kappa: H_{U(n)}(\mathcal{M})\to H(\textnormal{Hilb}^{n}(\mathbb{C}^{2}))$.

\begin{lemma*}[cf. \cite{HP1} Lemma 4.9]
The equivariant Kirwan map 
$
\kappa_{S}:H_{U(n)\times S}(\mathcal{M}) \to H_{S}(\textnormal{Hilb}^{n}(\mathbb{C}^{2}))
$
is also surjective.
\end{lemma*}

\begin{proof}
Since the $S$-moment map is proper and bounded below $H_{S}(\textnormal{Hilb}^{n}(\mathbb{C}^{2}))$ is equivariantly formal, i.e. $H_{S}(\textnormal{Hilb}^{n}(\mathbb{C}^{2})) \simeq H(\textnormal{Hilb}^{n}(\mathbb{C}^{2}))  \otimes H_{S}(pt)$ as $H_{S}(pt)$-modules (\cite{HP1}, \cite{TW}). Recall that 
$H_{S}(pt) \simeq \mathbb{R}[\sigma]$ and $H_{U(n)}(pt) \simeq \mathbb{R}[\tau_{1},\ldots,\tau_{n}]^{S_{n}}$, hence
$H_{U(n)}(\mathcal{M}) \simeq \mathbb{R}[\tau_{1},\ldots,\tau_{n}]^{S_{n}}$ and $H_{U(n)\times S}(\mathcal{M}) \simeq \mathbb{R}[\tau_{1},\ldots,\tau_{n},\sigma]^{S_{n}}$. Moreover, we have the following commutative diagram
$$
\xymatrix{
H_{U(n)\times S}(\mathcal{M}) \ar[r]^{\kappa_{S}} & H_{S}(\textnormal{Hilb}^{n}(\mathbb{C}^{2})) 
\ar[d]^{\pi}
\\
H_{U(n)}(\mathcal{M}) \ar[u]^{i} \ar[r]_{\kappa} & H(\textnormal{Hilb}^{n}(\mathbb{C}^{2}))
}
$$
Consider the grading $\mathcal{H} = \oplus_{k\geq0}\mathcal{H}^{k}$ on $H_{U(n)\times S}(\mathcal{M})$ corresponding to $\oplus_{k\geq0}H^{k}(\textnormal{Hilb}^{n}(\mathbb{C}^{2}))\otimes H_{S}(pt)$. Recall that $\kappa_{S}$ is $H_{S}(pt)$ linear, hence $\mathcal{H}^{0}$ is in the image of $\kappa_{S}$. Let $\beta_{k} \in \mathcal{H}^{k}$. By surjectivity of the ordinary Kirwan map there is $\alpha_{k} \in H_{U(n)}(\mathcal{M})$ such that 
$$
\kappa(\alpha_{k}) = \pi(\beta_{k}).
$$ 
We may consider $\alpha_{k}$ as an element of $H_{U(n)\times S}(\mathcal{M})$ and we have
$\kappa_{S}(\alpha_{k}) - \beta_{k} \in \ker\pi$. Moreover, $\ker\pi = \sigma\cdot H_{S}(\textnormal{Hilb}^{n}(\mathbb{C}^{2}))$, thus 
$$
\kappa_{S}(\alpha_{k}) - \beta_{k} = \sigma \beta_{k-1}
$$ 
for some $\beta_{k-1}\in \mathcal{H}_{k-1}$. By inductive hypothesis there is $\alpha_{k-1}$ such that $\kappa_{S}(\alpha_{k-1}) = \beta_{k-1}$, hence
$$
\beta_{k} = \kappa_{S}(\alpha_{k} -\sigma\alpha_{k-1}).
$$
\end{proof}
By surjectivity of $\kappa_{S}$ we have $H_{S}(\textnormal{Hilb}^{n}(\mathbb{C}^{2})) \simeq \mathbb{R}[\sigma,\tau_{1},\ldots,\tau_{n}]^{S_{n}}\big/\ker\kappa_{S}$ and by formality
$$
H(\textnormal{Hilb}^{n}(\mathbb{C}^{2})) \simeq H_{S}(\textnormal{Hilb}^{n}(\mathbb{C}^{2})) \big/ \sigma H_{S}(\textnormal{Hilb}^{n}(\mathbb{C}^{2}))
$$
as rings. There is a non-degenerate pairing (\cite{HP2}) on $\widehat{H}_{S}(\textnormal{Hilb}^{n}(\mathbb{C}^{2})) := H_{S}(\textnormal{Hilb}^{n}(\mathbb{C}^{2}))\otimes \mathbb{R}[1/\sigma]$ given by
$$
\langle \widehat{\eta},\,\widehat{\gamma} \rangle 
:= 
\oint_{\textnormal{Hilb}^{n}(\mathbb{C}^{n})}  \widehat{\eta}\,\widehat{\gamma},
\qquad \forall\,  \widehat{\eta},\,\widehat{\gamma} \in \widehat{H}_{S}(\textnormal{Hilb}^{n}(\mathbb{C}^{2})).
$$
Therefore, $\eta \in H_{S}(\textnormal{Hilb}^{n}(\mathbb{C}^{2}))$ is zero exactly when
$$
\oint_{\textnormal{Hilb}^{n}(\mathbb{C}^{2})} \eta\, \gamma = 0,\qquad \forall\gamma \in H_{S}(\textnormal{Hilb}^{n}(\mathbb{C}^{2})).
$$
Hence, by the surjectivity of $\kappa_{S}$ we have
$$
\ker \kappa_{S} = \Big\{ \alpha \in H_{U(n)\times S}(\mathcal{M})  \,\Big|\,  \oint_{\textnormal{Hilb}^{n}(\mathbb{C}^{2})} \kappa_{S}(\alpha\beta) = 0,\ \forall \beta \in H_{U(n)\times S}(\mathcal{M})  \Big\}.
$$
We use Theorem \ref{Thm-JK-4} to compute this kernel.

\begin{thm}\label{Thm-HS-1}
For any $\alpha \in H_{U(n)\times S}(\mathcal{M}) \simeq \mathbb{R}[\sigma,\tau_{1},\ldots,\tau_{n}]^{S_{n}}$ we have
$$
\oint_{\textnormal{Hilb}^{n}(\mathbb{C}^{2})} \kappa_{S}(\alpha) = \frac{1}{vol(T)} \left( \frac{N+1}{N\sigma}\right)^{n} \sum_{\lambda \vdash n} \alpha(p_{\lambda},\sigma) b_{\lambda}(\sigma),
$$
where $b_{\lambda}(\sigma)$ is defined in (\ref{Eq-HS-9}) and $p_{\lambda}$ is defined (up to reorder) as follows. Into each box of the Young diagram of $\lambda$ we write an element of $\mathbb{Z}\sigma$ according to the next algorithm: we start with the bottom left box and we write $-\sigma$ into it; if a box is filled with $q\sigma$ then we fill the upper neighbor with $(q-N)\sigma$ and the right neighbor with $(q-1)\sigma$. We put all entries from the boxes into the vector $p_{\lambda}$.
\end{thm}

\begin{lemma}\label{Lem-HS-1}
Let $\mathbb{K}$ be a field and let $\{q_{1},\ldots,q_{m}\} \subset \mathbb{K}^{n}$, $q_{i} \neq q_{j}$ if $i\neq j$. Define $\mathcal{L}:\mathbb{K}[x_{1},\ldots,x_{n}] \to \mathbb{M}$, $P \mapsto \sum_{i=1}^{m} a_{i}P(q_{i})$, where $a_{i}\in \mathbb{K}^{*}$. Then 
$$
\{ 
P\in\mathbb{K}[x_{1},\ldots,x_{n}]  \,|\, \mathcal{L}(PQ) = 0,\ \forall Q
\}
=
\cap_{i=1}^{m} \ker(ev_{q_{i}}) 
= 
\cap_{i=1}^{m} \langle x_{j} - q_{ij}  \,|\, j=1,\ldots,n \rangle,
$$ 
where $q_{i} = (q_{i1},\ldots,q_{in}).$
\end{lemma}

\begin{proof}
Obviously, $\cap_{i=1}^{m}\ker(ev_{q_{i}}) \subset \{ 
P\in\mathbb{K}[x_{1},\ldots,x_{n}]  \,|\, \mathcal{L}(PQ) = 0,\ \forall Q
\} $. Consider the polynomial 
$$
Q_{i} = \prod_{k=1}^{n} \prod_{\substack{j=1 \\ q_{jk}\neq q_{ik} }}^{m} (x_{k} - q_{jk})
$$
Then $Q_{i}(q_{j}) = 0$ if $j\neq i$ and $Q_{i}(q_{i}) \neq 0$. If $P \in \{ 
P\in\mathbb{K}[x_{1},\ldots,x_{n}]  \,|\, \mathcal{L}(PQ) = 0,\ \forall Q
\}$  then for all $i$ we have $\mathcal{L}(PQ_{i}) = a_{i}P(q_{i})Q_{i}(q_{i}) = 0$, hence $P(q_{i}) = 0$. Thus $P \in \ker(ev_{q_{i}}).$
\end{proof}

From Theorem \ref{Thm-HS-1} and Lemma \ref{Lem-HS-1} follows

\begin{thm}\label{Thm-H-2}
Denote $C_{i}(\tau)$ the $i$th elementary symmetric polynomial in $\tau_{1},\ldots,\tau_{n}$ and let 
$$
I_{\lambda} = \langle C_{i}(\tau) - C_{i}(p_{\lambda})  \,|\,  i=1,\ldots,n \rangle,
$$
where $\lambda \vdash n$. Then
$$
H_{S}(\textnormal{Hilb}^{n}(\mathbb{C}^{2})) \simeq \mathbb{R}[C_{1}(\tau),\ldots,C_{n}(\tau),\sigma]\big/\cap_{\lambda\vdash n} I_{\lambda}.
$$
\end{thm}

\begin{proof}[Proof of Theorem \ref{Thm-HS-1}]
For $\alpha\in \mathbb{R}[\sigma,\tau_{1},\ldots,\tau_{n}]^{S_{n}}$ we have
$$
\oint_{\textnormal{Hilb}^{n}(\mathbb{C}^{2})} \kappa_{S}(\alpha) 
=
\lim_{\epsilon \to 0^{+}} \oint_{\textnormal{Hilb}^{n}(\mathbb{C}^{2})} \kappa_{S}(\alpha e^{\epsilon (\omega_{\mathbb{R}} - \mu_{\mathbb{R}} +\xi - \varphi\sigma)})= 
$$
$$
\lim_{\epsilon \to 0^{+} } \frac{1}{n!\, vol(T)} \textnormal{EqRes}_{(\sigma,\tau)}
\left[
\frac{
\alpha(\sigma,\tau) ((N+1)\sigma)^{n} \prod\limits_{1\leq i\neq j\leq n}(\tau_{i} - \tau_{j})((N+1)\sigma + \tau_{i} -\tau_{j}) e^{\epsilon \sum\limits_{i=1}^{n} \frac{\tau_{i}}{2}} 
}{
N^{n}\sigma^{2n} \prod\limits_{1\leq i\neq j\leq n} (\sigma + \tau_{i} - \tau_{j})(N\sigma + \tau_{i} - \tau_{j}) \prod\limits_{k=1}^{n} (\sigma + \tau_{k})(N\sigma - \tau_{k})
}
\right]=
$$
$$ 
\lim_{\epsilon \to 0^{+} } \frac{1}{n!\,vol(T)} \left( \frac{N+1}{N\sigma}\right)^{n} \textnormal{EqRes}_{(\sigma,\tau)}
\left[
\frac{
\alpha(\sigma,\tau) \prod\limits_{1\leq i\neq j\leq n}(\tau_{i} - \tau_{j})((N+1)\sigma + \tau_{i} -\tau_{j}) e^{\epsilon \sum\limits_{i=1}^{n} \frac{\tau_{i}}{2}} 
}{
\prod\limits_{1\leq i\neq j\leq n} (\sigma + \tau_{i} - \tau_{j})(N\sigma + \tau_{i} - \tau_{j}) \prod\limits_{k=1}^{n} (\sigma + \tau_{k})(N\sigma - \tau_{k})
}
\right],
$$
by Theorem \ref{Thm-JK-4} for $\mathcal{M}/\!\!/\!\!/\!\!/_{(\xi,0)}U(n)$.
Remark that since $\sum_{i=1}^{n}\tau_{i}$ is generic we may use the polarizing bases $\{\sigma,\tau\}$ for calculating the equivariant Jeffrey-Kirwan residue. 
We choose scalar product on $\mathfrak{t}^{*}$ such that $\tau$ is an orthonormal bases.

Let $\mathcal{A} = \{N\sigma-\tau_{i},\,\sigma + \tau_{i},\, \sigma +\tau_{i} -\tau_{j},\, N\sigma +\tau_{i} - \tau_{j} \,|\, 1\leq i\neq j\leq n \}$. Let $V$ be a hyperplane in $\mathfrak{t}^{*} \oplus \mathfrak{s}^{*}$ spanned by some of elements of $\mathcal{A}$ such that $V$ projects bijectively to $\mathfrak{t}^{*}$, i.e. $V$ is a pole of 
$$
F(\tau,\sigma) 
=
\frac{
\prod\limits_{1\leq i\neq j\leq n}(\tau_{i} - \tau_{j})((N+1)\sigma + \tau_{i} -\tau_{j})  
}{
\prod\limits_{1\leq i\neq j\leq n} (\sigma + \tau_{i} - \tau_{j})(N\sigma + \tau_{i} - \tau_{j}) \prod\limits_{k=1}^{n} (\sigma + \tau_{k})(N\sigma - \tau_{k})
}. 
$$
Consider the system of equation $\alpha = 0$ for all $\alpha \in \mathcal{A}_{V} = \{\alpha\in\mathcal{A}  \,|\,  \alpha \in V\}$ in $\tau_{1},\ldots,\tau_{n}$ unknowns and $\sigma$ parameter. It has a unique solution of form $\tau_{i} = p_{i}\sigma$, $i=1,\ldots,n,$ and 
$$
V = \langle \tau_{1} - p_{1}\sigma,\,\ldots,\,\tau_{n} - p_{n}\sigma \rangle.
$$ Denote $p_{V} = (p_{1}\sigma,\ldots,p_{n}\sigma)$. We say that a pole $V$ does not contribute if 
$$
\textnormal{JKRes}_{v}\alpha(\tau(v,\sigma),\sigma) F(\tau(v,\sigma),\sigma)e^{\frac{1}{2}\sum\limits_{i=1}^{n} \tau_{i}(v,\sigma)}dv = 0
$$ 
for all $\alpha \in \mathbb{R}[\sigma,\tau_{1},\ldots,\tau_{n}]^{S_{n}}$ and where $v$ is the bases on $V$ induced by $\{\sigma,\tau_{1},\ldots,\tau_{n}\}$.
Recall that $V$ does not contribute if $\frac{1}{2}\sum_{i=1}^{n}\tau_{i} \notin \textnormal{pr}_{\mathfrak{t}^{*}}(Cone(\alpha  \,|\, \alpha\in \mathcal{A}_{V}))$, where $\textnormal{pr}_{\mathfrak{t}^{*}}:\mathfrak{t}^{*}\oplus\mathfrak{s}^{*} \to \mathfrak{s}^{*}$ is the projection.

\begin{lemma}\label{Lem-HS-2}
A pole $V$ does not contribute (or exists) if any of the following holds:
\begin{enumerate}

\item\label{Eq-HS-1a} there exist $i$ such that there is no $\alpha \in \mathcal{A}_{V}$ with positive $\tau_{i}$-coefficient, i.e. $\sigma + \tau_{i},\, \sigma + \tau_{i} - \tau_{j}, N\sigma+\tau_{i}-\tau_{j} \notin \mathcal{A}_{V}$ for all $j$.

\item\label{Eq-HS-1b} $p_{k}=m$ with $m\geq 0$ for some $k$, 

\item\label{Eq-HS-1c} $N\sigma-\tau_{i}\in\mathcal{A}_{V}$ for some $i$,

\item\label{Eq-HS-1d} $\sigma+\tau_{i},\  \varepsilon\sigma+\tau_{i}-\tau_{j}\in\mathcal{A}_{V}$ for some $i\neq j$ $(\varepsilon\in\{1,N\})$,

\item\label{Eq-HS-1e} $\sigma+\tau_{i}-\tau_{j},\ N\sigma+\tau_{i}-\tau_{j}\in\mathcal{A}_{V}$ for some $i\neq j$,

\item $\varepsilon\sigma+\tau_{i}-\tau_{j},\ \varepsilon'\sigma+\tau_{j}-\tau_{i}\in\mathcal{A}_{V}$ for some $i\neq j$ $(\varepsilon,\varepsilon'\in\{1,N\})$,

\item $\sigma+\tau_{i}\notin\mathcal{A}_{V}$ for all $i$.


\end{enumerate}
\end{lemma}

\begin{proof}
\begin{enumerate}

\item If $\frac{1}{2}\sum_{i=1}^{n}\tau_{i}\in \textnormal{pr}_{\mathfrak{t}^{*}}(Cone(\alpha\ |\ \alpha\in\mathcal{A}_{V}))$ then for all $i$ there must be an $\alpha\in\mathcal{A}_{V}$ with positive $\tau_{i}$ coefficient.

\item Let $M$ be an index such that $p_{M}$ is maximal, hence $p_{M}\geq 0$. Then $\varepsilon\sigma+\tau_{M}-\tau_{l}\notin \mathcal{A}_{V}$ ($\varepsilon\in\{1,N\}$) for all $l$ since it would imply that $p_{l}=p_{M}+\varepsilon>p_{M}$. Moreover, $\sigma+\tau_{M}\notin \mathcal{A}_{V}$ since it would imply that $p_{M}=-1<0$. Therefore, $\mathcal{A}_{V}$ does not contain element with positive $\tau_{M}$ coefficient.

\item It is equivalent that $p_{i}=N>0$.

\item It follows that $p_{i}=-1$ and $p_{j}=\varepsilon -1\geq 0$.

\item $(N-1)\sigma=(N\sigma+\tau_{i}-\tau_{j})-(\sigma+\tau_{i}-\tau_{j})\in V$, hence if $n\geq 2$ then $\sigma\in V$ which implies that $V$ is not a pole.

\item Similar to the previous one.

\item By (\ref{Eq-HS-1c}) we can suppose that $\mathcal{A}_{V}$ contains only elements of form $\varepsilon\sigma+\tau_{i}-\tau_{j}$ $(\varepsilon\in\{1,N\})$, but these are perpendicular to $\sum_{i=1}^{n}\tau_{i}$ with respect to the usual Euclidean scalar product on $\mathfrak{t}^{*}\oplus \mathfrak{s}^{*}$, hence $\sum_{i=1}^{n}\tau_{i}\notin \textnormal{pr}_{\mathfrak{t}^{*}}(Cone(\alpha\ |\ \alpha\in\mathcal{A}_{V}))$.

\end{enumerate}
\end{proof}

\begin{lemma}\label{Lem-HS-3}
Let $V$ be a pole. We cannot decompose $\{0,1,\ldots,n\}=N'\uplus N''$ as disjoint union of two non-empty subsets with properties:
	\begin{itemize}
\item if $N\sigma-\tau_{i}\in\mathcal{A}_{V}$ or $\sigma+\tau_{i}\in\mathcal{A}_{V}$ then $0,i\in N'$ or $0,i\in N''$.

\item if $\varepsilon\sigma+\tau_{i}-\tau_{j}\in\mathcal{A}_{V}$ $(\varepsilon\in\{1,N\})$ then $i,j\in N'$ or $i,j\in N''$.
	\end{itemize}
\end{lemma}
\begin{proof}
Suppose that $0\in N'$. The decomposition $\{0,1,\ldots,n\}=N'\uplus N''$ induces a decomposition $\mathcal{A}_{V}=\mathcal{A}'_{V}\uplus\mathcal{A}''_{V}$ according to an element $\alpha\in\mathcal{A}_{V}$ involving terms $\tau_{i}$ we have $i\in N'$ or $i\in N''$. Remark that $\mathcal{A}''_{V}$ may contain only elements of form $\varepsilon\sigma+\tau_{i}-\tau_{j}$. Moreover, elements of $\textnormal{pr}_{\mathfrak{t}^{*}}\mathcal{A}''_{V}$ cannot span $\mathfrak{t}^{*}$. Therefore, $\mathcal{A}'_{V}\neq\emptyset$.

We will show that $\mathcal{A}''_{V}=\emptyset$ which implies that $N''=\emptyset$. The elements $\textnormal{pr}_{\mathfrak{t}^{*}}\mathcal{A}'_{V}$ span a subpace $V'$ of $\langle \tau_{i}\ |\ i\in N'\setminus\{0\}\rangle$ with dimension at most $|N'|-1$. The elements $\textnormal{pr}_{\mathfrak{t}^{*}}\mathcal{A}''_{V}$ span a subspace $V''$ of $\langle\tau_{i}\ |\ i\in N''\rangle$. If $N''\neq\emptyset$ then the dimension of $V''$ is at most $|N''|-1$ because all elements of $\textnormal{pr}_{\mathfrak{t}^{*}}\mathcal{A}''_{V}$ are perpendicular to $\sum_{i\in N''}\tau_{i}$. Finally, if $N''\neq\emptyset$  then $V'\cap V''=\{0\}$ and $\dim V'+\dim V''\leq|N'|-1+|N''|-1=n-1$ which leads to contradiction since $V$ is a pole.
\end{proof}

To each pole $V$ 
we associate an oriented graph $\Gamma_{V}$ with vertices $\{0,1,\ldots,n\}$. The vertices will lie on the lattice $\mathbb{Z} \times N\mathbb{Z}$. We allow that two vertex lie on same lattice point. The vertex $0$ is always on the lattice point  $(0,0)$. The edges of $\Gamma_{V}$ correspond to elements of $\mathcal{A}_{V}$ as follows.
\begin{center}
\begin{tabular}{c|c}
$\alpha\in\mathcal{A}_{V}$ & edges of $\Gamma_{V}$ 
\\\hline
$N\sigma+\tau_{i}-\tau_{j}$ & vertical edge from $j$ to $i$ 
\\
$N\sigma-\tau_{i}$ & vertical edge from $i$ to $0$ 
\\
$\sigma+\tau_{i}-\tau_{j}$ & horizontal edge from $j$ to $i$ 
\\
$\sigma+\tau_{i}$ & horizontal edge from $0$ to $i$
\end{tabular}
\end{center}
The vertical and horizontal edges go as $(a,bN)\to(a,(b+1)N)$
and $(a,b)\to(a+1,b)$, respectively. We construct $\Gamma_{V}$ inductively as follows. Choose a bases $\{\alpha_{1},\ldots,\alpha_{n}\}$ of $\mathcal{A}_{V}$. We may suppose that $\alpha_{1} = \sigma + \tau_{i}$ or $\alpha_{1} = N\sigma - \tau_{i}$. We place first the vertex $0$ to $(0,0)$. Suppose that we have already placed some vertices on the lattice. In a turn, we place all other vertices which can be connected to already placed vertices following the rules in the above table for bases vectors. 
Finally, we can place all vertices by Lemma \ref{Lem-HS-3} and the edges corresponding to bases vectors form a tree (i.e. no loop). 
The shortest oriented edge path $(e_{1},\ldots,e_{m})$ on this tree from $0$ to $i$ gives relation
$$
\tau_{i} - p_{i}\sigma = \sum_{j=1}^{m} a_{j}\alpha_{k_{j}},
$$
where $\alpha_{k_{j}}$ is the bases vector corresponding to the edge $e_{j}$ and $a_{j}=\pm 1$ depending on if the orientation of the edge $e_{j}$ agrees with our walking direction on the path or no. Since $m\leq n$ we get
$$
p_{i} = a + bN
$$
with $-n \leq a + b \leq n$. Thus $p_{i}$ determines the coordinates of the vertex $i$ on the lattice: the vertex $i$ has coordinates $(-a, -bN)$ where $a = p_{i}\mod N. $ Remark that the vertex $i$ is on the line $x+y +p_{i} = 0$. Moreover, the coordinates of the vertices does not depend on the choice of the bases $\{\alpha_{1},\ldots,\alpha_{n}\}$, implying that for any $\alpha\in \mathcal{A}_{V}$ we can draw an edge according to the above table. Remark that, if vertices $i$ and $j$ lie on the same lattice point and there is an edge $i\to k$ (or $k \to i$) then there exists also an edge from $j\to k$ (or $k\to j$ respectively). Hence drawing edges will not lead to confusion if vertices lie on the same lattice point.

\begin{remark*}
The graph $\Gamma_{V}$ is full in the sense that if there is a vertex $i$ at coordinates $(a,b)$ and another vertex $j$ at coordinates $(a+1,b)$ ($(a,b+N)$ respectively) then there must be an edge $i\to j$ if $j\neq0$ ($i\neq 0$, respectively). 
\end{remark*}

We rephrase the Lemma \ref{Lem-HS-2} for graphs.

\begin{corollary}\label{Cor-HS-1}
The pole $V$ does not contribute (or does not exist)  if any of the following holds.
\begin{enumerate}


\item There are multiple edges between two vertices.

\item There is an edge with target vertex $0$.

\item\label{Eq-HS-12c} There is a vertex $i\neq0$ which has no tail. A tail of $i$  is an edge $k\to i$.

\item\label{Eq-HS-11d} There is a vertex $i\neq 0$ in the region $\{x\leq 0\}$, i.e $i$ has coordinates $(a,b)$ with $a\leq0$.

\item\label{Eq-HS-11e} There is a vertex $i$ in the region $\{y<0\}$, i.e. $i$ has coordinates $(a,b)$ with $b<0$.

\end{enumerate}
\end{corollary}

\begin{proof}
\begin{enumerate}
\item By Lemma \ref{Lem-HS-2} (\ref{Eq-HS-1e}).

\item By Lemma \ref{Lem-HS-2} (\ref{Eq-HS-1c})

\item By Lemma \ref{Lem-HS-2} (\ref{Eq-HS-1a})

\item Suppose that there is a vertex $i\neq0$ in the region $\{x \leq 0\}$. Then let  $k\neq0$ be a vertex in $\{x\leq0\}$ lying on the line $x+y+m=0$ with $m$ maximal. Hence $k$ has no tail. Indeed, if there is an edge $l \to k$ then $l$ is in $\{x\leq0\}$. Moreover, $l\neq0$ since there is only horizontal arrow from $0$. But $l$ is on the line $x+y+m+1 = 0$ if $l\to k$ is horizontal and it is on the line $x+y+m+N=0$ if $l\to k$ is vertical, leading to contradiction with maximality of $m$.

\item Let $k$ be a vertex in $\{y<0\}$ lying on the line $x+y+m = 0$ with $m$ maximal. Then $k$ has no tail. Indeed, suppose that there is a edge $l\to k$. Then $l$ is in the region  $\{y<0\}$ and it is on the line $x+y+m+1 = 0$ if $l\to k$ is horizontal and it is on the line $x+y+m+N =0$ if $l\to k$ is vertical contradicting to the maximality of $m$.

\end{enumerate}
\end{proof}

Set $F_{\Gamma_{V}}$ to $F$. If $\Gamma_{V}$ contains a subgraphs
\begin{equation}\label{Eq-HS-2}
\xymatrix{ & k \\ 0\neq i\ar[r] & j\ar[u] }
\end{equation} 
then we have partial fraction decomposition
\begin{equation}\label{Eq-HS-3}
F_{\Gamma_{V}}
=
G\frac{((N+1)\sigma+\tau_{k}-\tau_{i})}{(\sigma+\tau_{j}-\tau_{i})(N\sigma+\tau_{k}-\tau_{j})}
=
G\frac{1}{N\sigma+\tau_{k}-\tau_{j}}
+
G\frac{1}{\sigma+\tau_{j}-\tau_{i}}.
\end{equation}
Let $\widehat{\Gamma}'_{V}$ (respectively $\widehat{\Gamma}''_{V}$ ) be the extended graph obtained from $\Gamma_{V}$ by removing the edge $i\to j$ (respectively $j \to k$) and adding the dashed diagonal edge $i\dashrightarrow k$. To each dashed diagonal edge $i\dashrightarrow k$ we associate the term $(N+1)\sigma + \tau_{k} - \tau_{i}$ in the nominator of $F_{\Gamma_{V}}$. 
Denote $\Gamma'_{V}$ and $\Gamma''_{V}$ the graphs obtained from $\widehat{\Gamma}'_{V}$ and $\widehat{\Gamma}''_{V}$, respectively by forgetting the dashed diagonal edges. Moreover, to extended graphs $\widehat{\Gamma}'_{V}$ and $\widehat{\Gamma}''_{V}$  we associate  fractions 
$$
F_{\widehat{\Gamma}'_{V}} = G\frac{1}{N\sigma+\tau_{k}-\tau_{j}} \quad\textnormal{and}\quad
F_{\widehat{\Gamma}''_{V}} =G\frac{1}{\sigma+\tau_{j}-\tau_{i}}.
$$ 
Then $F_{\widehat{\Gamma}'_{V}}$ (and $F_{\widehat{\Gamma}''_{V}}$) can be obtained from $F_{\Gamma_{V}}$ by removing from the nominator and denominator the terms corresponding to the added dashed edges and removed normal edges, respectively. In picture the relation (\ref{Eq-HS-3}) looks like
\begin{equation}\label{Eq-HS-4}
\xymatrix{ 
& k \ar@{}[dr] |{\longmapsto} &  & k \ar@{}[dr] |{+} & & k  
\\ 
0\neq i\ar[r] & j\ar[u] & 
0\neq i\ar@{-->}[ru] & j\ar[u] &  
0\neq i\ar@{-->}[ru] \ar[r] & j }. 
\end{equation}
We have a similar decomposition corresponding to
\begin{equation}\label{Eq-HS-5}
\xymatrix{ 
l \ar[r] & k \ar@{}[dr] |{\longmapsto} & l \ar[r] & k \ar@{}[dr] |{+} & l & k  
\\ 
i\ar[u] &  & 
i \ar@{-->}[ru] &  &  
i \ar[u] \ar@{-->}[ru]  &  }, 
\end{equation}
namely,
$$
F_{\Gamma_{V}}
=
G\frac{((N+1)\sigma+\tau_{k}-\tau_{i})}{(\sigma+\tau_{k}-\tau_{l})(N\sigma+\tau_{l}-\tau_{i})}
=
G\frac{1}{\sigma+\tau_{k}-\tau_{l}}
+
G\frac{1}{N\sigma+\tau_{l}-\tau_{i}}.
$$
We can continue these type of decomposition for $F_{\widehat{\Gamma}'_{V}}$ and $F_{\widehat{\Gamma}''_{V}}$ to get a partial fraction decomposition of form
$$
F_{\Gamma_{V}} = \sum_{\widehat{\Upsilon}} F_{\widehat{\Upsilon}}
$$
where $\widehat{\Upsilon}$ is an extended graph  such that $\Upsilon$ is a subgraph of $\Gamma_{V}$ and obtained from $\widehat{\Upsilon}$ by forgetting all dashed diagonal edges. The fraction $F_{\widehat{\Upsilon}}$ is encoded by $\widehat{\Upsilon}$ as follows:  we remove all linear terms from the nominator and denominator of $F_{\Gamma_{V}}$ corresponding to added dashed diagonal edges and removed normal edges, respectively. Remark that if $\widehat{\Upsilon}$ contains the subgraph
$$
\xymatrix{ & k \\ 0\neq i\ar[r] \ar@{-->}[ru] & j\ar[u] }
$$
then we cannot decompose $F_{\widehat{\Upsilon}}$  to $G'\frac{1}{N\sigma+\tau_{k}-\tau_{j}} + 
G'\frac{1}{\sigma+\tau_{j}-\tau_{i}}$ (according to the picture (\ref{Eq-HS-4}))
 since $(N+1)\sigma + \tau_{k} -\tau_{i}$ is missing from the nominator of $F_{\widehat{\Upsilon}}$. An analog statement holds for the decomposition (\ref{Eq-HS-5}).

Similarly as Corollary \ref{Cor-HS-1} (\ref{Eq-HS-12c}) we have

\begin{lemma}\label{Lem-HS-4}
For all $\alpha \in \mathbb{R}[\sigma,\tau_{1},\ldots,\tau_{n}]^{S_{n}}$ we have
$$
\textnormal{JKRes}_{v} \alpha(\tau(v,\sigma),\sigma) F_{\widehat{\Upsilon}} (\tau(v,\sigma),\sigma) e^{\frac{1}{2}\sum\limits_{i=1}^{n}\tau_{i}(v,\sigma)}dv = 0 
$$  
if in $\Upsilon$ there is a vertex $i\neq 0$ without tail.
\end{lemma}

\begin{lemma}\label{Lem-HS-5}
\begin{enumerate}

\item Suppose that $\Gamma_{V}$ contains the subgraph (\ref{Eq-HS-2}). If there is no vertex $l$ such that $\Gamma_{V}$ contains the subgraph
$
\xymatrix{
l \ar[r] & k \\ 0\neq i \ar[u] \ar[r] & j \ar[u] }
$
then $V$ does not contribute.

\item Suppose that $\Gamma_{V}$ contains the subgraph 
$
\xymatrix{l \ar[r] & k \\ i\ar[u]& }.
$
 If there is no vertex $j$ such that $\Gamma_{V}$ contains the subgraph
$
\xymatrix{
l \ar[r] & k \\ 0\neq i \ar[u] \ar[r] & j \ar[u] }
$
then $V$ does not contribute.

\item If there are multiple vertices on the same lattice point.

\end{enumerate}
\end{lemma}

\begin{proof}
\begin{enumerate}
\item By decomposition (\ref{Eq-HS-4}) we have $F_{\Gamma_{V}} = F_{\widehat{\Gamma}'_{V}} + F_{\widehat{\Gamma}''_{V}}$. If there is no horizontal edge $l\to k$ then $k$ has no tail in $\Gamma''_{V}$, hence
\begin{multline}\label{Eq-HS-6}
\textnormal{JKRes}_{v} \alpha(\tau(v,\sigma),\sigma) F_{\Gamma_{V}} (\tau(v,\sigma),\sigma) e^{\frac{1}{2}\sum\limits_{i=1}^{n}\tau_{i}(v,\sigma)}dv 
\\=
\textnormal{JKRes}_{v} \alpha(\tau(v,\sigma),\sigma) F_{\widehat{\Gamma}'_{V}} (\tau(v,\sigma),\sigma) e^{\frac{1}{2}\sum\limits_{i=1}^{n}\tau_{i}(v,\sigma)}dv   
\end{multline}
Since $j$ has no horizontal tail in $\Gamma'_{V}$, the right hand side of (\ref{Eq-HS-6}) vanishes unless there is a vertical tail $j_{1} \to j$. If $j_{1}$ has no tail then (\ref{Eq-HS-6}) still vanishes. If it has a horizontal tail by 
a similar decomposition and argument we can get rid of it. So is it safe to suppose that it may have only vertical tail. But $j$  may have only finitely many vertices below it, hence the last one has no vertical tail and (\ref{Eq-HS-6}) anyhow vanishes.

If $l \to k$ horizontal edge exists then the existences of $i\to l$ follows from
$$
N\sigma + \tau_{l} - \tau_{i} = (N\sigma + \tau_{k} - \tau_{j}) - (\sigma +\tau_{k} -\tau_{l}) +(\sigma + \tau_{j}-\tau_{i}) \in V.
$$
\item The proof is similar to the previous case.

\item Suppose that there are two vertices $i$ and $j$ on the same lattice point $(a,b)$. We may suppose that there are no more double vertices on lattice points $(a,b) + \mathbb{Z}_{<0}\times \mathbb{Z}_{<0}$. 

If there are only vertical or horizontal tail to $i$ (and therefore to $j$, too) then by decomposition
	\begin{equation}\label{Eq-HS-10}
F 
= 
G\frac{ \tau_{i} - \tau_{j} }{ (\varepsilon\sigma + \tau_{i} - \tau_{k})(\varepsilon\sigma + \tau_{j} - \tau_{k}) } 
= 
-G\frac{1}{ \varepsilon\sigma + \tau_{i} - \tau_{k} } + G\frac{1}{ \varepsilon\sigma + \tau_{j} - \tau_{k} }, 
\qquad
\varepsilon\in \{1,N\}
	\end{equation}
or 
$$
F 
= 
G\frac{ \tau_{i} - \tau_{j} }{ (\sigma + \tau_{i})(\sigma+\tau_{j}) }
=
-G\frac{1}{\sigma + \tau_{i}} + G\frac{1}{\sigma+\tau_{j}}.
$$
we get two graphs in which $i$ or $j$ has no tail. By Lemma \ref{Lem-HS-4} the pole $V$ does not contribute.

Suppose that $i$ (and therefore $j$) has both vertical and horizontal tails. By Corollary \ref{Cor-HS-1} (\ref{Eq-HS-11d}) and (\ref{Eq-HS-11e}) we can suppose that $a\neq1$ and $b\neq0$. By decomposition (\ref{Eq-HS-10}) we get

\begin{align*}
F 
={}&
G\frac{(\tau_{i} - \tau_{j})(\tau_{j} - \tau_{i})}{ (\sigma + \tau_{i}-\tau_{l}) (\sigma + \tau_{j}-\tau_{l})  (N \sigma + \tau_{i}-\tau_{k})  (N\sigma + \tau_{j}-\tau_{k})}
\\
={}&
-G\frac{1}{ (\sigma + \tau_{i}-\tau_{l})  (N\sigma + \tau_{i}-\tau_{k})}
+
G\frac{1}{ (\sigma + \tau_{i}-\tau_{l})  (N\sigma + \tau_{j}-\tau_{k})}
\\
&
+
G\frac{1}{ (\sigma + \tau_{j}-\tau_{l})  (N\sigma + \tau_{i}-\tau_{k})}
-
G\frac{1}{ (\sigma + \tau_{j}-\tau_{l})  (N\sigma + \tau_{j}-\tau_{k})}.
\end{align*}
In picture
$$
\xymatrix{
l \ar@<.5ex>[r]^{i}\ar@<-.5ex>[r]_{j}   &     i,j \ar@{}[dr] |{\mapsto}     &
l  \ar[r]^{i}   &   i,j \ar@{}[dr] |{+}     &     
l \ar[r]^{i}    &       i,j  \ar@{}[dr] |{+}    &     
l \ar[r]_{j}    &       i,j   \ar@{}[dr] |{+}   &     
l \ar[r]_{j}    &       i,j  
\\
      &      k\ar@<.5ex>[u]^{i}\ar@<-.5ex>[u]_{j}    &
      &      k \ar[u]^{i}    & 
      &      k \ar[u]_{j}    &
      &      k \ar[u]^{i}   &
      &      k \ar[u]_{j}
}
$$
(the labels on the edges indicate the target of the edge). On the first and last graphs on the right hand side $i$ and respectively $j$ have no tails, so they do not contribute. By symmetry consider only the second graph on the right hand side. Suppose that $l$ has a vertical tail $h \to l$. By decomposition (\ref{Eq-HS-5}) we get
$$
\xymatrix{
l \ar[r]^{i}    &     i,j \ar@{}[dr] |{\mapsto}     &        l\ar[r]^{i}     &     i,j  \ar@{}[dr] |{+}    &        l     &     i,j    
\\
h \ar[u]    &      k   \ar[u]_{j}     &        h \ar@{-->}[ru]^{i}    &      k\ar[u]_{j}      &        h\ar[u]\ar@{-->}[ru]^{i}     &      k\ar[u]_{j}    
}
$$
and on the last graph $i$ does not have a tail (a non-dashed edge). Iterating the same argument we may assume that all vertices in the row of $l$ have only horizontal tail. Since $b\neq0$, the last vertex in that row will not have tail. Therefore $V$ does not contribute. 
\end{enumerate}
\end{proof}

Lets summarize what we showed so far. A pole $V$ does not contribute unless
\begin{itemize}

\item there are no double vertices, i.e. two vertex on the same lattice point,

\item there are no double edges, i.e. two edges between two vertex,

\item the vertices $\{1,\ldots,n\}$ are in the quadrant $\mathbb{Z}_{>0} \times \mathbb{Z}_{\geq0}$,

\item every vertex $\{1,\ldots,n\}$ has a tail,

\item all lattice points in $\{(x,y)\in \mathbb{Z}_{>0} \times N\mathbb{Z}_{\geq0} \,|\, x<a,\ y<b\}$ carry a vertex if $(a,b)$ carries a vertex (follows from Lemma \ref{Lem-HS-5}). 
\end{itemize}
For such a pole if we draw a box around each non-zero vertex we get a Young diagram, labeled by $\{1,\ldots,n\}$. Conversely, for any Young diagram such that each box contains in the center a unique lattice point in $\mathbb{Z}_{>0}\times N\mathbb{Z}_{\geq0}$ and it is labeled by $1,\ldots,n$, we can associate a pole $V=\langle \tau_{1}- p_{1}\sigma,\ldots,\tau_{n}-p_{n}\sigma \rangle$ with $p_{i} = -(a_{i} + Nb_{i})$ if the box labeled by $i$ contains the  lattice point $(a_{i},Nb_{i})$.

Let $V$ be a pole obtained from a labeled Young diagram $D$. We can decompose $F= F_{\Gamma_{V}}$ to partial fractions with the following algorithm. Take the top- and leftmost subgraph
$$
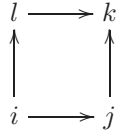
\xymatrix{ l \ar[r] & k \\ i \ar[u]\ar[r] & j\ar[u]}
$$
and apply the decomposition corresponding to the picture (\ref{Eq-HS-5}):
$F_{\Gamma_{V}} = F_{\widehat{\Gamma}'_{V}} + F_{\widehat{\Gamma}''_{V}}$, where $\widehat{\Gamma}'_{V}$ (respectively  $\widehat{\Gamma}''_{V}$)  is the graph  obtained from $\Gamma_{V}$ by removing the edge $l\to k$ (respectively $i\to l$) and adding the dashed diagonal edge $i\dashrightarrow k$. Remark that in  $\widehat{\Gamma}''_{V}$ the vertex $l$ has no tail. Hence (\ref{Eq-HS-6}) holds. We continue the algorithm as above with  $\widehat{\Gamma}'_{V}$ in place of $\Gamma_{V}$. Finally, we get that
\begin{multline}\label{Eq-HS-7}
\textnormal{JKRes}_{v} \alpha(\tau(v,\sigma),\sigma) F_{\Gamma_{V}} (\tau(v,\sigma),\sigma) e^{\frac{1}{2}\sum\limits_{i=1}^{n}\tau_{i}(v,\sigma)}dv 
\\=
\textnormal{JKRes}_{v} \alpha(\tau(v,\sigma),\sigma) F_{\widehat{\Upsilon}} (\tau(v,\sigma),\sigma) e^{\frac{1}{2}\sum\limits_{i=1}^{n}\tau_{i}(v,\sigma)}dv,   
\end{multline}
where $\widehat{\Upsilon}$ is the extended graph obtained from $\Gamma_{V}$ by removing all horizontal edges except the ones in the bottom row and by adding dashed diagonal edges under each horizontal edge removed. The graph $\Upsilon$ has exactly $n$ edges (non-dashed edges), thus in the denominator $F_{\widehat{\Upsilon}}$ has exactly $n$ linear term which are in $V$. They form a bases of $V$, since each $\tau_{i}$ appears with positive coefficient in a unique linear term. It also follows that $V$ contributes and (\ref{Eq-HS-7}) equals
\begin{equation}\label{Eq-HS-8}
\frac{1}{vol(T)}\alpha(p_{V},\sigma)b_{\widehat{\Upsilon}}(p_{V},\sigma) e^{-\frac{1}{2}\sum\limits_{i=1}^{n}p_{i}\sigma},
\end{equation}
where $b_{\widehat{\Upsilon}}$ is obtained from $F_{\widehat{\Upsilon}}$ by removing all linear terms in $V$ from the denominator. Remark that by permuting labels of $D$ we act freely on the set of poles, but $b_{\widehat{\Upsilon}}(p_{V},\sigma)$ remains invariant, therefore we denote it by
\begin{equation}\label{Eq-HS-9}
b_{\lambda}(\sigma),
\end{equation}
where $\lambda$ is the partition corresponding to the Young diagram $D$. Moreover,   for all poles $V$ obtained from the same Young diagram $D$ but with different labeling (\ref{Eq-HS-8}) yields the same result, since $\alpha$ is symmetric in $\tau_{1},\ldots,\tau_{n}$. Finally,
$$
\oint_{\textnormal{Hilb}^{n}(\mathbb{C}^{2})} \kappa_{S}(\alpha) 
= 
\frac{1}{vol(T)} \left( \frac{N+1}{N\sigma}\right)^{n} \sum_{\lambda \vdash n} \alpha(p_{\lambda},\sigma) b_{\lambda}(\sigma).
$$
\end{proof}



\begin{thebibliography}{123}
\bibitem[A]{A} M.F. Atiyah, {\it Convexity and commuting Hamiltonians}, Bull. London Math. Soc. 14 (1982), 1-15.
\bibitem[AB]{AB} M.F. Atiyah, R. Bott, {\it The moment map and equivariant cohomology}, Topology, 23(1):1Đ28, 1984.
\bibitem[BeV]{BeVe} N. Berline, M. Vergne, {\it Classes caract\'eristiques \'equivariantes. Formule de localisation en cohomologie \'equivariante,} C. R. Acad. Sci. Paris S\'er. I Math., 295(9):539-541, 1982.
\bibitem[BT]{BT} R. Bott, L. W. Tu, {\it Equivariant characteristic classes in the Cartan model}, arXiv preprint math/0102001 (2001).
\bibitem[BiV]{BV}M. Brion, M. Vergne, {\it Arrangements of hyperplane. I. Rational functions and the Jeffrey-Kirwan residue}, Annales scientifiques de lŐ\'E.N.S. 4e s\'erie, tome 32, $\textnormal{n}^{\circ}$ 5 (1999), p. 715 - 741
\bibitem[DH]{DH} J.J. Duistermaat and G.J. Heckman, {\it On the variation in the cohomology of the symplectic form of the reduced phase space}, Invent. Math., 69(2):259-268, 1982.
\bibitem[ES]{ES} G. Ellingsrud and S.A. Str\o mme, {\it On the homology of the Hilbert scheme of points in the plane}, Invent. Math. 87 (1987), 343-352.
\bibitem[GKG]{GKG} V. Guillemin, Y. Karshon, V.L. Ginzburg, {\it Moment maps, cobordisms, and Hamiltonian group actions}, Vol. 98. Amer Mathematical Society, 2002.
\bibitem[GS1]{GS} V. Guillemin, S. Sternberg, {\it Supersymmetry and equivariant de Rham theory.} Vol. 2. Springer, 1999.
\bibitem[GS2]{GS2} V. Guillemin, S. Sternberg, {\it Convexity properties of the moment mapping I.} Invent. Math. 67 (1982) 491-513.
\bibitem[HP1]{HP1} M. Harada, N. Proudfoot, {\it Properties of the residual circle action on a hypertoric variety.} Pacific J. Math 214.2 (2004): 263-284.
\bibitem[HP2]{HP2} T. Hausel, N. Proudfoot, {\it Abelianization for hyperk\"ahler quotients}, Topology 44.1 (2005): 231-248
\bibitem[JK1]{JK1} L.C. Jeffrey, F.C. Kirwan, {\it Localization for nonabelian group actions}, Topology 34 (1995), no. 2, 291 - 327.
\bibitem[JK2]{JK2} L.C. Jeffrey, Frances C. Kirwan, {\it Localization and the quatization conjecture}, Topology, Volume 36, No. 3, pp. 647 - 693, 1997.
\bibitem[JKo]{JKo} L. Jeffrey, M. Kogan, {\it Localization theorems by symplectic cuts.} in "The breadth of symplectic and Poisson geometry" (2005): 303-326. 
\bibitem[Ki]{Ki} F.C. Kirwan, {\it Cohomology of quotients in symplectic and algebraic geometry}, volume 31 of Mathematical Notes. Princeton University Press, Princeton, NJ, 1984.
\bibitem[Ler]{Ler} E. Lerman, {\it Symplectic cuts.} Math. Res. Lett 2.3 (1995): 247-258.
\bibitem[LMTW]{LMTW}. Lerman, E. Meinrenken, S. Tolman, C. Woodward, {\it Nonabelian convexity by symplectic cuts}, Topology 37 (1998), no. 2, 245-259
\bibitem[LS]{LS} M. Lehn, C. Sorger. {\it Symmetric groups and the cup product on the cohomology of Hilbert schemes}, Duke Mathematical Journal 110.2 (2001): 345-357.
\bibitem[Ma1]{Martens} J. Martens, {\it Equivariant volumes of non-compact quotients and instanton counting}, Communications in Mathematical Physics 281.3 (2008): 827-857.
\bibitem[Ma2]{Martin} S. Martin, {\it Symplectic quotients by a nonabelian group and by its maximal torus.} arXiv preprint math/0001002 (2000).
\bibitem[MS]{MS} D. McDuff, D. Salamon, {\it Introduction to symplectic topology.} Oxford University Press, USA, 1999.
\bibitem[Me]{Me} E. Meinrenken, {\it Symplectic Surgery and the Spin$^{c}$-Dirac Operator},  Advances in mathematics 134.2 (1998): 240-277
\bibitem[MNS]{MNS} G. Moore, N. Nekrasov, S. Shatashvili. {\it Integrating over Higgs branches}, Communications in Mathematical Physics 209.1 (2000): 97-121.
\bibitem[Na]{Na} H. Nakajima, {\it Lectures on Hilbert schemes of points on surfaces}. Vol. 18. Amer Mathematical Society, 1999.
\bibitem[PW]{PW} E. Prato, S. Wu. {\it Duistermaat-Heckman measures in a non-compact setting}, Compositio Math 94.2 (1994): 113-128.
\bibitem[TW]{TW} S. Tolman, J. Weitsman, {\it The cohomology rings of symplectic quotients.} Communications in Analysis and Geometry 11.4 (2003): 751-774.
\bibitem[Va]{Va} E. Vasserot, {\it Sur l'anneau de cohomologie du sch\'ema de Hilbert de $\mathbb{C}^{2}$}, Comptes rendus de l'Acad\'emie des sciences. S\'erie 1, Math\'ematique 332.1 (2001): 7-12.

\end{thebibliography}
\end{document}